\documentclass{amsart}

\usepackage{amssymb}
\usepackage[mathscr]{eucal}
\usepackage{eufrak,verbatim}
\usepackage{amsrefs,amsthm,amsmath,amsfonts,amssymb,graphicx,latexsym,color}
\usepackage{hyperref}
\usepackage{setspace}
\usepackage{url,color}

\numberwithin{equation}{section}

\theoremstyle{theorem}
\newtheorem{theorem}{Theorem}

\newtheorem*{EulerInit}{Initial Data for the rE System}
\newtheorem{proposition}{Proposition}[section]
\newtheorem{lemma}[proposition]{Lemma}
\newtheorem{corollary}[proposition]{Corollary}
\newtheorem{conjecture}{Conjecture}
\theoremstyle{definition}

\newtheorem{remark}{Remark}[section]

\theoremstyle{remark}

\newcommand{\barx}{\bar{x}}
\newcommand{\barp}{\bar{P}}
\newcommand{\barq}{\bar{Q}}
\newcommand{\phat}{\hat{P}}
\newcommand{\qhat}{\hat{Q}}
\newcommand{\eqdef}{\overset{\mbox{\tiny{def}}}{=}}
\newcommand{\pressure}{p}
\newcommand{\Ent}{\eta} 
\newcommand{\Temp}{\theta} 

\newcommand{\effRE}{F_{R;\varepsilon}} 

\begin{document}
\title[Hilbert Expansion from relativistic Boltzmann to Euler]{Hilbert Expansion from the Boltzmann equation to relativistic Fluids}

\author{Jared Speck}
\thanks{J.S. was supported by the Commission of the European Communities, \\ 
\indent ERC Grant Agreement No 208007.}
\address{(JS) University of Cambridge, Department of Pure Mathematics \& Mathematical Statistics,
Wilberforce Road, Cambridge, CB3 0WB, United Kingdom}
\email{jspeck@math.princeton.edu}

\author{Robert M. Strain}
\thanks{R.M.S. was supported in part by the NSF grant DMS-0901463.}
\address{(RMS) University of Pennsylvania, Department of Mathematics, David Rittenhouse
Lab, 209 South 33rd Street, Philadelphia, PA 19104-6395, USA} 
\email{strain at math.upenn.edu}
\urladdr{http://www.math.upenn.edu/~strain/}

\begin{abstract}
We study the local-in-time hydrodynamic limit of the relativistic Boltzmann equation using a Hilbert expansion. 
More specifically, we prove the existence of local solutions to the relativistic Boltzmann equation that are nearby the local relativistic Maxwellian constructed from a class of solutions to the relativistic Euler equations that includes a large 
subclass of near-constant, non-vacuum fluid states. In particular, for small Knudsen number, these solutions to the relativistic Boltzmann equation have dynamics that are effectively captured by corresponding solutions to the relativistic Euler equations. 
\end{abstract}


\maketitle

\setcounter{tocdepth}{2}
\tableofcontents

\thispagestyle{empty}

\section{Introduction and main results} \label{S:Intro}

The special relativistic Boltzmann (rB from now on) equation provides a statistical description of a gas of relativistic particles that are interacting through binary collisions in Minkowski space, which we denote by $M.$ The dynamic variable is the one-particle empirical measure  $F^\varepsilon \geq 0,$ which represents the average number of particles of four-momentum $P$ at each space-time point $x \in M.$ The four-momentum $P$ of a particle of rest mass $m_0$ is future-directed\footnote{In the inertial coordinate system we use throughout this article, future-directed vectors $P$ satisfy $P^0 > 0.$} and satisfies the normalization condition $P_{\kappa}P^{\kappa} = - m_0^2 c^2,$ where the constant $c$ denotes the speed of light. Consequently,
we may view $F$ as a function of time $t \in \mathbb{R},$ space $\barx \in \mathbb{R}^3,$ and 3-momentum $\barp \eqdef (P^1,P^2,P^3) \in \mathbb{R}^3,$ with\footnote{The formula for $P^0$ holds only in a coordinate system in which the spacetime metric $g$ has the components $g_{\mu \nu} = \mbox{diag}(-1,1,1,1).$} 
$$
P^0 = \sqrt{m_0^2c^2 + \sum_{a=1}^3 (P^a)^2}.
$$ 
A more geometric point of view is offered in Section \ref{SS:massshell}, where it is explained how to view $F$ as function on the \emph{mass shell} 
$$
\mathfrak{M} \eqdef \lbrace (x,P) \in M \times T_x M \ | \ P_{\kappa} P^{\kappa}= - m_0^2 c^2, \ P \ \mbox{is future-directed} \rbrace,
$$ 
which is a submanifold of $TM,$ the tangent bundle of $M,$ and which is diffeomorphic to $\mathbb{R}^4 \times \mathbb{R}^3.$

The relativistic Boltzmann equation (rB from now on) in the unknown $F^\varepsilon$ is
\begin{align} \label{E:rBintro}
P^\kappa \partial_\kappa F^\varepsilon
=
\frac{1}
{\varepsilon}
\mathcal{C}(F^\varepsilon, F^\varepsilon),
\end{align}
where $\mathcal{C}(\cdot,\cdot)$ is the \emph{collision operator} (defined in \eqref{collisionCMintro}), and the dimensionless parameter $\varepsilon$ is the \emph{Knudsen number}. It is the ratio of the particle mean free path to a characteristic  (physical) length scale.  Intuitively, when $\varepsilon$ is small, the continuum approximation of fluid mechanics is expected to be valid. In this setting, we anticipate that the system of particles can be faithfully modeled through the use of macroscopic quantities, such as \emph{pressure}, \emph{proper energy density}, etc., whose evolution is prescribed by the equations of relativistic fluid mechanics, that is, the relativistic Euler equations.

As a first rigorous step in this direction, we show that any sufficiently regular solution 
$\big(n(t,\barx),\Temp(t,\barx),u(t,\barx)\big)$ of the relativistic Euler system (rE from now on) satisfying the technical conditions \eqref{juttnerB} can be used to construct a corresponding family of classical solutions $F^\varepsilon$ of the rB equation. Here, $\Temp$ denotes the the fluid \emph{temperature}, $n$ denotes the fluid \emph{proper number density}, and  $u$ denotes the fluid \emph{four-velocity}. Roughly speaking, the technical conditions are the assumption that $\Temp(t,\barx)$ is uniformly positive with only mild fluctuations, that $n(t,\barx)$ is uniformly bounded from above and below away from $0,$ and that the spatial components of the four-velocity, namely $u^1(t,\barx), u^2(t,\barx),$ and $u^3(t,\barx),$ are uniformly small. In Lemma \ref{L:nearconstantstates}, we give a simple proof that these conditions are always satisfied near the constant fluid states. Under these assumptions, our estimates show that that as $\varepsilon \to 0^+$ (in the hydrodynamic limit),
the rB solution $F^\varepsilon$ converges to the \emph{relativistic Maxwellian}\footnote{These are also known as \emph{J\"{u}ttner distributions}.} $\mathcal{M}$ associated to the solution of the rE system; see Section \ref{SS:Maxwellians} for the definition of a Maxwellian, which is a function in equilibrium with the collision process; i.e., $\mathcal{C}(\mathcal{M},\mathcal{M}) = 0.$ Thus, our results show that for small $\varepsilon,$ there are near-local equilibrium solutions to the rB equation whose underlying dynamics are effectively captured by the rE system. For the Newtonian Boltzmann equation, these are called ``normal solutions'' in Grad \cite{MR0135535}.


The main strategy of our proof is to perform a \emph{Hilbert expansion} (see Section \ref{SS:HilbertExpansion}) for 
$F^\varepsilon$.  We write 
\begin{equation}
F^\varepsilon = F_0+\sum_{k = 1}^6 \varepsilon^k F_k + \varepsilon^3 \effRE.
\label{hilbertE}
\end{equation}
Inserting this expansion into \eqref{E:rBintro} and equating like powers of $\varepsilon$ results in a hierarchy of equations. It turns out that $F_0, \ldots, F_6,$ which do not depend on $\varepsilon,$ can be solved for: $F_0$ must be Maxwellian, while $F_1, \ldots, F_6$ solve linear equations with inhomogeneities. Thus all of the difficult (and $\varepsilon-$dependent) analysis is contained in the analysis of the \emph{remainder term} $\effRE$, which is carried out in Section \ref{SS:rBestimates}. 

Our methods and results can be viewed as an extension of the program initiated by Caflisch \cite{MR586416}, who proved analogous results for the non-relativistic Boltzmann equation and Euler equations \cite{MR586416}. 
We will  utilize strategies from Caflisch as
well as  Guo \cite{G2} and 
Guo-Jang-Jiang
 \cite{MR2472156} to perform the Hilbert expansion.  We will also use relativistic Boltzmann estimates from the work of the second author \cite{strainSOFT,strainNEWT}  in order to control the expansion. Additionally, we develop several necessary tools to study this problem in the setting of special relativity. In particular, we develop a mathematical theory for the kinetic equation of state, which is described just below. A more detailed discussion of the existing literature related to our result is located in Section \ref{SS:History}.

Before proving the aforementioned results, we will sketch a proof of local existence (in Section \ref{S:rE}) for the relativistic Euler equations. This local existence result ensures that there are in fact solutions to the rE system that can be used in the aforementioned construction. However, we mention upfront that during the course of our investigation, we ran into several technical difficulties that, to our surprise, seem to be unresolved in the literature. The first concerns the fundamental question of which fluid variables can be used as state-space variables in the rE system. In addition to the four-velocity $u,$ there are five other fluid variables that play a role in the ensuing discussion: the aforementioned variables $n$ and $\Temp,$ together with the \emph{entropy per particle} $\Ent,$ the \emph{pressure} $p,$ and the \emph{proper energy density} $\rho.$ In order to close the Euler equations, one must assume relations between the
fluid variables. In this article, we assume that the three relations \eqref{E:PnTEulerrelationintro} - \eqref{E:nEntzrelationintro} hold between the five non-negative variables $n, \Temp, \Ent, \pressure,$ and $\rho.$ As will be discussed below, these choices were not made arbitrarily, but are in fact satisfied by the macroscopic quantities $ n[\mathcal{M}], \Temp[\mathcal{M}], \Ent[\mathcal{M}], \pressure[\mathcal{M}],$ and $\rho[\mathcal{M}]$ corresponding to a relativistic Maxwellian $\mathcal{M};$ these quantities are defined in Section \ref{SS:Macroscopicquantities}. We emphasize that \textbf{the relations \eqref{E:PnTEulerrelationintro} - \eqref{E:nEntzrelationintro} are required in order for the rE system to arise from the rB equation in the hydrodynamic limit}. Now it is commonly \emph{assumed} 
that as a consequence of the three relations, any two of $n, \Temp, \Ent, \pressure, \rho$ uniquely determine the remaining
three. In particular, during our construction of the fluid solutions, we need to be able to go back and forth between the variables $(n, \Temp)$ and the variables $(\Ent, \pressure),$ i.e., we need to be able invert the smooth map 
$(n,z) \rightarrow \big(\mathfrak{H}(n,z), \mathfrak{P}(n,z) \big),$ where $z = \frac{m_0c^2}{k_B \Temp},$ 
$k_B > 0$ denotes \emph{Boltzmann's constant}, and $\mathfrak{H}$ and $\mathfrak{P}$ are defined in \eqref{E:changeofstatespacevariablesH} - \eqref{E:changeofstatespacevariablesP} below. However, we were unable to find a fully rigorous proof of the invertibility of this map in the literature. Consequently, in Lemma \ref{L:invertiblemaps} below, we use asymptotic expansions for Bessel functions to rigorously verify the local invertibility of this map outside of a compact set. In particular, we show that the map is locally invertible whenever $\Temp$ is sufficiently large, and whenever $\Temp$ is sufficiently small and positive. Additionally, the numerical plot in Figure \ref{fig:partialppartialz} (see Section \ref{SS:SolveforTemp}), which covers the compact set in question, strongly suggests that the map $(n,z) \rightarrow (\mathfrak{H}(n,z),\mathfrak{P}(n,z)) = (\Ent, \pressure),$ is an auto-diffeomorphism of the region $(0, \infty) \times (0, \infty).$ This would imply that  we
  can always smoothly transform back and forth between $(n, \Temp)$ and $(\Ent, \pressure)$ in the region of physical relevance, i.e., the region in which all of the quantities are positive; see Conjecture \ref{C:Tempisgood} in Section \ref{S:rE} below.  Similarly, in Lemma \ref{L:speedofsound}, we rigorously prove that outside of the same compact set of temperature values, there exists a \emph{kinetic equation of state} $\pressure = f_{kinetic}(\Ent,\rho),$ which gives the fluid \emph{pressure} $\pressure$ as a function of the \emph{entropy per particle} $\Ent$ and the \emph{proper energy density} $\rho.$ 

A related issue is the fact that in order for the rE system to be hyperbolic and causal\footnote{By causal, we mean that the speed of sound is less than the speed of light.} under a general equation of state $p = f(\Ent,\rho),$ it is necessary and sufficient to prove that $0 < \left. \frac{\partial f}{\partial \rho} \right|_{\Ent} < 1.$ We explain why the positivity of $\left. \frac{\partial f}{\partial \rho} \right|_{\Ent}$ is needed for our proof of local existence in Remark \ref{R:hyperbolicity} below, while the mathematical connection between the inequality $\left. \frac{\partial f}{\partial \rho} \right|_{\Ent} < 1$ and the speed of sound propagation being less than the speed of light is explained in e.g. \cite{jS2008a}. Now in the case of the kinetic equation of state $\pressure = f_{kinetic}(\Ent,\rho),$ $\left. \frac{\partial f_{kinetic}}{\partial \rho} \right|_{\Ent}$ can be written as a function of $\Temp$ alone. It is possible to write down
a closed form expression for this latter quantity (see equation \eqref{E:inversespeedofsoundsquaredBessel}), but since the formula is a rather complicated one involving ratios of Bessel functions, we have only analytically verified the inequality $0 < \left. \frac{\partial f_{kinetic}}{\partial \rho} \right|_{\Ent} < 1$ (again using asymptotic expansions for Bessel functions) outside of the same compact set discussed in the previous paragraph; see Lemma \ref{L:speedofsound}. Therefore, the fully rigorous version of our local existence result is currently limited to initial data whose temperature
avoids the compact set in question. However, we have numerically observed that in fact, the stronger inequality $0 < \left. \frac{\partial f_{kinetic}}{\partial \rho} \right|_{\Ent} < \frac{1}{3}$ should hold for \emph{all} $\Temp > 0;$ see Conjecture \ref{C:Speedisreal} in Section \ref{C:Speedisreal}, and Figure \ref{fig:toucan} in Section \ref{SS:hyperbolicity}. 
This stronger inequality would imply that the speed of sound under the kinetic equation of state is never larger than $\sqrt{1/3}$ times the speed of light. 

In view of these complications, when stating the hypotheses for our local existence theorem (Theorem \ref{MainThmEuler} in Section \ref{SS:Mainresults}), we make careful assumptions on the fluid initial data that are designed to ensure that they fall within the regime of hyperbolicity, and within a regime in which the aforementioned map $(n, \Temp) \rightarrow (\Ent, \pressure)$ is invertible with smooth inverse. However, if our two conjectures are in fact correct, then the relations \eqref{E:PnTEulerrelationintro} - \eqref{E:nEntzrelationintro} imply that many of these assumptions are automatically verified
whenever the fluid variables are positive. Aside from these complications, our local existence theorem is a standard result. However, there are several additional aspects of it that are worthy of mention. First, we avoid the use of symmetrizing variables in our proof. We instead use the framework of \emph{energy currents}, which was first applied by Christodoulou to the rE system in \cite{dC2007}, and which was later expounded upon by the first author in \cite{jS2008a}. We also remark that our local existence result only applies to initial data with proper energy density $\rho$ that is \emph{uniformly positive}. In particular, we avoid addressing the complicated issue of the free-boundary problem for the relativistic Euler equations.
A related comment is that our local existence result produces a spacetime slab $[0,T] \times \mathbb{R}^3$ on which the uniform positivity property is preserved. 
We remark that in view of the assumptions on $n, \Temp$ mentioned near the beginning of the article, we will only study fluid solutions belonging to \emph{compact} subsets of the regions of interest to us, that is, regions where the maps $\mathfrak{H}$ and $\mathfrak{P}$ are rigorously known to be invertible. Thus, on such compact subsets, the uniform positivity of $\rho$ is an automatic consequence of the continuity of the map $(n,z) \rightarrow \rho,$ which is implicitly defined by the relations
\eqref{E:PnTEulerrelationintro} - \eqref{E:nrhoprelationintro}.

\subsection{Notation and conventions} \label{SS:Notation}
We now summarize some notation and conventions that are used throughout the article. $M$ denotes Minkowski space, while $\mathfrak{M}$ denotes the mass shell. In general, Latin (spatial) indices $a,b,j,k,$ etc., take on the values $1,2,3,$ while Greek indices $\kappa, \lambda, \mu, \nu,$ etc., take on the values $0,1,2,3.$
Indices are raised and lowered with the Minkowski metric $g_{\mu \nu}$ and its inverse $(g^{-1})^{\mu \nu}.$ For most of the article, we work in a fixed inertial coordinate system on $M,$ in which case \begin{eqnarray} \label{E:spacetimemetric}
	g_{\mu \nu} = (g^{-1})^{\mu \nu} = \mbox{diag}(-1,1,1,1),
\end{eqnarray}
and $P^{\kappa}Q_{\kappa} = - P^0 Q^0 + \sum_{a=1}^3 P^a Q^a.$ Here and throughout, we use Einstein's summation convention that repeated indices, with one``up" and one ``down" are summed over.

When differentiating with respect to state-space variables, we use the notation 
\begin{align}
\notag
	\partial_U|_{V},
\end{align}
to mean partial differentiation with respect to the quantity $U$ while $V$ is held constant.
We define
$\barx = (x^1,x^2,x^3)$,
$\barp \eqdef (P^1,P^2,P^3),$ $P^0 \eqdef (m_0^2 c^2 + |\barp|^2)^{1/2},$ and $|\bar{P}|^2 \eqdef \sum_{a=1}^3 (P^a)^2$.
Furthermore $\barq, \barp', \barq'$ are treated similarly.
We also define 
$\phat \eqdef (P^0)^{-1} \barp,$ and similarly for $\qhat.$
We use the symbol 
$$
\partial_{\barx} \eqdef \Big(\frac{\partial}{\partial x^1}, \frac{\partial}{\partial x^2}, \frac{\partial}{\partial x^3} \Big) =(\partial_1, \partial_2, \partial_3),
$$ 
to denote 
the \emph{spatial} coordinate gradient. The Sobolev norm $\| \cdot \|_{H_{\barx}^N}$ of a Lebesgue measurable function $f(\barx)$ on $\mathbb{R}^3$
is defined in the usual way:
\begin{align}     \label{E:HNnormdef}
	\| f \|_{H_{\barx}^N}
 		\eqdef \left( \sum_{|\vec{\alpha}| \leq N} \| \partial_{\vec{\alpha}}f \|_{L_{\barx}^2}^2 \right)^{1/2},
\end{align}
where 
$
\partial_{\vec{\alpha}} = \partial_1^{n_1} \partial_2^{n_2} \partial_3^{n_3},
$
${\vec{\alpha}}=(n_1,n_2,n_3)$ is a \emph{spatial} coordinate-derivative multi-index, and 
$|\vec{\alpha}| = n_1 + n_2 + n_3.$ Here and throughout, we use the abbreviation $H_{\barx}^N \eqdef H_{\barx}^N(\mathbb{R}_{\barx}^3).$ We also define the $L_{\barp}^2$ inner product of two functions $F(\barp), G(\barp)$ as follows:
\begin{align}
	\langle F, G \rangle_{\barp} \eqdef \int_{\mathbb{R}^3_{\barp}} F(\barp) G(\barp) d\barp.
	\label{L2inner}
\end{align}
For brevity, we sometimes write
$\langle \cdot, \cdot \rangle \eqdef \langle \cdot, \cdot \rangle_{\barp}$.
The $L^2(\mathbb{R}^3_{\barp})$ norm is denoted $|\cdot |_2$.  
Similarly, we define the $L_{\barx;\barp}^2$ inner product of two functions $F(\barx, \barp), G(\barx, \barp)$ as
\begin{align}
	\langle F, G \rangle_{\barx;\barp} \eqdef \int_{\mathbb{R}^3_{\barx}}  \int_{\mathbb{R}^3_{\barp}} F(\barx,\barp) G(\barx,\barp) d \barx d \barp.
	\notag
\end{align}
We denote the corresponding norm by 
$
\| f \|_{2}
\eqdef
\| f \|_{H^0(\mathbb{R}^3_{\barx} \times \mathbb{R}^3_{\barp})}
=  \| f \|_{L^2(\mathbb{R}^3_{\barx} \times \mathbb{R}^3_{\barp})}.
$
We furthermore define the norm
\begin{equation}
\notag
\| h\|_\infty
\eqdef
{\rm ess~sup}_{x\in \mathbb{R}^3_{\barx}, p\in\mathbb{R}^3_{\barp}}  |h(\barx,\barp)|.
\end{equation}
For each $\ell \geq 0,$ we also define the weight function $w_\ell$ as
\begin{equation}
w_\ell =
w_\ell(\barp) \eqdef 
\left(1+|\barp|^2\right)^{\ell /2}. \label{weight}
\end{equation}
We then define a corresponding weighted $L^\infty$ norm by
\begin{equation}
\notag
\| h\|_{\infty,\ell}
\eqdef
{\rm ess~sup}_{x\in \mathbb{R}^3_{\barx}, p\in\mathbb{R}^3_{\barp}}  | w_\ell(\barp)  h(\barx,\barp)|.
\end{equation}
We define the $H_{\barx}^N$ norm of a Lebesgue measurable function $f(\barx)$ over a measurable 
subset $E \subset \mathbb{R}_{\barx}^3$ by
\begin{align}     
\notag
	\| f \|_{H_{\barx}^N(E)}
 		\eqdef \left( \sum_{|\vec{\alpha}| \leq N} \| \partial_{\vec{\alpha}}f \|_{L_{\barx}^2(E)}^2 \right)^{1/2},
\end{align}
where 
\begin{align}\notag
	\| f \|_{L_{\barx}^2(E)} \eqdef \left(\int_{E} |f|^2 \ d \barx \right)^{1/2},
\end{align}
and similarly for the other norms and inner products over a subset. If $X$ is a normed function space, then we use the notation
$C^j([0,T],X)$ to denote the set of $j$-times continuously differentiable maps from $(0,T)$ into
$X$ that, together with their derivatives up to order $j,$ extend continuously to $[0,T].$ We sometimes use the notation $A \lesssim B$ to mean that there exists an inessential uniform constant $C$ such that $A \le C B$. Generally $C$ will denote an inessential uniform constant whose value may change from line to line.  For essential constants, we always write down their dependence explicitly.

\subsection{Lorentzian geometry and the mass shell $\mathfrak{M}$} \label{SS:massshell}

In this article, we primarily work in a fixed \emph{inertial} coordinate system on $M,$ which is a global rectangular coordinate system $\lbrace x^{\mu} \rbrace_{\mu = 0,1,2,3}$ in which the \emph{spacetime metric} $g_{\mu \nu}$ has the form \eqref{E:spacetimemetric}. Note the sign convention of \eqref{E:spacetimemetric}. This is the most common sign convention found in the relativity literature, but it is opposite of the sign convention that is sometimes found in the relativistic Boltzmann literature.

Our coordinate system $\lbrace x^{\mu} \rbrace_{\mu = 0,1,2,3}$ represents a special choice of a ``space-time" splitting. We identify $x^{0}$ with $ct,$ where $c$ is the speed of light and $t$ is time, while we identify $(x^1, x^2, x^3) \eqdef \barx$
with a ``spatial coordinate":
\begin{align*}
	x=(x^0,x^1,x^2,x^3) = (ct,\barx).
\end{align*}
We often work with the coordinate $t$ rather than $x^0.$  Note that $\partial_0 = \frac{1}{c}\partial_t.$

The \emph{mass shell} is a subset of 
$$
T M \eqdef \cup_{x \in M} T_x M,
$$ 
the tangent bundle of $M.$ In the following, we use an inertial coordinate system 
$\lbrace x^{\mu}, Y^{\nu} \rbrace_{\mu, \nu = 0,1,2,3}$  on $TM,$ where $\lbrace x^{\mu} \rbrace_{\mu = 0,1,2,3}$  is the inertial coordinate system on $M,$ and
$$
T_x M =\left\lbrace \left. Y^{\kappa} \frac{\partial}{\partial x^{\kappa}}\right|_x  \ \big| \ (Y^1,Y^2,Y^3,Y^4) \in \mathbb{R}^4 \right\rbrace.
$$
In the above expression, $(Y^1,Y^2,Y^3,Y^4)$ are the coordinates of vectors in $T_x M \simeq \mathbb{R}^4$ relative to the basis $\lbrace \left.\frac{\partial}{\partial x^0}\right|_x, 
\left. \frac{\partial}{\partial x^1}\right|_x, \left. \frac{\partial}{\partial x^2}\right|_x, 
\left. \frac{\partial}{\partial x^3}\right|_x \rbrace.$ The mass shell $\mathfrak{M}$ is defined to be
\begin{align}
	\mathfrak{M} \eqdef \lbrace (x,P) \in TM \ | \ P_{\kappa} P^{\kappa} = - m_0^2c^2 \ \mbox{and} \ P^0 > 0 \rbrace. 
	\nonumber
\end{align}
Let us also define
\begin{align}
	\mathfrak{M}_x \eqdef \mathfrak{M} \cap T_x M,
	\nonumber
\end{align}
and
\begin{align*}
	\phi(\bar{P}) \eqdef P^0 = \sqrt{m_0^2c^2 + |\bar{P}|^2}.
\end{align*} 
It follows that the map $\Phi: \mathbb{R}^3 \rightarrow \mathfrak{M}_x$ defined by
\begin{align}
	\Phi(P^1,P^2,P^3) \eqdef (\phi(P^1,P^2,P^3),P^1,P^2,P^3),
	\nonumber
\end{align}
is a diffeomorphism, and we can use it to put coordinates on $\mathfrak{M}_x;$ i.e., any element of 
$\mathfrak{M}_x,$ viewed as a submanifold of $T_x M,$ has components $ (P^0,P^1,P^2,P^3) 
= (\phi(\bar{P}), \bar{P})$ relative to our rectangular coordinate system. It follows that \\
$\lbrace x^{\mu}, P^j \rbrace_{\mu =0,1,2,3; j=1,2,3}$ is a global coordinate system on $\mathfrak{M}.$ If $(x,P)$ is an element of $\mathfrak{M},$ then we often slightly abuse notation by identifying $(x,P)$ with $(t,\barx,\bar{P}).$ We similarly identify $F(x,P)$ with $F(t,\barx,\barp).$

We recall that there is a canonical measure
\begin{align} \label{E:gbarvolumeform}
	d \mu_{\bar{g}} \eqdef 
	\sqrt{|\mbox{det}(\bar{g})}| d \barp,
\end{align} 
associated to $\bar{g},$ the first fundamental form of $\mathfrak{M}_x.$ We remark that $\bar{g}$ is Riemannian since $\mathfrak{M}_x$ is a spacelike hypersurface in $T_x M.$ This measure will allow us to define (in a geometrically invariant manner) integration over the surface $\mathfrak{M}_x.$ Recall that since $\mathfrak{M}_x$ is 
(relative to the coordinate system $\lbrace Y^{\nu} \rbrace_{\nu = 0,1,2,3}$ on $T_x M$) the level set 
$$
\mathfrak{M}_x = \lbrace (P^0,P^1,P^2,P^3) \in T_x M \ | \ P^0 - \phi(P^1,P^2,P^3) = 0 \rbrace,
$$ 
it follows that $\bar{g} = \Phi_* g,$ where $\Phi_* g$ is the pullback of $g$ by $\Phi.$ Simple calculations imply that in our inertial coordinate system, we have
\begin{align} \label{E:bargab}
	\bar{g}_{jk} = \frac{\partial \Phi^{\kappa}}{\partial P^j} \frac{\partial \Phi^{\lambda}}{\partial P^k} g_{\kappa 
		\lambda} = - \frac{P^j P^k}{(P^0)^2} + \delta_{jk}, && (j,k=1,2,3).
\end{align}
Using \eqref{E:gbarvolumeform} and \eqref{E:bargab}, we compute that
the canonical measure associated to $\bar{g}$ can be expressed as follows relative to the coordinate system
$\barp$ on $\mathfrak{M}_x:$
\begin{align} \label{E:measureonPx}
	d \mu_{\bar{g}} = \frac{1}{P^0} d \barp,
\end{align} 
where we have used the fact that $|\mbox{det}(\bar{g})| = (P^0)^{-2}.$ 
We remark that \eqref{E:measureonPx} is valid
only in an inertial coordinate system, and that integrals relative to the measure  $d \mu_{\bar{g}}$ will 
play a central role in the definitions and analysis of Section \ref{SS:Macroscopicquantities}.

\subsection{Hypotheses on the collision kernel}

In order to state our hypotheses on the collision kernel, we introduce the following expression for the Boltzmann collision operator, which is local in $(t,\barx)$; we will elaborate upon it in Section \ref{SS:alternateexpression}:
\begin{equation} \label{collisionCMintro}
\mathcal{C}(F,G) 
= P^0
\int_{\mathbb{R}^3\times \mathbb{S}^2} v_{\o} \sigma (\varrho,\vartheta)[F(\barp')G(\barq')-F(\barp)G(\barq)] d \barq d\omega,
\end{equation}
where we have suppressed the dependence of $F$ and $G$ on $(t,\barx)$. Note that the collision operator
acts only on the $P$ variables. In the above expression,
$v_{\o}=v_{\o}(\barp,\barq),$ the \emph{M\o ller velocity}, is defined by
\begin{equation}
v_{\o}=
v_{\o}(\barp,\barq) \eqdef 
\frac{c}{2}\sqrt{\left| \frac{\barp}{P^{0}}-\frac{\barq}{Q^{0}}\right|^2-\frac{1}{c^2}\left| \frac{\barp}{P^{0}}\times\frac{\barq}{Q^{0}}\right|^2}
= \frac{c}{4}\frac{\varrho \sqrt{s}}{P^{0} Q^{0}},
\label{moller}
\end{equation}
where $\times$ denotes the cross product in $\mathbb{R}^3.$ In \eqref{collisionCMintro}, $\sigma$ is the \emph{differential cross-section}, or the \emph{collision kernel}.  The \emph{relative momentum} $\varrho(\barp,\barq)$ is defined by
\begin{gather}
\varrho 
\eqdef
\sqrt{(P^\kappa-Q^\kappa) (P_\kappa-Q_\kappa)}
=\sqrt{-2(P^\kappa Q_\kappa + m_0^2c^2)} \geq 0,
\label{gDEFINITION}
\end{gather} 
while the \emph{scattering angle} $\vartheta(\barp,\barq,\bar{P}',\bar{Q}')$ is defined by
\begin{equation} 
\cos\vartheta
\eqdef
(P^\kappa - Q^\kappa) (P'_{\kappa} -Q'_{\kappa})/\varrho^2.
\label{angle}
\end{equation}
Here the variables $P'$ and $Q'$ are defined in terms of $\barp,\barq$ below in \eqref{postCOLLvelCMsec2}. Finally, $s(\barp,\barq),$ which is defined by $c^2 s =$ the energy in a center-of-momentum frame\footnote{A center-of-momentum frame is a Lorentz frame in which $P^{\mu} + Q^{\mu} = P'^{\mu} + Q'^{\mu} = (\sqrt{s},0,0,0)$}, can be expressed as
\begin{eqnarray}
s =
-(P^\kappa+Q^\kappa) (P_\kappa+Q_\kappa)
=
2\left( -P^\kappa Q_\kappa + m_0^2c^2 \right)\geq 0.
\label{sDEFINITION}
\end{eqnarray}
Notice that $s=\varrho^2 + 4c^2$. We warn the reader that this notation, which is used in \cite{sdGsvLwvW1980}, may differ from other authors notation by a constant factor.
Furthermore, $\omega$ is an element of $\mathbb{S}^2$ (viewed as a submanifold of $\mathbb{R}^3$), which can (with the exception of the north pole) be parameterized by the angles $(\vartheta, \varphi) \in (0,\pi] \times (0, 2 \pi],$ where $\vartheta$ is  from above, and $\varphi$ is an azimuthal angle. Relative to these coordinates, we have that $d \omega = \mbox{sin} \vartheta d \vartheta d \varphi.$

The function $\sigma$ depends on the chosen model of particle interaction. For the remainder of the article, we 
assume the following:  \\

\noindent {\bf Hypotheses on the collision kernel:}  
{\it We assume that there are constants $C_1$, $C_2$, $a$, and $b$ such that
the differential cross-section $\sigma$, which is listed above in \eqref{collisionCMintro} and below in \eqref{collisionCM}, 
satisfies the inequalities
$$
\sigma (\varrho,\vartheta) 
\le 
\left( C_1 ~ \varrho^{a}+C_2 ~ \varrho^{-b}\right) ~ \sigma_0(\vartheta),
$$
where
\begin{itemize}
	\item $C_1$ and $C_2$ are non-negative
	\item There exists a number $\gamma > -2$ such that $0\le \sigma_0(\vartheta) \le \sin^\gamma\vartheta$
	\item $0\le a < \min(2,2+\gamma),$ and $0 \le b <\min(4,4+\gamma).$
\end{itemize}
We also assume that there exist constants $C \geq 1$ and $\beta \in (-4,2)$ such that 
		\begin{equation}
		\frac{1}{C} (P^0)^{\beta/2}
		\le
		\nu(\barp) 
		\le C (P^0)^{\beta/2}. \label{hypNU}
		\end{equation}
Here $\nu(\barp),$ the collision frequency, is defined by
$$
\nu(\barp)
\eqdef
\int_{\mathbb{R}^3}  d\barq ~
\int_{\mathbb{S}^{2}} d\omega ~
~ v_{\o} ~
 \sigma(\varrho,\vartheta) ~
J(\barq),
$$ 
and $J(\barq)$ is the global relativistic Maxwellian defined by 
\begin{equation} \label{juttner}
J(\barq)
\eqdef 
 e^{-c Q^0/(k_B \Temp_M)},
\end{equation}
and $\Temp_M >0$ is a constant. The lower bound in \eqref{hypNU} is satisfied if there is a suitable lower bound for $\sigma.$  See \cite{strainSOFT,strainPHD,MR1379589,MR933458}, and we refer specifically to \cite{strainSOFT} for more details. 
}  \\

The upper bound in \eqref{hypNU} is sufficient to deduce some of the estimates that we use below, such as Lemma \ref{newNONlin}, Lemma \ref{boundK2}, Lemma \ref{boundKinfX}, and Lemma \ref{sSOFT}. Our hypothesis originates from the general physical assumption introduced in \cite{MR933458}; see also \cite{DEnotMSI} for further discussions. Standard references in relativistic Kinetic theory include \cite{MR1898707,sdGsvLwvW1980,MR1379589,MR0088362,nla.cat-vn2540748}.

\subsection{Maxwellians} \label{SS:Maxwellians}
We now introduce the Maxwellians, a special class of functions on $\mathfrak{M}$ that play a fundamental role in
connecting the relativistic Boltzmann equation to the relativistic Euler equations. Given any functions
$n = n(t,\barx)$, $\Temp = \Temp(t,\barx)$, $u^{\mu} = u^{\mu}(t,\barx)$ on $M$ such that 
$
n > 0, \Temp > 0, u^{0} > 0, u_{\kappa} u^{\kappa} = -c^2,
$
we define the 
corresponding \emph{Maxwellian} $\mathcal{M} = \mathcal{M}(n,\Temp,u;\barp)$ as follows:
\begin{align}
	\mathcal{M} = \mathcal{M}(n,\Temp,u;\barp) & \eqdef n \frac{z}{4 \pi m_0^3 c^3 K_2(z)} 
		\exp\Big(\frac{z u_{\kappa}P^{\kappa}}{m_0c^2} \Big), \label{E:Maxwelliandef} 
\end{align}
where the dimensionless variable $z$ is defined by
\begin{equation}
\label{zdef}
z \eqdef \frac{m_0c^2}{k_B \Temp},
\end{equation}
$k_B$ is again \emph{Boltzmann's constant}, and $K_2(z)$ is the Bessel function defined in \eqref{E:Besseldef}. It can be shown (see e.g. \cite[Chapter 2]{sdGsvLwvW1980}) that
\begin{align} \label{E:MaxwellianCondition}
\mathcal{C}(F,F)  \equiv 0 
\quad \iff 
\quad F \ \mbox{is a Maxwellian of the form} \ \mathcal{M}(n,\Temp,u;\barp).   
\end{align}
For this reason, a Maxwellian $\mathcal{M}$ is said to be in \emph{local equilibrium} with the collision process.
If $n(t,\barx), \Temp(t,\barx),$ and $u^{\mu}(t,\barx)$ are constant valued, then $\mathcal{M}$ is said to be in \emph{global equilibrium}. Note also that the function $J(\barp)$ appearing in \eqref{juttner} is a global Maxwellian. It will play a distinguished role in the analysis of Section \ref{SS:rBestimates}. In fact, we will assume that the  
local Maxwellian corresponding to the fluid solution is uniformly comparable to powers of $J(\barp);$ see \eqref{juttnerB}.

\begin{remark}
	Although $\mathcal{C}(\mathcal{M},\mathcal{M}) = 0,$ it is not true in general that $\mathcal{M}$ is a solution to the rB equation \eqref{E:rBintro}; i.e., in general, $P^{\kappa} \partial_{\kappa} \big(\mathcal{M}(n,\Temp,u;\barp) \big) \neq 0.$
Nevertheless, our main result (Theorem \ref{MainThm}) shows that under certain assumptions, including that $(n,\Temp,u)$
are solutions to the rE system, there is a solution of \eqref{E:rBintro} near $\mathcal{M}.$

\end{remark}

\subsection{The energy-momentum tensor and the particle current for rB}
\label{SS:conservationlawsrBe}

We now define $T_{Boltz}[F],$ which is the \emph{energy-momentum-stress-density tensor} (energy-momentum tensor for short) for the relativistic Boltzmann equation, and $I_{Boltz}[F],$ which is the \emph{particle current}. Given any function $F(\barp),$ these quantities are defined as follows:
\begin{equation} 
\label{E:TBoltzmanndef}
\begin{split} 
T_{Boltz}^{\mu \nu}[F] & \eqdef c \int_{\mathbb{R}^3_{\barp}} P^{\mu} P^{\nu} F(\barp) \frac{d \barp}{P^{0}},  \quad (0 \leq \mu,\nu \leq 3),
				\\
I_{Boltz}^{\mu}[F] & \eqdef c \int_{\mathbb{R}^3_{\barp}} P^{\mu} F(\barp) \frac{d \barp}{P^{0}},  \quad  (0 \leq \mu \leq 3). 
\end{split} 
\end{equation} 
Whenever there is no possibility of confusion, we abbreviate $T_{Boltz} = T_{Boltz}[F],$ and similarly for the additional quantities depending on $F$ that appear below. The conservation laws can be summarized as follows (see Lemma \ref{L:Boltzmannconservation}): whenever $F(t,\barx,\barp)$ is a  classical solution to the relativistic Boltzmann equation, the following conservation laws hold
\begin{equation} 
\label{E:particlecurrentconservationBoltzmann}
\begin{split} 
	\partial_{\kappa} (T_{Boltz}^{\mu \kappa}[F]) & =  0, \quad (0 \leq \mu \leq 3),
				\\
	\partial_{\kappa} (I_{Boltz}^{\kappa}[F]) & =  0. 
\end{split} 
\end{equation} 
It is explained in the next section that whenever $F = \mathcal{M}(n,\Temp,u;\barp),$ and $(n,\Temp,u)$ are a solution
to the rE system, then the above conservation laws for $F$ hold \emph{even though $F$ need not be a solution to the rB equation}.  This fact will play an important role during our discussion of the relationship of the rE system to the rB equation.

\subsection{The relativistic Euler equations and their relationship to the relativistic Boltzmann equation}
\label{SS:rEintro}

In this section, we recall some basic facts about the rE equations in Minkowski space. This is intended to serve as background for Theorem \ref{MainThmEuler}, which is stated in Section \ref{SS:Mainresults}, and proved in Section \ref{S:rE}. For a detailed discussion of the rE system, we refer the reader to Christodoulou's survey article \cite{dC2007b}; here we only provide a brief introduction. Our other goal in this section is to illustrate some of the formal correspondences between the rE system and the rB equation. These provide a heuristic basis for the expectation that the rE system should emerge from the rB equation in the hydrodynamic limit. 

The rE system models the evolution of a perfect fluid evolving in a Lorentzian spacetime. In Minkowski space, the spacetime of special relativity, they are 
\begin{equation}  
\label{E:nuconservationlawintro}
\begin{split} 
\partial_{\kappa} T_{fluid}^{\mu \kappa} & = 0, \quad (0 \leq \mu \leq 3), 
	\\
	\partial_{\kappa} I_{fluid}^{\kappa} & = 0. 
\end{split} 
\end{equation} 
Here the \emph{energy-momentum-stress-density tensor} (energy-momentum tensor for short) for a perfect fluid
has components
\begin{align}
 \label{E:Tfluiddef}
	T_{fluid}^{\mu \nu}  & \eqdef c^{-2}(\rho + \pressure)u^{\mu} u^{\nu} + \pressure (g^{-1})^{\mu \nu}, && (\mu, \nu = 0,1,2,3),
\end{align}
where $\rho \geq 0$ is the \emph{proper energy density}, $\pressure \geq 0$ is the \emph{pressure},
and $u$ is the \emph{four-velocity}.  The four-velocity is a future-directed 
(i.e., $u^0 > 0$ in our inertial coordinate system) vectorfield that satisfies the normalization condition
\begin{align}
\label{E:unormalized}  
	u_{\kappa} u^{\kappa} & = - c^2.
\end{align}
The vectorfield $I_{fluid}^{\mu}$ is the \emph{particle current}, it is proportional to the four-velocity:
\begin{align*} 
	I_{fluid}^{\mu} & \eqdef n u^{\mu}, && (\mu =0,1,2,3).
\end{align*}
The quantity $n \geq 0$ is the \emph{proper number density}. All of these
quantities are functions of $(t,\barx) \in \mathbb{R} \times \mathbb{R}^3.$ 

By projecting the first divergence in \eqref{E:nuconservationlawintro} in the direction parallel to $u$ and onto the $g-$orthogonal complement of
$u,$ we  can rewrite \eqref{E:nuconservationlawintro} (omitting some standard calculations) in the well-known form
\begin{subequations}
\begin{align}
	\partial_{\kappa}(n u^{\kappa}) & = 0, &&
	\label{E:nandulaw2} 
	\\
	u^{\kappa} \partial_{\kappa} \rho + (\rho + \pressure) \partial_{\kappa} u^{\kappa} & = 0, && \label{E:Euleru} 
	\\
	(\rho + \pressure)u^{\kappa}\partial_{\kappa} u^{\mu} + \Pi^{\mu \kappa} \partial_{\kappa} \pressure 
	& = 0, && (\mu = 0,1,2,3), \label{E:Euleruperp}
\end{align}
\end{subequations}
where $\Pi,$ the two-tensor that projects onto the $g-$orthogonal complement of $u$ (i.e. $\Pi^{\mu \kappa}u_{\kappa} = 0$)
has the components
\begin{align} \label{E:Pidef}
	\Pi^{\mu \nu} & \eqdef c^{-2}u^{\mu} u^{\nu} + (g^{-1})^{\mu \nu}, && (\mu, \nu = 0,1,2,3).
\end{align}
The above equations are redundant in the following sense: if \eqref{E:unormalized} and \eqref{E:Euleru} hold, and if \eqref{E:Euleruperp} holds for $\mu=1,2,3,$ then it follows that \eqref{E:Euleruperp} also holds when $\mu=0.$

The equations \eqref{E:nuconservationlawintro} are not closed because there are more unknowns than equations. In order to close the rE system in a manner compatible with the rB equation, we will make use of the additional fluid variable $\Ent,$ a non-negative quantity known as the \emph{entropy per particle}. We now make the following assumptions, which are of crucial importance: we assume that the fluid variables $n$, $\Temp$, $\Ent$, $\pressure$, $\rho$ are bound by the relations 
\begin{subequations}
\begin{align}
	\pressure & = k_B n \Temp = m_0 c^2 \frac{n}{z},  \label{E:PnTEulerrelationintro}
	\\
	\rho & = m_0 c^2 n \frac{K_1(z)}{K_2(z)} + 3 \pressure,
		\label{E:nrhoprelationintro} \\
	n & = 4\pi e^4 m_0^3 c^3 h^{-3} \exp \Big(\frac{-\Ent}{k_B} \Big) \frac{K_2(z)}{z}\exp\Big(z 
			\frac{K_1(z)}{K_2(z)} \Big), \label{E:nEntzrelationintro}
\end{align}
\end{subequations}
where $k_B > 0$ is \emph{Boltzmann's} constant, $h>0$ is Plank's constant, 
and $z$ is defined in \eqref{zdef}. Furthermore, the $K_j(\cdot)$ are modified second order Bessel functions, which are defined in Lemma \ref{L:Bessel functions}. The origin
of these relations, which are fundamentally connected to the properties of Maxwellians, is explained later in this section.

Using the above relations, we can deduce the \emph{local} solvability of any one of the variables $n$, $\Temp$, $\Ent$, $\pressure$, $\rho$ in terms of any two of the others \emph{whenever we know that the necessary partial derivatives are non-zero (knowing this allows us to apply the implicit function theorem)}. In particular, whenever $\left. \frac{\partial \rho}{\partial p} \right|_{\Ent} > 0,$ we can locally solve for $\pressure$ as a function $f_{kinetic}$ of $\Ent$ and $\rho:$ 
\begin{align} \label{E:EOS}
	\pressure = f_{kinetic}(\Ent,\rho).
\end{align}
We refer to $f_{kinetic}$ as the \emph{kinetic equation of state}. 
Note that this equation of state is discussed in Synge \cite{MR0088362}.
For future use, we denote by $\mathfrak{H}$ and $\mathfrak{P}$ respectively the smooth maps from $(n, z)$ to $\Ent, \pressure$ induced by the above relations:
\begin{subequations} 
\begin{align} 
	\Ent & = \mathfrak{H}(n,z) = k_B \ln \bigg\lbrace (4\pi e^4 m_0^3 c^3 h^{-3} n^{-1} 
		\frac{K_2(z)}{z}\exp\Big(z \frac{K_1(z)}{K_2(z)} \Big) \bigg\rbrace, \label{E:changeofstatespacevariablesH} \\
	\pressure & = \mathfrak{P}(n,z) = m_0 c^2 n z^{-1}. \label{E:changeofstatespacevariablesP}
\end{align}
\end{subequations}
In the next section, we sketch a standard proof, which is based on energy estimates, of Theorem \ref{MainThmEuler}, i.e., for local existence for the rE system under the relations \eqref{E:PnTEulerrelationintro} - \eqref{E:nEntzrelationintro}.
During the proof, we work with the unknowns $(\Ent,\pressure,u^{1},u^{2},u^{3}),$ the reason being that a framework for deriving energy estimates in these variables via the method of energy currents has been developed; this is explained in detail during the proof of the theorem. To derive the energy estimates, we of course need an equivalent (for $C^1$ solutions)
formulation the rE system \eqref{E:nuconservationlawintro} in terms of $(\Ent,\pressure,u^{1},u^{2},u^{3}).$ To prove the equivalence of the systems, one needs the following identity, which is shown to be a consequence of the relations \eqref{E:PnTEulerrelationintro} - \eqref{E:nEntzrelationintro} in Proposition \ref{L:Maxwellrelations} below: 
\begin{subequations}
\begin{align}
	\rho + \pressure & = n \left. \frac{\partial \rho}{\partial n} \right|_{\Ent}. \label{E:nrhoprelationship}
\end{align}
\end{subequations}
Using \eqref{E:nandulaw2}, \eqref{E:PnTEulerrelationintro} - \eqref{E:nEntzrelationintro}, and several applications of the chain rule, we deduce the following well-known version of the rE system:
\begin{subequations}
\begin{align}
	u^{\kappa} \partial_{\kappa} \Ent & = 0, && \label{E:rEentropy}
	\\
	u^{\kappa} \partial_{\kappa} \pressure + q \partial_{\kappa} u^{\kappa} & = 0, && \label{E:rEpressure}
	\\
	(\rho + \pressure)u^{\kappa}\partial_{\kappa} u^{\mu} + \Pi^{\mu \kappa} \partial_{\kappa} \pressure & = 0, 	
		&& (\mu = 0,1,2,3), \label{E:rEfourvelocity}
	\\
	\Pi^{\mu \nu} & \eqdef c^{-2}u^{\mu} u^{\nu} + (g^{-1})^{\mu \nu}, && (\mu, \nu = 0,1,2,3), \label{E:rEPi}
	\\
	q & \eqdef	c^2 (\rho + \pressure) \left. \frac{\partial \pressure}{\partial \rho} \right|_{\Ent}.  && \label{E:rEq}
\end{align}
\end{subequations}

Although our local existence theorem is proved using the variables $(\Ent,\pressure,u^1,u^2,u^3),$ in order to construct the relativistic Maxwellian \eqref{E:Maxwelliandef}, which plays a fundamental role in our analysis of the rB equation, we 
require the availability of the variables $(n, \Temp, u^1,u^2,u^3).$ Thus, we need to be able to solve for $(n,z)$ as a smooth function of $(\Ent, \pressure).$ Remarkably, we could not find such a result in the literature. 
Thus, in Lemma \ref{SS:SolveforTemp}, we rigorously show that we can locally solve for $(n,z)$ in terms of $(\Ent,\pressure)$ 
if $0 < z \leq 1/10$ or $z \geq 70.$ Furthermore, based on numerical observations in the region $1/10 < z < 70,$ we make the following conjecture:

\begin{conjecture} \label{C:Tempisgood}
	The map $(n,z) \rightarrow (\mathfrak{H}(n,z),\mathfrak{P}(n,z))$ is auto-diffeomorphism of the region $(0, \infty) \times 
	(0, \infty),$ where the maps $\mathfrak{H}$ and $\mathfrak{P}$ are defined in
	\eqref{E:changeofstatespacevariablesH} - \eqref{E:changeofstatespacevariablesP}.
\end{conjecture}
This conjecture is based on a numerical plot (see Figure \ref{fig:partialppartialz}), in which $\left. z^5 \frac{\partial \pressure}{\partial z} \right|_{\Ent}$ appears to be negative for all $z > 0.$ Examining the proof of Lemma  \ref{SS:SolveforTemp}, it is clear that the negativity of $\left. \frac{\partial \pressure}{\partial z} \right|_{\Ent}$
would imply that the conjecture is true. We remark that by \eqref{E:PderivativeoverPagain}, analytically verifying the negativity of this quantity is equivalent to demonstrating the following inequality for all $z > 0:$
\begin{align} \label{E:pdecreasesinzinequality}
	3 \frac{K_1(z)}{K_2(z)} + z  \Big( \frac{K_1(z)}{K_2(z)} \Big)^2 - z - \frac{4}{z} < 0.
\end{align}

Another fundamental quantity in the system \eqref{E:rEentropy} - \eqref{E:rEq} is the \emph{speed of sound}, the square of which is defined to be $c^2 \left. \frac{\partial \pressure}{\partial \rho} \right|_{\Ent}.$ It is a fundamental thermodynamic assumption that $0 < \left. \frac{\partial \pressure}{\partial \rho} \right|_{\Ent} < 1$ for physically relevant equations of state. As we alluded to in the introduction, the positivity of $\left. \frac{\partial \pressure}{\partial \rho} \right|_{\Ent}$ is required in order for the rE system to be hyperbolic, while the upper bound of $1$ implies that the speed of sound propagation is less than the speed of light. In Remark \ref{R:hyperbolicity}, we explain exactly how we use the positivity in our proof of local existence, while the analytic and geometric connections between the upper bound of $1$ and the speed of sound propagation being less than the speed of light is explained in e.g. \cite{jS2008a}. Remarkably, there seems to be no rigorous proof in the literature that these inequalities hold for the kinetic equation of state \eqref{E:EOS}
in every regime. Consequently, in Lemma \ref{L:speedofsound}, we analytically verify that the equation of state \eqref{E:EOS} exists and satisfies these inequalities in the same regime discussed above, namely for $0 < z \leq 1/10$ and $z \geq 70.$ Moreover, we conjecture that the following stronger statement is true.

\begin{conjecture} \label{C:Speedisreal}
Under the relations \eqref{E:PnTEulerrelationintro} - \eqref{E:nEntzrelationintro}, $\pressure$ can be written as a smooth, positive function of $\Ent,\rho$ on the domain $(0,\infty) \times (0,\infty)$,  i.e., the equation of state \eqref{E:EOS} is well-defined for all $(\Ent,\rho) \in (0,\infty) \times (0,\infty).$ Furthermore, on $(0,\infty) \times (0,\infty),$ we have that
\begin{align} \label{E:speedofsoundbound}
		0 < \left. \frac{\partial \pressure}{\partial \rho} \right|_{\Ent}(\Ent,\rho) \eqdef \left. \frac{\partial 
		f_{kinetic}(\Ent,\rho)}{\partial \rho} \right|_{\Ent} & < \frac{1}{3}. 
\end{align}
\end{conjecture}
Our conjecture is based upon a plot of $(\partial_z|_{\Ent} \rho)/(\partial_z|_{\Ent} \pressure)$ (which can be expressed in terms of a function of $z \eqdef \frac{m_0c^2}{k_B \Temp}$ alone)
that covers the region in question, i.e., the region $ 1/10 < z < 70.$
Our plot
is labeled as Figure \ref{fig:toucan} of Section \ref{SS:hyperbolicity}. Furthermore, we note that according to equation \eqref{E:inversespeedofsoundsquaredBessel}, proving the inequality \eqref{E:speedofsoundbound} is equivalent to proving that the following inequality holds for $z >0:$
\begin{equation} \label{E:weakspeedinequality}
3 <
3 + z \frac{K_1(z)}{K_2(z)}
+  
\frac{
 4 \frac{K_1(z)}{K_2(z)} + z\Big( \frac{K_1(z)}{K_2(z)} \Big)^2 -z
}{ 3 \frac{K_1(z)}{K_2(z)} + z \Big( \frac{K_1(z)}{K_2(z)} \Big)^2 - z - \frac{4}{z} }	
< \infty.
\end{equation}
In this inequality 
$(\partial_z|_{\Ent} \rho)/(\partial_z|_{\Ent} \pressure)
= \Big( \left. \frac{\partial f_{kinetic}}{\partial \rho} \right|_{\Ent} \Big)^{-1}$ is the middle term above expressed as a function of $z$.

Now that we have provided a thorough discussion of the subtleties, we state our:

\begin{quotation}
\textbf{Running Assumption:} 
For the remainder of the article, we restrict our attention to regimes in which any three of the five macroscopic variables $n$, $\Temp$, $\Ent$, $\pressure$, $\rho$ can be written as smooth functions of the remaining two, and in which 
$0 < \left. \frac{\partial \pressure}{\partial \rho} \right|_{\Ent}(\Ent,\rho).$ \textbf{To avoid burdening the paper, we do not point out this assumption every time it is made.} Thus, under these running assumptions, any two of the five variables, together with the spatial components $u^1,u^2,u^3$ of the four-velocity, may be considered to be state-space variables (i.e., they completely determine state of the system), and the phrase ``a solution to the rE system" is a well-defined concept, independent of the choice of state-space variables. The role of Lemmas \ref{L:invertiblemaps} and \ref{L:speedofsound}, then, is to guarantee that there is a large regime in which these conditions hold. Furthermore if Conjectures \ref{C:Tempisgood} and \ref{C:Speedisreal} are correct, then due to the relations \eqref{E:PnTEulerrelationintro} - \eqref{E:nEntzrelationintro}, these conditions are automatically verified.
\end{quotation}

We now explain the connection between solutions to the rE system and solutions to the rB equation. First, given 
any relativistic Maxwellian $\mathcal{M}(n,\Temp,u;\barp)$ as defined in 
\eqref{E:Maxwelliandef}, 
we can define an associated proper number density
$n[\mathcal{M}],$ temperature $\Temp[\mathcal{M}],$
entropy per particle $\Ent[\mathcal{M}],$ pressure $\pressure[\mathcal{M}],$
and proper energy density $\rho[\mathcal{M}].$ The first two quantities are defined to be the first two arguments 
in $\mathcal{M}(n,\Temp,u;\barp),$ while 
precise definitions of the remaining three (as well as equivalent definitions of the first two) are given in Section \ref{SS:Macroscopicquantities}. Furthermore, it can be shown that
\begin{align} \label{E:fluidBoltzmannagreement}
T_{Boltz}^{\mu \nu}[\mathcal{M}]
	& = T_{fluid}^{\mu \nu}, &&
	I_{Boltz}^{\mu }[\mathcal{M}]
	= I_{fluid}^{\mu}.
\end{align}
Equation \eqref{E:fluidBoltzmannagreement} should be interpreted as follows: the tensor obtained by substituting $p = p[\mathcal{M}]$, $\rho = \rho[\mathcal{M}]$, and $u^{\mu} = I^{\mu}[\mathcal{M}]/n[\mathcal{M}]$ into the definition \eqref{E:Tfluiddef} agrees with the tensor calculated using \eqref{E:TBoltzmanndef}. 
The key point is the following: as shown in Proposition \ref{P:Maxwellianidentities}, the macroscopic quantities $n[\mathcal{M}]$, 
$\Temp[\mathcal{M}]$, $\Ent[\mathcal{M}]$, $\pressure[\mathcal{M}]$, $\rho[\mathcal{M}]$ satisfy the relations \eqref{E:PnTEulerrelationintro} - \eqref{E:nEntzrelationintro}. Therefore we must use the same relations to close the rE equations.

Alternatively, if we are given smooth functions $(\Ent,\pressure,u^1,u^2,u^3)$ of $(t,\barx)$ that verify the rE equations \eqref{E:rEentropy} - \eqref{E:rEq} on $[0,T] \times \mathbb{R}^3,$ we can construct a local relativistic Maxwellian $\mathcal{M}(n,\Temp,u;\barp)$ (as in \eqref{E:Maxwelliandef}) \emph{as long as the maps \eqref{E:changeofstatespacevariablesH} - \eqref{E:changeofstatespacevariablesP} are invertible in a neighborhood of $(\Ent,\pressure)([0,T] \times \mathbb{R}^3)$}. Now for any solution of rE system, the identity
\eqref{E:fluidBoltzmannagreement} implies that \eqref{E:particlecurrentconservationBoltzmann} holds for the 
Boltzmann energy-momentum tensor constructed out of $\mathcal{M}.$
Furthermore, as noted in Section \ref{SS:Maxwellians}, the right-hand side of the rB equation vanishes for $\mathcal{M}: \mathcal{C}(\mathcal{M},\mathcal{M}) = 0.$  At this point one may be tempted to conclude that $\mathcal{M}$ itself is a solution to the rB equation. In fact, this is not the case in general; this is exactly the difficulty that Theorem \ref{MainThm} addresses. Its content can roughly be expressed as follows: if $\varepsilon$ is
small enough and positive, then $\mathcal{M}$ is ``close to" a solution $F^\varepsilon$ of the rB equation. We remark that even though $\mathcal{M}$ itself is not in general a solution of \eqref{E:rBintro}, it nevertheless plays a prominent role in our construction of a local solution of \eqref{E:rBintro}. In particular, the validity of the Hilbert expansion of Section \ref{SS:HilbertExpansion} relies heavily on the fact that $\mathcal{M}(n(t,\barx),\Temp(t,\barx),u(t,\barx);\barp)$ is a Maxwellian and that $\big(n(t,\barx), \Temp(t,\barx), u(t,\barx) \big)$ is a solution to the rE system.

\subsection{Statement of main results} \label{SS:Mainresults}

In this section, we state our two independent theorems. The first concerns the existence of local solutions to the rE system 
satisfying the technical assumptions that are used in the proof of our main result. The second is our main result, which provides criteria that ensure that solutions to the rE system are the hydrodynamic limit of solutions to the rB equation. For clarity, we first state our assumptions on the initial data in a separate paragraph.

\begin{EulerInit}
	Let $\mathring{\mathbf{V}} \eqdef (\mathring{n}, \mathring{\Temp}, \mathring{u}^1, \mathring{u}^2, \mathring{u}^3)$ be
	initial data for the relativistic Euler equations under the relations \eqref{E:PnTEulerrelationintro} - 
	\eqref{E:nEntzrelationintro}, and define $\mathring{z} \eqdef \frac{m_0 c^2}{k_B \mathring{\Temp}}.$ 
	Let $\Omega_{n,z} \eqdef (\mathring{n}, \mathring{z})(\mathbb{R}^3) \subset
	(0,\infty) \times (0,\infty)$ be the image of the initial data $(\mathring{n}, \mathring{z}),$ and assume that there exists 
	a compact \textbf{convex}
	set $\Omega'_{n,z} \subset$ $(0,\infty) \times (0,\infty)$ containing
	$\Omega_{n,z}$ in its interior. Let $\Omega'_{\Ent,\pressure} \eqdef (\mathfrak{H},\mathfrak{P})(\Omega'_{n,z}) 
	\subset (0,\infty) \times (0,\infty),$ where the functions $\mathfrak{H}(n,z)$ and $\mathfrak{P}(n,z)$ are defined in 
	\eqref{E:changeofstatespacevariablesH} and \eqref{E:changeofstatespacevariablesP}. Let 
	$\Omega''_{\Ent,\pressure}$ be a compact \textbf{convex} subset of $(0,\infty) \times (0,\infty)$ containing
	$\Omega'_{\Ent, \pressure}.$ Note in particular that our assumptions imply that there exist constants 
	$n_1 > 0$ and $\Temp_1 > 0$ such that
	
	\begin{align*}
		n_1 \leq \inf_{\barx \in \mathbb{R}^3} \mathring{n}(\barx), \\
		\Temp_1 \leq \inf_{\barx \in \mathbb{R}^3} \mathring{\Temp}(\barx).
	\end{align*}
	Assume that the map $(n,z) \rightarrow \big(\mathfrak{H}(n,z), 
	\mathfrak{P}(n,z) \big) \eqdef (\Ent, \pressure)$ has a well-defined, smooth inverse 
	defined on $\Omega''_{\Ent,\pressure}$ that maps 
	$\Omega''_{\Ent,\pressure}$ into a compact subset of $(0,\infty) \times (0,\infty)$ containing
	$\Omega'_{n,z}.$ Let $\rho = m_0 c^2 n 
	\frac{K_1(z)}{K_2(z)} + 3 \pressure$ be as defined in \eqref{E:nrhoprelationintro}, and view $\rho, \Ent,$ and $\pressure$ as 
	functions of $(n,z)$ on $\Omega'_{n,z}.$ Assume that on the set $\big\lbrace \big(\Ent(n,z), \rho(n,z) 
	\big) | \ (n,z) \in \Omega'_{n,z} \big\rbrace,$ $\pressure$ can be written as a smooth function 
	$\pressure = f_{kinetic}(\Ent,\rho)$ and
	that $0 < \frac{\partial f_{kinetic}(\Ent,\rho)}{\partial \rho}.$
\end{EulerInit}

\begin{remark}
As discussed at the end of the introduction, the careful assumptions on the fluid initial data are designed to ensure that they fall within the regime of hyperbolicity, and to ensure the invertibility of the maps between the solution variables $(n, \Temp)$ and $(\Ent, \pressure).$ However, if our two conjectures are correct, then both of these conditions are automatically verified whenever the fluid variables are positive. In any case, Lemmas \ref{L:invertiblemaps} and \ref{L:speedofsound} together show that whenever $\mathring{\Temp}$ is uniformly small and positive, or in whenever $\mathring{\Temp}$ is uniformly large, these conditions hold.
\end{remark}

\begin{remark} \label{R:Convexity}
	The additional convexity assumptions on $\Omega'_{n,z}$ and $\Omega''_{\Ent,\pressure}$ are
	technical conditions that are used in our proof of Theorem \ref{MainThmEuler}, e.g. to conclude 
	\eqref{E:WinitialdatafiniteHNnorm}; see \cite[Proposition B.0.4]{jS2008c} for a discussion on the role of convexity in this 
	context (roughly speaking, convexity is needed so that one can apply the mean value theorem).
\end{remark}

\begin{theorem} \label{MainThmEuler} \textbf{(Local Existence for the rE System)}
Consider initial data that are subject to the restrictions described above.  
	Assume that $N \geq 3$ and that there exist constants $\underline{n} > 0, \ \underline{\Temp} > 0$ such that $(\underline{n}, 
	\frac{m_0 c^2}{k_B \underline{\Temp}}) \in \Omega'_{n,z}$ (defined above), and such that
	\begin{align} \label{E:VdatainHN}
		\| \mathring{\mathbf{V}} - \underline{\mathbf{V}} \|_{H_{\barx}^N} < \infty,
	\end{align}
	where $\underline{\mathbf{V}} \eqdef (\underline{n}, \underline{\Temp}, 0, 0, 0).$
 	Then these data launch a unique classical solution $\mathbf{V} = (n, \Temp, u^1, u^2, u^3)$ 
	to the rE system existing on a nontrivial slab $[0,T] \times \mathbb{R}^3$ upon which
	\begin{align*}
		0 < \inf_{(t,\barx) \in [0,T]\times \mathbb{R}^3} n(t,\barx), \\
		0 < \inf_{(t,\barx) \in [0,T]\times \mathbb{R}^3} \Temp(t,\barx).
	\end{align*}
	$\mathbf{V}$ has the following regularity: $\mathbf{V} - 
	\underline{\mathbf{V}} \in C^{N-2}([0,T] \times \mathbb{R}^3) \bigcap_{k=0}^{N-2} C^k([0,T], H^{N-k}).$ 
	
	If in addition the initial data are such that the inequalities \eqref{juttnerB} below are 
	\emph{strictly} verified at $t=0,$ then there exists a spacetime slab $[0,T'] \times \mathbb{R}^3$ of existence,
	with $0 < T' \leq T,$ upon which the condition \eqref{juttnerB} remains verified.
	
	Finally, if $\| \mathring{\mathbf{V}} - \underline{\mathbf{V}} \|_{H_{\barx}^N} < \delta,$ and
	$\delta$ is sufficiently small, then there exist constants $\Temp_* > 0$ and $C'>0,$
	and a slab of existence $[0,T''] \times \mathbb{R}^3,$
	with $T'' \geq C'/ \delta,$ upon which the following bounds are satisfied:
  \begin{align}		
		\Temp_* < \Temp(t,\barx) & < 2 \Temp_*,  \label{assumption1-T} \\
		|c^{-1} u^j(t,\barx)| & \leq C \delta, \quad (j=1,2,3).  \label{assumption2-T} 
	\end{align}
	As is shown below in Lemma \ref{L:nearconstantstates}, if $\delta$ is sufficiently small, then the bounds 
	\eqref{assumption1-T} - \eqref{assumption2-T} will also imply that the technical conditions 
	\eqref{juttnerB} are verified on $[0,T''] \times \mathbb{R}^3.$
\end{theorem}

\begin{remark}
The conclusions of this theorem regarding the slab $[0,T'] \times \mathbb{R}^3$
follow easily from  the regularity properties of the solution.
\end{remark}

\begin{remark}
There are two main obstacles to extending our principal result, which is Theorem \ref{MainThm} below, to a global-in-time existence result for the rB solution. The first is that local solutions to the rE system tend to form shocks in finite time. In fact, in \cite{dC2007}, Christodoulou showed that there are data arbitrarily close to that of a uniform, quiet fluid state (i.e., $\mathbf{V} \equiv \underline{\mathbf{V}}$) that launch solutions which form shocks in finite time. Since our construction of a local solution to the rB equation relies on the availability of the solution to the rE system, the breakdown of the fluid solution could in principle allow for a breakdown in the Boltzmann solution.

The second obstacle is the possible breakdown of the technical condition \eqref{juttnerB}
satisfied by the fluid solution; this breakdown is sometimes avoidable. More specifically, the condition \eqref{juttnerB}, which plays a key role in the analysis of Section \ref{SS:rBestimates}, 
may break down in finite time even before the shock happens. However, we are aware of a class of data that launch solutions for which \eqref{juttnerB} holds until the time of shock formation. The details are contained in \cite{dC2007}; we offer a quick summary. One considers initial data for the rE system that satisfy the assumptions of Theorem \ref{MainThmEuler}, 
and the following additional assumptions: the data are \emph{irrotational} and \emph{isentropic} (i.e, $\Ent \equiv const$), and they coincide with the constant state $\underline{\mathbf{V}} = (\underline{n},\underline{\Temp},0,0,0)$ outside of the unit sphere centered at the origin of the Cauchy hypersurface $\lbrace t = 0 \rbrace.$ If the departure\footnote{The notion of smallness is measured by a Sobolev norm of suitably high order.}, from constant state is $\leq \delta,$ where $\delta$ is  sufficiently small, then the estimates of \cite[Theorem 13.1]{dC2007} imply the following fact: on the exterior of an outgoing sound cone $\mathcal{C}$ emanating from a sphere of radius $1 - \epsilon$ contained in $\lbrace t = 0 \rbrace,$ where 
$\epsilon = \epsilon(\delta)$ is a sufficiently small positive number satisfying $0< \epsilon \leq 1/2,$ $n$ and $\Temp$ can be continuously extended\footnote{Even though the $L^{\infty}$ norm of the solution remains bounded, the Sobolev norm of the solution blows up as $t \uparrow T_{max}.$} to $[0,T_{\max}] \times \mathbb{R}^3 / \mathcal{C}_{int},$ where $\mathcal{C}_{int}$ denotes the interior of $\mathcal{C},$ and $[0,T_{max})$ is the maximal time interval of classical existence, i.e., the time of first shock formation. Furthermore, the estimates $|n(t,x) - \underline{n}| \leq C \delta,$ 
$|\Temp(t,\barx) - \underline{\Temp}| \leq C \delta,$ and $|c^{-1}u^j(t,\barx)| \leq C \delta,$ $(j=1,2,3),$ hold on the region on $[0,T_{\max}] \times \mathbb{R}^3 / \mathcal{C}_{int}.$  Christodoulou's theorem does not prove that the same estimates hold in $\mathcal{C}_{int},$ but on pg. 6, he remarks that the estimates \emph{do} hold on the interior, and are in fact easier to prove than the exterior estimates. Under these conditions, we may piece together the conclusions from the two regions
$[0,T_{\max}] \times \mathbb{R}^3 / \mathcal{C}_{int}$ and $\mathcal{C}_{int}$ to conclude the following:
for sufficiently small $\delta,$ an inequality of the form \eqref{juttnerB} is satisfied on any slab $[0,T] \times \mathbb{R}^3,$ with $T < T_{max}.$ Consequently, the conclusions of Theorem \ref{MainThm} hold on such a slab.
\end{remark}

Now that we have a large class of suitable solutions to the rE system available, we are ready to state our main theorem. Note that our main theorem is independent of Theorem \ref{MainThmEuler}; the role of Theorem \ref{MainThmEuler} is to ensure that there are fluid solutions that can be used in the hypotheses of Theorem \ref{MainThm}.

\begin{theorem} \label{MainThm}
Let 
$\big(n(t,\barx), \theta(t,\barx), u(t,\barx)\big)$ be a sufficiently regular (see Remark \ref{R:Regularity} below) solution to the relativistic Euler equations \eqref{E:nuconservationlawintro} for $(t,\barx) \in [0,T] \times \mathbb{R}^3_{\barx}.$
Construct the local Maxwellian $\mathcal{M}(n(t,\barx),\Temp(t,\barx),u(t,\barx);\barp)$ as in \eqref{E:Maxwelliandef}. Assume that there exist constants $C>0,$ $\Temp_M> 0,$ and $\alpha \in (1/2,1)$ such that for every 
$
(t,\barx,\barp)
\in 
[0,T] \times \mathbb{R}^3_{\barx} \times \mathbb{R}^3_{\barp},
$
the global Maxwellian $J(\barp) \eqdef e^{-cP^0/(k_B\Temp_M)}$ from \eqref{juttner} verifies the inequalities
\begin{equation}
\frac{J(\barp)}{C}
\le
\mathcal{M}(n(t,\barx),\Temp(t,\barx),u(t,\barx);\barp)
\le
C J^{\alpha}(\barp).
\label{juttnerB}
\end{equation}
Define initially
$$
F^\varepsilon(0,\barx,\barp)
=
\mathcal{M}(0,x,\barp)
+
\sum_{n=1}^6 \varepsilon^n F_n(0,\barx,\barp)
+
\varepsilon^3 \effRE(0,\barx,\barp) 
\ge 0.
$$
Then 
$\exists \varepsilon_0 >0$ 
such that for each 
$ 0< \varepsilon \le \varepsilon_0$ there exists a unique classical solution $F^{\varepsilon}$ of the relativistic Boltzmann equation \eqref{E:rBintro} 
of the form \eqref{hilbertE}
for all $(t, \barx, \barp) \in [0, T ] \times \mathbb{R}^3_{\barx} \times \mathbb{R}^3_{\barp}.$ Furthermore, there exists a constant
$
C_T = C_T (\mathcal{M}, F_1,\ldots, F_6) >0
$
such that for all 
$
\varepsilon \in (0, \varepsilon_0)
$
and for any 
$
\ell \ge 9,
$
the following estimates hold:
\begin{multline*}
\varepsilon^{3/2} 
\sup_{0\le t \le T}
  \left\|    \effRE(t)/  \sqrt{\mathcal{M}(t)}  \right\|_{\infty,\ell}
+
\sup_{0\le t \le T} \left\|   \effRE(t)/  \sqrt{\mathcal{M}(t)}  \right\|_2
\\
\le
C_T
\left\{
\varepsilon^{3/2}  \left\| \effRE(0)/  \sqrt{\mathcal{M}(0)}  \right\|_{\infty,\ell}
+
\left\| \effRE(0)/  \sqrt{\mathcal{M}(0)}  \right\|_2
+
1
\right\}.
\end{multline*}
Recall that $\effRE$ is the remainder from \eqref{hilbertE}.
Moreover, we have that
$$
\sup_{0\le t\le T}\| F^\varepsilon(t) - \mathcal{M}(t) \|_2
+
\sup_{0\le t\le T}\| F^\varepsilon(t) - \mathcal{M}(t) \|_\infty \le C_T \varepsilon,
$$
where the constants $C_T>0$ are independent of $\varepsilon$. 
\end{theorem}

\begin{remark} \label{R:upperLOWER}
Conditions  which would imply \eqref{juttnerB} are standard in the Hilbert expansion literature.  It generally seems to be unclear at the moment how to remove them in the context of classical solutions.  
We choose to make the assumption \eqref{juttnerB} rather than a more stringent assumption of moderate temperature variation in order to make it clear that \eqref{juttnerB} is all that is needed.  The condition \eqref{juttnerB} is used ensure that the 
local Maxwellian $\mathcal{M}$, and the terms in the Hilbert expansion $F_n$, have sufficient momentum decay; as used in Section \ref{SS:rBestimates}.  
\end{remark}

\begin{remark} \label{R:hilbertINITIAL}
We have not specified precisely the initial conditions for the terms $F_1(0)$, $\ldots$, $F_6(0)$, and $\effRE(0)$ in Theorem \ref{MainThm}.
They are constructed  via the Hilbert expansion in Section \ref{SS:HilbertExpansion}, after one has the fluid initial conditions as in Theorem \ref{MainThmEuler}.
\end{remark}

\begin{remark} \label{R:Regularity}
In the hypotheses of Theorem \ref{MainThm}, we have not been specific in stating the regularity needed from the relativistic Euler equations that is required to build the terms $F_1,\ldots, F_6$
in the Hilbert expansion. This is because we are not attempting to optimize the amount of regularity needed in  the Hilbert expansion. In particular, including additional terms in the  expansion would require a smoother solution to the rE system. However, for the purposes of this article, it is certainly sufficient to have 
$\big(n(t,\barx), \Temp(t,\barx), u(t,\barx)\big) \in C^{7}([0,T] \times \mathbb{R}^3_{\barx})$ with suitable $L^2_{\barx}$ integrability.
\end{remark}

In the next lemma, we prove that near-constant, non-vacuum fluid states necessarily verify an inequality  
of the form \eqref{juttnerB}.

\begin{lemma} \label{L:nearconstantstates}
Assume that the initial data for the rE system satisfy the assumptions of Theorem \ref{MainThmEuler}, including the smallness assumption in $\delta,$ and let $\mathcal{M}(n,\Temp,u;\barp)$ be the local Maxwellian \eqref{E:Maxwelliandef} corresponding to the solution. Then if $\delta$ is sufficiently small, there exist a non-trivial slab of the form $[0,C'/\delta] \times \mathbb{R}^3_{\barx},$ and constants $C>0,$ $\Temp_M> 0,$ $\alpha \in (1/2,1),$
such that for every 
$
(t,\barx,\barp)
\in 
[0,C'/\delta] \times \mathbb{R}^3_{\barx} \times \mathbb{R}^3_{\barp},
$
the global Maxwellian $J(\barp) \eqdef e^{-cP^0/(k_B\Temp_M)}$ from \eqref{juttner} 
and $\mathcal{M}(n,\Temp,u;\barp)$ verify the inequalities \eqref{juttnerB}.
\end{lemma}

\begin{proof}
To prove that \eqref{juttnerB} holds, we first recall that by the conclusions of Theorem \ref{MainThmEuler},
there exist a non-trivial slab $[0,C'/\delta] \times \mathbb{R}^3_{\barx},$ and
constants $n_{min} > 0, n_{max} > 0, \Temp_{*} > 0,$ and $C > 0,$ such that
relative to our inertial coordinate system, the solution satisfies the following inequalities 
on $[0,C'/\delta] \times \mathbb{R}^3_{\barx}:$
\begin{align*}
0 < n_{min} \leq n(t,\barx) & < n_{max}, && \\ 
0 < \Temp_* < \Temp(t,\barx) & < 2\Temp_*, && \\
|c^{-1} u^j(t,\barx)| & \leq C \delta, && (j=1,2,3).
\end{align*}
Therefore, using the Cauchy-Schwarz inequality, it is easy to show that there exist dimensionless constants $\alpha \in (1/2,1), C_1 > 0, C_2 > 0,$ depending on the dimensionless constants $m_0^{-1} c^{-2}k_B \Temp_*$ and $\delta,$ such that if $\delta$ is sufficiently small, then for all $(t,\barx,\barp) \in [0,C'/\delta] \times \mathbb{R}_{\barx}^3 \times \mathbb{R}_{\barp}^3$ we have
\begin{align} \label{E:argumentofexpbound}
	m_0^{-1}c^{-1} C_1 P^0 & \leq \Big|\frac{u_{\kappa} P^{\kappa}}{k_B \Temp} \Big| \leq m_0^{-1}c^{-1} C_2 P^0, \\
	C_2 & \leq \alpha^{-1} C_1.
	\notag
\end{align}
We also observe that since $u, P$ are both future-directed and timelike, we have that
$u_{\kappa} P^{\kappa} \leq - |u_{\kappa} u^{\kappa}|^{1/2} |P_{\lambda} P^{\lambda}|^{1/2} = - m_0 c^2.$ In particular, $u_{\kappa} P^{\kappa} < 0.$ It thus follows from definitions \eqref{juttner} and 
\eqref{E:Maxwelliandef}, the uniform bounds on $n(t,\barx)$ and $\Temp(t,\barx)$ (which are \emph{strict} for $\Temp$), 
and inequality \eqref{E:argumentofexpbound}, 
that for any positive constant $\Temp_M$ satisfying
\begin{align} 
\notag
	C_2 \leq m_0 c^2/(k_B \Temp_M) \leq \alpha^{-1} C_1,
\end{align}
there exists a dimensionless constant $C > 1,$ depending on the constants \\
$m_0^{-3}c^{-3}h^{3}n_{min}$, $m_0^{-3}c^{-3}h^{3}n_{max}$,
$m_0^{-1} c^{-2}k_B \Temp_*,$ and $\delta,$ such that inequality \eqref{juttnerB} holds.
\end{proof}

\subsection{Historical background} \label{SS:History}
There have been numerous important contributions to the subject of fluid dynamic limits of the Newtonian Boltzmann equation.  Due to length constraints, it is impossible to give a comprehensive list.  We only point out a brief few works, including the early work of Grad \cite{MR0135535,MR0156656}.  
In the context of DiPerna-Lions \cite{MR1014927} renormalized weak solutions, we mention the fluid limits shown in \cite{MR2025302,arsenioPHD,MR1054287,MR1842343,MR1115587,MR1213991,MR1650310,MR1029125,MasmoudiLevermore,MR2517786}.
We refer to the review articles \cite{MR2182829,MR1911233,MR1942465} for a more comprehensive list of references.

For fluid limits in the context of strong solutions to the non-relativistic Euler and Boltzmann equations, we mention the work of Nishida \cite{MR0503305},  Caflisch \cite{MR586416},  Bardos-Ukai \cite{MR1115292},  Guo \cite{MR2172804,MR2275331}, 
Liu-Yang-Yu \cite{MR2043729}, Guo-Jang-Jiang \cite{MR2472156} and recently  Guo-Jang \cite{2009arXiv0910.5512G}.  In this work we will use the Hilbert expansion approach from Caflisch \cite{MR586416} combined with the recent developments in
Guo-Jang-Jiang \cite{MR2472156} to allow a non-zero initial condition for the remainder: $\effRE(0,\barx,\barp)$.

Other works which connect to and motivate different elements of our estimates/ results include: \cite{MR2209761,MR2366140,MR2100057,strainSOFT,strainNEWT,jS2008a,dC2007,MR2013332}.  They will be discussed in more detail at the appropriate time in the following developments.

Much less is known for the relativistic Boltzmann equation.  Formal fluid limit calculations are shown in the textbooks \cite{sdGsvLwvW1980,MR1898707}.
Linearized hydrodynamics are also studied in \cite{MR1031410}.  Our new contribution is to prove the existence of ``normal solutions" via a Hilbert expansion to the full non-linear relativistic Boltzmann equation.  It would be useful to check if these types of solutions can also be constructed for the relativistic Landau equation, as in for instance \cite{MR2100057}.

\subsection{Outline of the structure of the article}
We now briefly outline the remainder of this article. In Section \ref{S:rE}, we sketch a proof of local existence for the rE system that is based on well-known techniques. In Section \ref{S:rB}, we provide additional background for the relativistic Boltzmann equation. In Section \ref{SS:alternateexpression}, we provide an expression for the Boltzmann collision operator that will be used in the subsequent analysis. In Section \ref{SS:Macroscopicquantities}, we study macroscopic quantities associated to a particle density function $F(t,\barx,\barp),$ and we discuss the corresponding conservation laws that hold whenever $F$ is a solution to the rB equation.  In Section \ref{SS:Maxwellainmacroscopic}, we discuss the thermodynamic relations that hold between the macroscopic quantities associated to the special case $F = \mathcal{M},$ where $\mathcal{M}$ is a Maxwellian. In Section \ref{SS:SolveforTemp}, we discuss the issue of solving for $(n,\Temp)$ in terms of $(\Ent,\pressure)$, 
assuming that the aforementioned thermodynamic relations hold. In Section \ref{SS:hyperbolicity}, we prove that there are regimes of hyperbolicity for the rE system under the same thermodynamic relations. In Section \ref{SS:HilbertExpansion}, we carry out the Hilbert expansion for $F^\varepsilon$ in detail. Finally, in Section \ref{SS:rBestimates}, we provide a proof of our main main theorem.

\section{Local existence for the relativistic Euler equations} \label{S:rE}

In this section we sketch a proof of local existence for the rE system; i.e, we
sketch a proof of Theorem \ref{MainThmEuler} of Section \ref{SS:Mainresults}. During the course of the proof, 
which is mostly along standard lines, we provide references that indicate where one may find the omitted details.\\

\noindent \textbf{Sketch of the proof of Theorem \ref{MainThmEuler}:}
Theorem \ref{MainThmEuler} can be proved using the method of \emph{energy currents}, a technique which was 
first applied to the rE system by Christodoulou in \cite{dC2007}. For complete details, one can first
look in \cite{jS2008a}, which describes in detail how to derive energy estimates for the linearized rE system; we
sketch this derivation below. Once one has these energy estimates, the proof of local existence for the non-relativistic Euler equations given in \cite{aM1984} can be easily modified to handle the rE system. One may also consult \cite{lH1997},
\cite{jSmS1998}, or \cite{jS2008a} for the essential ideas on how to finish the proof once one has energy estimates for the linearized system. Furthermore, although we do not need it for this article, we remark that \cite{jS2008a} contains a proof of continuous dependence on initial data.

In order to apply this method, we use $\mathbf{W} \eqdef (\Ent, \pressure, u^1, u^2, u^3)$ as our unknowns
for the equations \eqref{E:rEentropy} - \eqref{E:rEq}. The assumptions on the data allow us to use the 
functions $\mathfrak{H}$ and $\mathfrak{P}$ from \eqref{E:changeofstatespacevariablesH} - \eqref{E:changeofstatespacevariablesP} to transform the initial data (see Remark \ref{R:Convexity})
$\mathring{\mathbf{V}} \eqdef (\mathring{n}, \mathring{\Temp}, \mathring{u}^1, \mathring{u}^2, \mathring{u}^3)$ to initial data $\mathring{\mathbf{W}} \eqdef (\mathring{\Ent}, \mathring{\pressure}, \mathring{u}^1, \mathring{u}^2, \mathring{u}^3)$ such that
\begin{align} \label{E:WinitialdatafiniteHNnorm}
	\| \mathring{\mathbf{W}} - \underline{\mathbf{W}} \|_{H_{\barx}^N} < \infty,
\end{align}
and such that
\begin{align*}
	0 & < \inf_{\barx \in \mathbb{R}^3}\mathring{\Ent}(\barx), \\
	0 & < \inf_{\barx \in \mathbb{R}^3} \mathring{\pressure}(\barx),
\end{align*}
where $\underline{\mathbf{W}} \eqdef (\underline{\Ent}, \underline{\pressure}, 0, 0, 0).$ Here, $\underline{\Ent}$ and $\underline{\pressure}$ are equal to $\mathfrak{H}(\underline{n}, \underline{\Temp})$ and $\mathfrak{P}(\underline{n}, \underline{\Temp})$ respectively. That \eqref{E:WinitialdatafiniteHNnorm} follows from \eqref{E:VdatainHN} can be shown via Sobolev-Moser type estimates; see e.g. the appendix of \cite{jS2008a}.

Typical proofs of local existence are based on either the construction of convergent sequence of iterates, or a contraction mapping argument. Both of these arguments require that one prove energy estimates for the linearized Euler equations, which are the following system:
\begin{subequations}
\begin{align}
	\widetilde{u}^\mu \partial_\mu \dot{{\Ent}} &= \mathfrak{F},  &&                              \label{E:EOV1} \\
 	\widetilde{u}^\mu \partial_\mu \dot{p} +
 	\widetilde{q}\frac{\widetilde{u}_k}{\tilde{u}^0}
  	\partial_0\dot{u}^k+ \widetilde{q} \partial_k \dot{u}^k  &= \mathfrak{G}, &&         \label{E:EOV2} \\
   	(\widetilde{\rho} + \widetilde{p})\widetilde{u}^\mu \partial_\mu
  	\dot{u}^j + \widetilde{\Pi}^{\mu j} \partial_\mu \dot{p} &= \mathfrak{H}^j, &&
   	(j =1,2,3).                                                 \label{E:EOV3} 
\end{align}
\end{subequations}
In the above equations, $\widetilde{\mathbf{W}} \eqdef (\widetilde{\Ent}, \widetilde{\pressure}, \widetilde{u}^1, \widetilde{u}^2, \widetilde{u}^3)$ is the ``background" (which can be thought of as the previous iterate in an iteration scheme), $\dot{\mathbf{W}} \eqdef (\dot{\Ent}, \dot{\pressure}, \dot{u}^1, \dot{u}^2, \dot{u}^3)$ is the \emph{variation} (which can be thought of as either the next iterate or one of its spatial derivatives), and the terms $\mathbf{b} \eqdef (\mathfrak{F},\mathfrak{G},\mathfrak{H}^1,\mathfrak{H}^2,\mathfrak{H}^3)$ are the inhomogeneous terms that arise from the iteration + differentiation procedure. Furthermore, $\widetilde{u}^0 = \sqrt{c^2 + \sum_{j=1}^3 (\widetilde{u}^j)^2},$ $\widetilde{\Pi}^{\mu \nu} \eqdef \widetilde{u}^{\mu} \widetilde{u}^{\nu} + (g^{-1})^{\mu \nu},$ and $\widetilde{q}$ is the function of the state-space variables from \eqref{E:rEq} evaluated at the background.

To deduce energy estimates for the linearized systems, one first defines energy currents $\dot{\mathscr{J}},$ which are vectorfields
that depend quadratically on the variations $\dot{\mathbf{W}}$:
\begin{gather}                                                                                       
\dot{\mathscr{J}}^0(\dot{\mathbf{W}},\dot{\mathbf{W}})  \eqdef \widetilde{u}^0 \dot{{\Ent}}^2   +
                    \frac{\widetilde{u}^0}{\widetilde{q}}\dot{\pressure}^2 +
                    2\frac{\widetilde{u}_k\dot{u}^k}{\widetilde{u}^0} \dot{\pressure}
                 + (\widetilde{\rho} + \widetilde{\pressure})\widetilde{u}^0 \Big[\dot{u}^k
                    \dot{u}_k - \frac{(\widetilde{u}_k
                    \dot{u}^k)^2}{(\widetilde{u}^0)^2} \Big], \label{E:EnergyCurrent}  
                    \\
\dot{\mathscr{J}}^j(\dot{\mathbf{W}},\dot{\mathbf{W}})  \eqdef \widetilde{u}^j \dot{{\Ent}}^2  +
                    \frac{\widetilde{u}^j}{\widetilde{q}} \dot{\pressure}^2 + 2 \dot{u}^j \dot{\pressure}
                 + (\widetilde{\rho} + \widetilde{\pressure})\widetilde{u}^j \Big[\dot{u}^k
                    \dot{u}_k - \frac{(\widetilde{u}_k \dot{u}^k)^2}{(\widetilde{u}^0)^2} \Big], \quad (j=1,2,3). \notag
\end{gather}
The two key properties of the above energy currents, which are discussed in detail in \cite{jS2008a},
are \textbf{(i)} it can be shown that $\dot{\mathscr{J}}^0(\dot{\mathbf{W}},\dot{\mathbf{W}})$ is a positive definite quadratic form in $\dot{\mathbf{W}},$ and \textbf{(ii)} if $\dot{\mathbf{W}}$ is a solution to the linearized rE system \eqref{E:EOV1} - \eqref{E:EOV3}, then it can be shown that $\partial_{\kappa}[\dot{\mathscr{J}}^{\kappa}(\dot{\mathbf{W}},\dot{\mathbf{W}})]$ does not depend on the derivatives of $\dot{\mathbf{W}}.$ More specifically, the following formula holds:
        \begin{multline}
 \label{E:EnergyCurrentDivergence}
                \partial_\mu \dot{\mathscr{J}}^\mu  =
                    (\partial_\mu \widetilde{u}^\mu) \dot{{\Ent}}^2 +
                    \partial_\mu \left(\frac{\widetilde{u}^\mu}{\widetilde{q}}\right) \dot{\pressure}^2
                   +  2 \partial_0 \left(\frac{\widetilde{u}_k}{\widetilde{u}^0}\right) \dot{u}^k \dot{\pressure}                
                   \\
                  + \partial_\mu[(\widetilde{\rho} + \widetilde{\pressure})\widetilde{u}^\mu]
                    \left[\dot{u}^k \dot{u}_k - \frac{(\widetilde{u}_k
                    \dot{u}^k)^2}{(\widetilde{u}^0)^2} \right]
                    - 2\widetilde{u}_k \dot{u}^k(\widetilde{\rho}
                    + \widetilde{\pressure})\left(\frac{\widetilde{u}^\mu}{\widetilde{u}^0}\right)
                    \partial_\mu \left(\frac{\widetilde{u}_j}{\widetilde{u}^0}\right)\dot{u}^j                  
                    \\
                   + 2 \dot{{\Ent}}\mathfrak{F} + 2\frac{\dot{\pressure}\mathfrak{G}}{\widetilde{q}} + 2
                    \dot{u}_k \mathfrak{H}^{k} - 2 \frac{\widetilde{u}_j \mathfrak{H}^j \widetilde{u}_k
                    \dot{u}^k}{(\widetilde{u}^0)^2},
\end{multline}
where $(\mathfrak{F},\mathfrak{G},\mathfrak{H}^1,\mathfrak{H}^2,\mathfrak{H}^3)$ are the inhomogeneous terms in \eqref{E:EOV1}- \eqref{E:EOV3}.

We therefore can define an  
energy $\mathscr{E}(t) \geq 0$ by
\begin{align} \label{E:Edef}
		\mathscr{E}^2(t) \eqdef 
		\int_{\mathbb{R}^3} 
			\dot{\mathscr{J}}^0 (\dot{\mathbf{W}}, \dot{\mathbf{W}} ) \, d \barx,
\end{align}
and $\mathscr{E}$ can be used to control $\| \dot{\mathbf{W}} \|_{L^2}.$ More specifically, property \textbf{(i)} can be used to show that if $\widetilde{q}, \widetilde{p}, \widetilde{\rho}$ are uniformly bounded from above and away from $0$ on $[0,t]\times \mathbb{R}^3,$ and the $|\widetilde{u}^j|$ are bounded from above on $[0,t] \times \mathbb{R}^3,$ then there exists a constant $C > 0$ depending only on the values of $\widetilde{\mathbf{W}},$ $\| \partial_t \widetilde{\mathbf{W}} \|_{L^{\infty}},$ and $\| \partial_{\barx} \widetilde{\mathbf{W}} \|_{L^{\infty}},$ such that on $[0,t],$ we have 
\begin{align} \label{E:EHNequivalent}
	C^{-1} \| \dot{\mathbf{W}} \|_{L^2} \leq \mathscr{E} \leq C \| \dot{\mathbf{W}} \|_{L^2}.
\end{align}

\begin{remark} \label{R:hyperbolicity}
	If $\widetilde{q} < 0,$ then $\dot{\mathscr{J}}^0$ is no longer positive definite, and
	inequality \eqref{E:EHNequivalent} fails. Since $\widetilde{q} = c^2 \left. \frac{\partial \widetilde{\pressure}}{\partial 
	\widetilde{\rho}} \right|_{\Ent} (\widetilde{\rho} + \widetilde{\pressure}),$ it is clear that the non-negativity of 
	$\left. \frac{\partial \widetilde{\pressure}}{\partial \widetilde{\rho}} \right|_{\Ent}$
	plays a fundamental role in the well-posedness of the rE system. This explains the significance of
	Lemma \ref{L:speedofsound} and Conjecture \ref{C:Speedisreal}. Furthermore, we note that the conditions on the initial data 
	guarantee that $\mathring{q}(\barx) \eqdef q(0,\barx)$ is uniformly positive. 
\end{remark}

Furthermore, the Cauchy-Schwarz inequality for integrals, \eqref{E:EnergyCurrentDivergence}, and \eqref{E:EHNequivalent} imply that
\begin{align} \label{E:SobolevMoserdivergence}
	 \| \partial_{\kappa} [\dot{\mathscr{J}}^{\kappa} (\dot{\mathbf{W}}, \dot{\mathbf{W}})] \|_{L^1}
		& \leq C \| \dot{\mathbf{W}} \|_{L^2} \| \mathbf{b} \|_{L^2} \leq C \mathscr{E} \| \mathbf{b} \|_{L^2}.
\end{align}
Using \eqref{E:Edef}, the divergence theorem (assuming suitable fall-off conditions at infinity), and \eqref{E:SobolevMoserdivergence}, we have that

\begin{align*} 
	\frac{d}{dt} \big(\mathscr{E}^2 \big) = c \frac{d}{dx^0} \big(\mathscr{E}^2 \big)
		& = c \int_{\mathbb{R}^3} \partial_{\kappa} [\dot{\mathscr{J}}^{\kappa} (\dot{\mathbf{W}}, \dot{\mathbf{W}})]  \, d \barx 
		\\ 
	& \leq C \mathscr{E} \| \mathbf{b} \|_{L^2},
\end{align*}
which leads to the following energy estimate for the linearized system:
\begin{align} \label{E:linearizedenergyestimate}
	\frac{d}{dt} \mathscr{E}(t) & \leq C \| \mathbf{b} \|_{L^2}.
\end{align}
The availability of \eqref{E:linearizedenergyestimate} is the fundamental reason that the rE system has local-in-time solutions
belonging to a Sobolev space. This concludes our abbreviated discussion of the existence aspect of Theorem \ref{MainThmEuler}
in terms of $\mathbf{W};$ we comment on the variables $\mathbf{V} \eqdef (n, \Temp, u^1, u^2, u^3)$ near the end of the proof.

Let us now assume that we have a local solution $\mathbf{W}$ to the nonlinear rE equations near the constant state $\underline{\mathbf{W}};$ we will make a few remarks about the time of existence. In the nonlinear case, one can define energies (with $N \geq 3$ so that various Sobolev-Moser type inequalities are valid)
\begin{align}\notag
	\mathscr{E}_N^2(t) \eqdef 
	\sum_{0 \leq |\vec{\alpha}| \leq N} \int_{\mathbb{R}^3}  \dot{\mathscr{J}}^0 \big(
	\partial_{\vec{\alpha}} (\mathbf{W} - \underline{\mathbf{W}}),\partial_{\vec{\alpha}} (\mathbf{W} - \underline{\mathbf{W}})
	\big) \, d \barx.
\end{align}
Furthermore, all of the inhomogeneities that arise upon differentiating the rE equations are of quadratic order or higher. Consequently, near a constant state, the inhomogeneous terms analogous to the term $\| \mathbf{b} \|_{L^2}$ on the right-hand side of \eqref{E:linearizedenergyestimate} can be bounded by $C \mathscr{E}_N^2(t),$ and the resulting energy estimate is
\begin{align} \label{E:NOTlinearizedenergyestimate}
	\frac{d}{dt} \mathscr{E}_N(t) \leq C \mathscr{E}_N^2(t).
\end{align}
Consequently, we may apply Gronwall's inequality to \eqref{E:NOTlinearizedenergyestimate} and 
use property \eqref{E:EHNequivalent} to deduce an a-priori estimate of the form
\begin{align} \label{E:WHNbound}
	\| \mathbf{W} - \underline{\mathbf{W}} \|_{H_{\barx}^N} & \leq 
	C \frac{\| \mathring{\mathbf{W}} - \underline{\mathbf{W}} \|_{H_{\barx}^N}}{1 - C t\| \mathring{\mathbf{W}} - 
		\underline{\mathbf{W}} \|_{H_{\barx}^N}}.
\end{align}
It follows from \eqref{E:WHNbound} that the time of existence is at least of size $\widetilde{C} /  \delta$.
This fact is a consequence of a standard continuation principle that is available for hyperbolic PDEs (consult e.g. \cite[Chapter 6]{lH1997} or \cite{jS2008b} for the essential ideas), which implies that the solution exists as long as the a-priori energy estimates lead to the conclusion that $\| \mathbf{W} - \underline{\mathbf{W}} \|_{H_{\barx}^N}$ is sufficiently small.

We remark that in the formula \eqref{E:WHNbound}, $C$ is a numerical constant that depends on an a-priori assumption concerning the subset of $(0,\infty) \times (0,\infty)$ to which the pair $(\Ent,\pressure)$ belongs, and on $B,$ where $B > 0$ is a fixed a-priori upper bound for $\| \mathbf{W} - \underline{\mathbf{W}} \|_{H_{\barx}^N}.$ These a-priori assumptions imply 
that the solution escapes neither the regime of hyperbolicity, nor a convex subset (see Remark \ref{R:Convexity}) of the regime in which the map $(n,z) \rightarrow \big(\mathfrak{H}(n,z), \mathfrak{P}(n,z) \big) =(\Ent,\pressure)$ is invertible; on the time interval $[0,\widetilde{C} /  \delta],$ these a-priori assumptions can be shown to hold through a bootstrap argument. Furthermore, once we have shown \eqref{E:WHNbound}, this inequality can be translated into an inequality for the original variables $\mathbf{V};$ Sobolev-Moser type estimates allow us to estimate (here we again use the convexity assumption discussed in Remark \ref{R:Convexity}) $\| \mathbf{V} - \underline{\mathbf{V}} \|_{H_{\barx}^N} \leq C \| \mathbf{W} - \underline{\mathbf{W}} \|_{H_{\barx}^N}.$ This is possible because $\mathbf{V}$ can be written as a smooth function of $\mathbf{W}$ whenever $(\Ent,\pressure)$ belongs to the domain of the inverse of the aforementioned map.

\section{The relativistic Boltzmann equation} \label{S:rB}

In this section, we provide some additional background material on the relativistic Boltzmann equation. The rB equation is often expressed in the physics literature (see e.g. \cite{MR1026740,sdGsvLwvW1980}) in the Lorentz-invariant form of \eqref{E:rBintro}
where the \emph{collision kernel} $\mathcal{C}$ is defined by
\begin{equation}
\mathcal{C}(F,G)
=
 \frac {c}{2}\int_{\mathbb{R}^3} \frac{d\barq}{Q^0} \int_{\mathbb{R}^3}\frac{d\barq'}
 {Q'^0}\int_{\mathbb{R}^3}\frac{d\barp'}{P'^0}W(\barp, \barq | \barp', \barq') [F(\barp')G(\barq')-F(\barp)G(\barq)].
\notag
\end{equation}
In our analysis below, it will often be convenient to divide both sides of \eqref{E:rBintro} by $P^0.$ We therefore introduce the \emph{normalized velocity} $\phat,$ defined by
\begin{align}  
	\phat \eqdef \frac{c}{P^0} \barp,
	\notag
\end{align}
and 
$
\mathcal{Q}(\cdot,\cdot) \eqdef \frac{c}{P^0} \mathcal{C}(\cdot,\cdot).
$
Then \eqref{E:rBintro} is equivalent to
\begin{align} \label{E:rBequivalent}
	\partial_t F + \phat \cdot \partial_{\barx} F  = \frac{1}{\varepsilon} \mathcal{Q}(F,F),
\end{align}
where $\phat \cdot \partial_{\barx} \eqdef \phat^1 \partial_1 + \phat^2 \partial_2 + \phat^3 \partial_3.$

On the right-hand side of the above expression for $\mathcal{C},$ the variables $P$ and $Q$ represent the \emph{pre-collisional} four-momenta of a pair of particles, while $P'$ and $Q'$ represent their \emph{post-collisional} four-momenta. We assume that the particle collisions are elastic, in which case the conservation of energy\footnote{In our inertial coordinate system, the ``energy" of a four-vector is $c$ times its $0$ component, while its 3-momentum comprises its final $3$ components.} and 3-momentum can be expressed as
\begin{align}
P^\mu+Q^\mu=P'^{\mu}+ Q'^{\mu}, && (\mu = 0,1,2,3).
\label{collisionalCONSERVATION}
\end{align}
The \emph{transition rate} $W(\barp, \barq | \barp', \barq'),$ which is a Lorentz scalar, can be expressed as 
\begin{equation} \label{transitionrate}
W(\barp, \barq | \barp', \barq') 
=
s\sigma(\varrho, \vartheta) \delta^{(4)}(P^\mu+Q^\mu-P'^{\mu}- Q'^{\mu}),
\end{equation} 
where $\delta^{(4)}$ is a Dirac delta function, and $\sigma(\varrho, \vartheta)$ is the \emph{differential cross-section} or \emph{scattering kernel}.  The other quantities are defined in \eqref{gDEFINITION},
\eqref{angle}, and \eqref{sDEFINITION}.

We have several remarks to make. First, the Lorentz invariance of the left-hand side of \eqref{E:rBintro} is manifest, while the Lorentz invariance of the right-hand side follows from that of $\frac{d\barp'}{P'^0}$ (and similarly for $\barq, \barq'$); see \eqref{E:measureonPx}. Next, we remark that $\vartheta$ in \eqref{angle} is well-defined under \eqref{collisionalCONSERVATION}, but that it may not be in general \cite{MR1379589}; i.e., in general, the right-hand side of \eqref{angle} may be larger than $1$ in magnitude. Finally, we point out that both the left and right-hand sides of \eqref{E:rBequivalent} are functions of $(t,\barx,\barp).$

\subsection{Expression for the collision operator} \label{SS:alternateexpression}

In this section, we provide an alternate expression \eqref{collisionCM} for the collision operator, which is derived from carrying out certain integrations in a center-of-momentum frame. This is the expression for the collision operator that we will use in our analysis for the remainder of the article. We remark that yet another expression for the collision operator was derived in \cite{MR1211782}; see \cite{strainNEWT,strainCOOR} for an explanation of the connection between the expression from \cite{MR1211782} and the one in \eqref{collisionCM} below.

One may use Lorentz transformations as described in \cite{sdGsvLwvW1980} and in some detail in \cite{strainPHD,strainCOOR} to reduce the delta functions in \eqref{transitionrate}, thereby obtaining
\begin{equation}
\mathcal{Q}(F,G)
=
\int_{\mathbb{R}^3}  d\barq ~
\int_{\mathbb{S}^{2}} d\omega ~
~ v_{\o} ~ \sigma (\varrho,\vartheta ) ~ [F(\barp')G(\barq')-F(\barp)G(\barq)],
\label{collisionCM}
\end{equation}
where $v_{\o}=v_{\o}(\barp,\barq),$ the M\o ller velocity, is given by
\eqref{moller}.

The post-collisional 3-momenta in the expression (\ref{collisionCM}) can be written as follows:
\begin{equation}
\begin{split}
\barp'&=\frac{\barp+\barq}{2}+\frac{\varrho}{2}\left(\omega+(\gamma-1)(\barp+\barq)\frac{(\barp+\barq)\cdot \omega}{|\barp+\barq|^2}\right), 
\\
\barq^\prime&=\frac{\barp+\barq}{2}-\frac{\varrho}{2}\left(\omega+(\gamma-1)(\barp+\barq)\frac{(\barp+\barq)\cdot \omega}{|\barp+\barq|^2}\right),
\label{postCOLLvelCMsec2}
\end{split}
\end{equation}where $\gamma =(P^{0} + Q^{0})/\sqrt{s}$, and $\cdot$ is the ordinary Euclidean dot product in $\mathbb{R}^3.$ 
The energies can be expressed as
\begin{equation}
\begin{split}
P'^0&=\frac{P^{0}+Q^{0}}{2}+\frac{\varrho}{2\sqrt{s}}\omega\cdot (\barp+\barq), 
\\
Q'^0&=\frac{P^{0}+Q^{0}}{2}-\frac{\varrho}{2\sqrt{s}}\omega\cdot (\barp+\barq).
\end{split}
\notag
\end{equation}
It is clear that $P,Q,P',Q'$ satisfy \eqref{collisionalCONSERVATION}. Additionally, it is explained in 
\cite{strainPHD}
that after carrying out the integrations in a center-of-momentum frame, 
the scattering angle becomes a function of $\barp, \barq,$ and satisfies
$
\cos\vartheta = k\cdot \omega,
$
where $\omega \in \mathbb{S}^2 \subset \mathbb{R}^3,$ and $k = k(\barp,\barq) \in \mathbb{S}^2 \subset \mathbb{R}^3.$
We remark that an expression for $k(\barp,\barq)$ is given in \cite[Equation (5.37)]{strainPHD} and \cite{strainCOOR}, but that its precise form is not needed here.

\subsection{Macroscopic quantities and conservation laws for rB}
\label{SS:Macroscopicquantities}
In this section, we study macroscopic quantities associated to a particle density function
$F(t,\barx,\barp).$ We also discuss conservation laws that hold whenever $F(t,\barx,\barp)$ is a solution to the relativistic Boltzmann equation \eqref{E:rBintro}. This material is quite standard \cite{sdGsvLwvW1980}, but we include it for convenience. We begin by recalling that the energy-momentum tensor $T_{Boltz}^{\mu \nu}[F]$ and the particle current 
$I_{Boltz}^{\mu}[F]$ for the relativistic Boltzmann equation 
are defined in \eqref{E:TBoltzmanndef}.
Note that the term $\frac{d \barp}{P^{0}}$ on the right-hand side of \eqref{E:TBoltzmanndef} is exactly
the canonical measure defined in \eqref{E:gbarvolumeform}.

We now define the following macroscopic quantities:
\begin{subequations}
\begin{align}
	S^{\mu} & \eqdef - k_B c \int_{\mathbb{R}^3} P^{\mu} F(t,\barx,\barp)  
		\Big \lbrace \mbox{ln}[h^3F(t,\barx,\barp)] - 1 \Big \rbrace \frac{d \barp}{P^{0}},
		\label{E:EntfourflowBoltzmann} 
	\\
	-c^2 n^2 & \eqdef g_{\kappa \lambda}I_{Boltz}^{\kappa} I_{Boltz}^{\lambda},
	\\
	u^{\mu} & \eqdef n^{-1} I_{Boltz}^{\mu}, 
	\\
	u_{\kappa} u^{\kappa} & = -c^2, 
	\\
	\rho & \eqdef c^{-2} u_{\kappa} u_{\lambda} T_{Boltz}^{\kappa \lambda}, 
	\label{E:rhoBoltzmanndef} 
	\\
	p^{\mu \nu} & \eqdef \Pi_{~\kappa}^{\mu} T_{Boltz}^{\kappa \lambda} \Pi_{~\lambda}^{\nu},
	\\
	\Ent & \eqdef - c^{-2} n^{-1} u_{\kappa} S^{\kappa}. 
	\label{E:EntBoltzmann} 
\end{align}
\end{subequations}
Here, $S$ is the \emph{entropy four flow}, 
$n$ is the \emph{proper number density},
$u$ is the \emph{four-velocity}, $\rho$ is the \emph{proper energy density}, $p^{\mu \nu}$ is the
\emph{pressure tensor}, $\Ent$ is the \emph{entropy per particle}, and $\Pi^{\mu \nu} \eqdef c^{-2} u^{\mu} u^{\nu} + (g^{-1})^{\mu \nu}$ is as defined in \eqref{E:Pidef}.

We also decompose the pressure tensor by defining the \emph{pressure} $\pressure$ and the
\emph{viscous pressure tensor} $v^{\mu \nu}:$
\begin{align}
	p^{\mu \nu} & \eqdef \pressure \Pi^{\mu \nu} + v^{\mu \nu},
	\notag
\end{align}
where
\begin{align} 
\label{E:Pressuredef}
	\pressure & \eqdef \frac{1}{3} \Pi_{\kappa \lambda} p^{\kappa \lambda} = \frac{1}{3} \Pi_{\kappa \lambda} T_{Boltz}^{\kappa \lambda}.
\end{align}
We remark that $\Pi^{\kappa \lambda} v_{\kappa \lambda} = 0.$

The following lemma will be useful in verifying numerous identities, especially those of
Proposition \ref{P:Maxwellianidentities}.
\begin{lemma} \label{L:TransformationProperties}
	Let $\Lambda_{\ \nu}^{\mu}$ be a matrix (of real numbers) such that 
	$$
		\Lambda_{\ \mu}^{\kappa} \Lambda_{\ \nu}^{\lambda} g_{\kappa \lambda} = g_{\mu \nu}, \qquad (\mu, \nu = 0,1,2,3),
	$$
	and such that $\mbox{det}(\Lambda) = 1$ (i.e. 
	a proper Lorentz transformation).   Consider the change of coordinates $P \leftarrow \widetilde{P}$ 
	on $T_x M$ from one inertial coordinate system to another induced by $\Lambda:$ $P^{\mu} \eqdef \Lambda_{~\kappa}^{\mu} 
	\widetilde{P}^{\kappa}.$ Then under this change of coordinates, $I^{\mu}, u^{\mu},$ and $S^{\mu}$ transform as the components 
	of a vector, while $T_{Boltz}^{\mu \nu}$ and $p^{\mu \nu}$ transform as the components of a two-contravariant tensor. In 
	an arbitrary coordinate system, these statements 
take the following form:
	\begin{subequations}
	\begin{align}
		I^{\mu} & = \Lambda_{\ \kappa}^{\mu} \widetilde{I}^{\kappa}, && (\mu = 0,1,2,3),
		\label{E:Itransform} 
		\\
		u^{\mu} & = \Lambda_{\ \kappa}^{\mu} \widetilde{u}^{\kappa}, && (\mu = 0,1,2,3),
		\\
		S^{\mu} & = \Lambda_{\ \kappa}^{\mu} \widetilde{S}^{\kappa}, && (\mu = 0,1,2,3),
		\\
		T_{Boltz}^{\mu \nu} & = \Lambda_{\ \kappa}^{\mu} \Lambda_{\ \lambda}^{\nu} \widetilde{T}_{Boltz}^{\kappa \lambda}, 
			&& (\mu, \nu = 0,1,2,3),
		\\
		p^{\mu \nu} & =\Lambda_{\ \kappa}^{\mu} \Lambda_{\ \lambda}^{\nu} \widetilde{p}^{\kappa \lambda}, 
			&& (\mu,\nu = 0,1,2,3).
	\end{align}
	\end{subequations}
\end{lemma}
\begin{proof}
	We give only the proof of \eqref{E:Itransform}; the remaining statements follow similarly. Equation \eqref{E:Itransform} is equivalent to the following statement, where we abbreviate $F(t,\barx,\barp) = F(\barp)$
\begin{align}
\label{E:Itransformequivalent}
\int_{\mathbb{R}^3} P^{\mu} F(P^1,P^2,P^3) \frac{d \bar{P}}{P^{0}} =
         \int_{\bar{\widetilde{P}} \in \mathbb{R}^3} \Lambda_{\ \kappa}^{\mu} \widetilde{P}^{\kappa}
         F( \Lambda_{\ \alpha}^1 \widetilde{P}^\alpha, \Lambda_{\ \beta}^2 \widetilde{P}^\beta,
   			\Lambda_{\ \gamma}^3 \widetilde{P}^\gamma) \frac{d \bar{\widetilde{P}}}{\widetilde{P}^0}.
\end{align}
Recall that Lorentz transformations preserve the form of the metric $g,$ so that
$g_{\mu \nu} = \widetilde{g}_{\mu \nu} = \mbox{diag}(-1,1,1,1),$ which implies in particular
that $P^{0} = \sqrt{m_0^2 c^2 + |\barp|^2}$ and $\widetilde{P}^0 = \sqrt{m_0^2c^2 + |\bar{\widetilde{P}}|^2}.$ It thus follows from the discussion in Section \ref{SS:massshell}, and in particular equation \eqref{E:measureonPx}, that $\frac{d \barp}{P^{0}} = \frac{d \bar{\widetilde{P}}}{\widetilde{P}^0}.$ Equation \eqref{E:Itransformequivalent} now follows from the standard change of variables formula for integration.
\end{proof}

The conservation laws for the relativistic Boltzmann equation are given in the following lemma.
\begin{lemma} \label{L:Boltzmannconservation}
\cite[Chapter 2]{MR1898707}
	Let $F(t,\barx,\barp)$ be a $C^1$ solution to \eqref{E:rBintro}. Then
	the conservation laws \eqref{E:particlecurrentconservationBoltzmann} hold.
\end{lemma}

	Note that  \eqref{E:particlecurrentconservationBoltzmann}
	corresponds to the fluid conservation laws \eqref{E:nuconservationlawintro}.

\subsection{Macroscopic quantities for a Maxwellian $\mathcal{M}= \mathcal{M}(n,\Temp,u;\barp)$} 
\label{SS:Maxwellainmacroscopic} 

In this section, we prove the following proposition:

\begin{proposition} \label{P:Maxwellianidentities}
	Let $n(t,\barx) > 0$, $\Temp(t,\barx) > 0$ be positive functions on $M,$ and let $u(t,\barx)$ 
	be a future-directed vectorfield on $M$ satisfying
	\eqref{E:unormalized}. Consider the corresponding local relativistic Maxwellian 
	$\mathcal{M}(n,\Temp,u;\barp)$ defined in \eqref{E:Maxwelliandef} with \eqref{zdef}.
	Then the following relations hold for the quantities defined in \eqref{E:EntfourflowBoltzmann} - \eqref{E:Pressuredef}:
	\begin{subequations}
	\begin{align}
		I^{\mu}[\mathcal{M}] & = n[\mathcal{M}] u^{\mu}, \quad (\mu = 0,1,2,3),
		 \label{E:InuMaxwellianidentity} 
		 \\
		\pressure [\mathcal{M}] & = k_B n[\mathcal{M}] \Temp
			= m_0 c^2 \frac{n[\mathcal{M}]}{z[\mathcal{M}]}, \quad \label{E:PnTequationofstateMaxwellian} 
			\\
		\rho[\mathcal{M}] &= 3 \pressure[\mathcal{M}] + m_0c^2n[\mathcal{M}] \frac{K_1(z[\mathcal{M}])}{K_2(z[\mathcal{M}])}   \label{E:rhoMaxwellian} \\
		& = m_0c^2n[\mathcal{M}] \frac{K_3(z[\mathcal{M}])}{K_2(z[\mathcal{M}])} - 
			\overbrace{k_B n[\mathcal{M}] \Temp}^{\pressure[\mathcal{M}]},  \nonumber
			\\
		T_{Boltz}^{\mu \nu}[\mathcal{M}] & = T_{fluid}^{\mu \nu}[\mathcal{M}], 
			\quad (\mu, \nu = 0,1,2,3), \label{E:TBoltzMaxwellian} 
			\\
		n[\mathcal{M}] & = 4\pi e^4 m_0^3 c^3 h^{-3} \exp\Big(\frac{-\Ent}{k_B} \Big) 
			\frac{K_2(z[\mathcal{M}])}{z[\mathcal{M}]}\exp\Big(z[\mathcal{M}] 
			\frac{K_1(z[\mathcal{M}])}{K_2(z[\mathcal{M}])} \Big)  \label{E:nzEntrelation} \\
			& = 4\pi m_0^3 c^3 h^{-3} \exp\Big(\frac{-\Ent}{k_B} \Big) 
				\frac{K_2(z[\mathcal{M}])}{z[\mathcal{M}]}
				\exp\Big(z[\mathcal{M}] \frac{K_3(z[\mathcal{M}])}{K_2(z[\mathcal{M}])} \Big). \notag 
	\end{align}
	\end{subequations}
\end{proposition}

By $T_{fluid}^{\mu \nu}[\mathcal{M}],$ we mean the energy-momentum tensor that results from inserting $\pressure [\mathcal{M}]$ and $\rho [\mathcal{M}]$ into the expression on the right-hand side of \eqref{E:Tfluiddef}.

\begin{proof}
	It is  well-known  that since $u(x)$ is a future-directed vector satisfying \eqref{E:unormalized}, there exists 
	an inertial coordinate system in which $(u^0,u^1,u^2,u^3) = (c,0,0,0)$ (at the spacetime point $x$); such a frame is known as 
	a ``rest frame" for $u.$
	By Lemma \ref{L:TransformationProperties}, it suffices to check that \eqref{E:InuMaxwellianidentity} - \eqref{E:nzEntrelation}
	hold in such a rest frame. Using definition \eqref{E:TBoltzmanndef}, it follows that $I^{\mu}$ can be expressed
	in a rest frame for $u$ as
	\begin{align} \label{E:Iproofexpression}
		I^{\mu}[\mathcal{M}] = c n \frac{z}{4 \pi m_0^3 c^3 K_2(z)} 
			\int_{\mathbb{R}^3} P^{\mu} 
			\exp\Big(\frac{-z P^{0}}{m_0 c}\Big) \frac{d \barp}{P^{0}}.
	\end{align}
	By symmetry, it follows that $I^{\mu}[\mathcal{M}]$ is proportional to $(1,0,0,0),$ and therefore also to $u.$ Let us denote 
	the proportionality constant by $A,$ so that in \emph{any} inertial coordinate system, we have that
	\begin{align} 
		I^{\mu}[\mathcal{M}] = Au^{\mu}.
		\notag
	\end{align}
	We thus have that $A = I^0/u^0.$ Furthermore, it follows from \eqref{E:Iproofexpression} that in a rest frame for $u,$ we 
	have that
	\begin{align} \label{E:nBoltzmannintegral}
		A = n \frac{z}{4 \pi m_0^3 c^3 K_2(z)} 
			\int_{\mathbb{R}^3} \exp\Big(\frac{-z P^{0}}{m_0 c}\Big) \, d \barp.
	\end{align}
	We carry out the integration in \eqref{E:nBoltzmannintegral} using spherical coordinates for
	$\barp,$ and making the change of variables $\lambda = z P^{0}/(m_0 c)$.  This implies that 
	$|\barp| = m_0 cz^{-1}(\lambda^2 - z^2)^{1/2},$ 
	$d |\barp| = m_0c z^{-1} \lambda (\lambda^2 - z^2)^{-1/2},$ and we obtain
	\begin{align} \label{E:nBoltzmannintegrated}
		A = n \frac{z}{4 \pi m_0^3 c^3 K_2(z)} 
			 4 \pi m_0^3 c^3 z^{-3} \int_{\lambda = z}^{\infty} \lambda e^{-\lambda}\big(\lambda^2 - z^2 \big)^{1/2} 
			\, d \lambda.
	\end{align} 
	From the Bessel function identity \eqref{E:Besselalternatedef} in the case $j=2,$ it follows that 
$A = n$,
	which completes the proof of \eqref{E:InuMaxwellianidentity}.

	To prove \eqref{E:PnTequationofstateMaxwellian}, we note that in a rest frame for $u,$ 
	it follows that $P^{\kappa} P^{\lambda} \Pi_{\kappa \lambda} = |\bar{P}|^2,$ 
	where $\Pi_{\mu \nu}$ is defined in \eqref{E:Pidef}.
	Inserting this formula into definition \eqref{E:Pressuredef}, using definition \eqref{E:TBoltzmanndef}, integrating with spherical 
	coordinates for $\barp,$ and using the integration variable $\lambda$ as above,  
	we have that
	\begin{align} 
	\label{E:PMaxwellianintegral}
		\pressure[\mathcal{M}] & = \frac{1}{3}c n \frac{z}{4 \pi m_0^3 c^3 K_2(z)} 
			\int_{\mathbb{R}^3} |\barp|^2 \exp\Big(\frac{-z P^{0}}{m_0 c}\Big) \frac{d \barp}{P^{0}} \\
		& = \frac{1}{3}c n \frac{z}{4 \pi m_0^3 c^3 K_2(z)}  4 \pi m_0^4 c^4 z^{-4} \int_{\lambda = z}^{\infty}
			e^{-\lambda}\big(\lambda^2 - z^2 \big)^{3/2} \, d \lambda.
		\notag
		\end{align}
	Referring to definition \eqref{E:Besseldef}, it follows that $\pressure[\mathcal{M}] =  n \frac{m_0 c^2}{z},$
	which proves \eqref{E:PnTequationofstateMaxwellian}. 

	To prove \eqref{E:rhoMaxwellian}, we note that in a rest frame for $u,$  
	$u_{\kappa} u_{\lambda} P^{\kappa} P^{\lambda} = c^2 (P^{0})^2.$
	Inserting this formula into definition \eqref{E:rhoBoltzmanndef}, carrying out the integration in spherical $\barp$ 
	coordinates, using the change of variables $\lambda$ as above,
	referring to definition \eqref{E:Besseldef}, and using the recursion formula \eqref{E:Besselrecursion} in the
	case $j=2,$ we have that
	\begin{align}  \label{E:rhointegral}
		\rho[\mathcal{M}] & = c n \frac{z}{4 \pi m_0^3 c^3 K_2(z)}
			\int_{\mathbb{R}^3} P^{0} \exp\Big(\frac{-z P^{0}}{m_0 c}\Big) d \barp 
			\\
		& = c n \frac{z}{4 \pi m_0^3 c^3 K_2(z)}  4 \pi m_0^4 c^4 z^{-4}
			\int_{\lambda = z}^{\lambda = \infty} e^{- \lambda} \lambda^2 (\lambda^2 - z^2)^{1/2} d \lambda 
			\notag
			 \\
		& = c n \frac{z}{4 \pi m_0^3 c^3 K_2(z)}  4 \pi m_0^4 c^4 z^{-4}
			\int_{\lambda = z}^{\lambda = \infty} e^{- \lambda} (\lambda^2 - z^2)^{3/2} d \lambda 
			\notag
			\\
		& \ \ \ + c n \frac{z}{4 \pi m_0^3 c^3 K_2(z)}  4 \pi m_0^4 c^4 z^{-2}
			\int_{\lambda = z}^{\lambda = \infty} e^{- \lambda} (\lambda^2 - z^2)^{1/2} d \lambda 
			\notag
			 \\ 
		& = m_0 c^2 n  \frac{z}{K_2(z)} \Big\lbrace \frac{3 K_2(z)}{z^2} + \frac{K_1(z)}{z} 
			\Big\rbrace \notag 
			\\
		& =  3 n k_B \Temp + m_0 c^2 n\frac{K_1(z)}{K_2(z)} = m_0c^2 n[\mathcal{M}]  \frac{K_3(z)}{K_2(z)} - k_B 
			n[\mathcal{M}]\Temp.
			\notag
	\end{align}	 
	Combining this identity with \eqref{E:PnTequationofstateMaxwellian}, we deduce \eqref{E:rhoMaxwellian}.
	
	To prove \eqref{E:TBoltzMaxwellian}, we 
	first note that in a rest frame for $u,$ we have that
	\begin{align}
		T_{fluid}^{\mu \nu}[\mathcal{M}] = 
			\mbox{diag}(\rho[\mathcal{M}],\pressure[\mathcal{M}],\pressure[\mathcal{M}],\pressure[\mathcal{M}]),
			\notag
	\end{align}
	and
	\begin{align} \label{E:TBoltzMaxwellianrestframe}
		T_{Boltz}^{\mu \nu}[\mathcal{M}] = c n \frac{z}{4 \pi m_0^3 c^3 K_2(z)}
			\int_{\mathbb{R}^3} P^{\mu} P^{\nu} \exp\Big(\frac{-z P^{0}}{m_0 c}\Big) \frac{d \barp}{P^{0}}.
	\end{align}
	The fact that $\mu \neq \nu \implies T_{Boltz}^{\mu \nu}[\mathcal{M}]= 0$ follows from symmetry. The fact
	that $T_{Boltz}^{00}[\mathcal{M}]= \rho[\mathcal{M}]$ follows directly from comparing the integral expressions 
	\eqref{E:rhointegral} and \eqref{E:TBoltzMaxwellianrestframe}, while the fact that $T_{Boltz}^{jj}[\mathcal{M}]= 
	\pressure[\mathcal{M}]$ (there is no summation in $j$ here) follows 
	from comparing the integral expressions \eqref{E:PMaxwellianintegral} and \eqref{E:TBoltzMaxwellianrestframe} using symmetry.

	To prove \eqref{E:nzEntrelation}, we  notice that
	 in a rest frame for $u,$ we have that
	$u_{\kappa} P^{\kappa} = - c P^{0}.$ Inserting this formula into definitions
	\eqref{E:EntfourflowBoltzmann} and \eqref{E:EntBoltzmann}, and evaluating
	the two integrals that arise as in \eqref{E:nBoltzmannintegrated} and
	\eqref{E:rhointegral}, we have that	
	\begin{align}
			\Ent[\mathcal{M}]  & = (- k_B n^{-1}) n \frac{z}{4 \pi m_0^3 c^3 K_2(z)}
			\int_{\mathbb{R}^3}  \exp\Big(\frac{-z P^{0}}{m_0 c}\Big) 
				\notag \\
			& \ \ \ \times
			\Big \lbrace \frac{-z P^{0}}{m_0 c} + \mbox{ln}\Big[h^3 n \frac{z}{4 \pi m_0^3 c^3 K_2(z)}\Big] - 1 \Big \rbrace 
			 d \barp \notag \\
		 & = k_B \frac{z}{4 \pi m_0^3 c^3 K_2(z)}  \frac{z}{m_0c} 
			\int_{\mathbb{R}^3} \exp\Big(\frac{-z P^{0}}{m_0 c}\Big) P^{0} d \barp 
		 \notag \\
		 & \ \ \ - k_B \frac{z}{4 \pi m_0^3 c^3 K_2(z)} 
			 \Big \lbrace \mbox{ln}\Big[h^3 n \frac{z}{4 \pi m_0^3 c^3 K_2(z)}\Big] - 1 \Big \rbrace
			\int_{\mathbb{R}^3} \exp\Big(\frac{-z P^{0}}{m_0 c}\Big) d \barp
			 \notag \\
		 & =  k_B  \frac{z}{4 \pi m_0^3 c^3 K_2(z)}  \frac{z}{m_0c}
			 4 \pi m_0^4 c^4 \Big\lbrace \frac{3 K_2(z)}{z^2} + \frac{K_1(z)}{z} 
			\Big\rbrace
			 \notag \\
		 & \ \ \ - k_B \frac{z}{4 \pi m_0^3 c^3 K_2(z)} 
			 \Big \lbrace \mbox{ln}\Big[h^3 n \frac{z}{4 \pi m_0^3 c^3 K_2(z)}\Big] - 1 \Big \rbrace
			 4 \pi m_0^3 c^3 \frac{K_2(z)}{z} 
			 \notag \\
		 & = k_B \Big\lbrace 3 + z \frac{K_1(z)}{K_2(z)} \Big\rbrace 
			- k_B \Big\lbrace \mbox{ln}\Big[h^3 n \frac{z}{4 \pi m_0^3 c^3 K_2(z)}\Big] - 1 \Big \rbrace. 
			\label{E:lastlinenzEntrelationproof}
		\end{align}
	\eqref{E:nzEntrelation} now follows from \eqref{E:lastlinenzEntrelationproof} and simple algebraic manipulation
	(solve for $n$).
	\end{proof}

The next lemma was used in Section \ref{S:rE} to derive an equivalent version of the rE system; i.e., equations \eqref{E:rEentropy} - \eqref{E:rEq}. More specifically, only equation \eqref{E:rhoPnMaxwellrelation} was used. However, as an aside, we also discuss the fundamental thermodynamic relation \eqref{E:nTrhoEntMaxwellrelation} (see \cite{dC2007b}), which can, in consideration of the positivity of $n$ and $\Temp,$ be used to show that $\Ent$ can be written as a smooth function of $n, \rho.$

\begin{proposition} \label{L:Maxwellrelations}
	Assume that the functional relations \eqref{E:PnTequationofstateMaxwellian}, \eqref{E:rhoMaxwellian}, and 
	\eqref{E:nzEntrelation} hold for the variables $\pressure, \rho, n, \Temp,$ and  $\Ent.$ Then the  additional 
	relations also hold:
	\begin{align}
		n\Temp & = \left. \frac{\partial \rho}{\partial \Ent} \right|_{n}, \label{E:nTrhoEntMaxwellrelation} 
		\\
		\rho + \pressure & = n \left. \frac{\partial \rho}{\partial n} \right|_{\Ent}. \label{E:rhoPnMaxwellrelation}
\end{align}
\end{proposition}

\begin{proof}
	To ease the notation, 
	we use the notation \eqref{zdef} and abbreviate 
	$K_j = K_j(z),$ $K_j' = \frac{d}{d z} K_j(z),$ where $K_j(z)$ is the Bessel function
	defined in \eqref{E:Besseldef}.
	To begin the proof of \eqref{E:nTrhoEntMaxwellrelation}, we first note that by the chain rule, it follows that
	\begin{align} \label{E:drhodEntchainrule}
		\left. \frac{\partial \rho}{\partial \Ent} \right|_{n} = \left. \frac{\partial \rho}{\partial z} \right|_{n} 
		\left. \frac{\partial z}{\partial \Ent} \right|_{n}.
	\end{align}
	We claim that equation \eqref{E:drhodEntchainrule} leads to the following identities:
	\begin{align} 
		\label{E:drhodEntidentities}
		\left. \frac{\partial \rho}{\partial \Ent} \right|_{n} & = \frac{nm_0c^2}{k_B z}\underbrace{\Big \lbrace
			\frac{K_3' - K_2^{-1}K_2'K_3 + z^{-2}K_2}{z^{-1}K_2' - z^{-2}K_2
				+ K_3' + z^{-1}K_3 - K_2^{-1}K_2'K_3} \Big \rbrace}_{= 1} 
				\\
			& = n\Temp, \notag
	\end{align}
	which completes 
	the proof of \eqref{E:nTrhoEntMaxwellrelation}.

	To see that \eqref{E:drhodEntidentities} holds, we note that differentiating the last equality in \eqref{E:rhoMaxwellian} 
	leads to the relation
	\begin{align} \label{E:drhodznconstant}
		\left. \frac{\partial \rho}{\partial z} \right|_{n} = m_0c^2 n \Big(\frac{K_3'}{K_2} - \frac{K_2'K_3}{K_2^2} 
		+ \frac{1}{z^2} \Big),
	\end{align}
	while differentiating each side of \eqref{E:nzEntrelation} with respect to $\Ent$ (while $n$ is held constant) leads to
	the following identity:
	\begin{align} \label{E:dzdEntnconstant}
		0 = \left. \frac{\partial z}{\partial \Ent} \right|_{n} \Big[ z^{-1}K_2' - z^{-2}K_2
				+ K_3' + z^{-1}K_3 - K_2^{-1}K_2'K_3 \Big] - k_B^{-1} z^{-1} K_2(z).
	\end{align}
	Inserting \eqref{E:drhodznconstant} and \eqref{E:dzdEntnconstant} into the left-hand
	side of \eqref{E:drhodEntidentities} implies the first equality. The second equality in
	\eqref{E:drhodEntidentities} follows from \eqref{E:Besseldifferentialidentity} in the case $j=2,$ 
	which implies that the term above
	the under-braces is equal to $1,$ and from the definition of $z.$
	
	The proof of \eqref{E:rhoPnMaxwellrelation} follows similarly using the chain rule identity \\
	$\left. \frac{\partial \rho}{\partial n} \right|_{\Ent} = \left. \frac{\partial \rho}{\partial n} \right|_{z} 
		+ \left. \frac{\partial \rho}{\partial z} \right|_{n}
		\left. \frac{\partial z}{\partial n} \right|_{\Ent},$ and we omit the calculations.
\end{proof}

\subsection{The invertibility of the maps $\mathfrak{H}(n,z)$ and $\mathfrak{P}(n,z)$} \label{SS:SolveforTemp}

In this short section, we state and prove Lemma \ref{L:invertiblemaps}, which addresses the issue of solving for 
$(n,\Temp)$ in terms of $(\Ent, \pressure).$ This lemma rigorously shows that Conjecture \ref{C:Tempisgood} is true outside of a compact set of $\Temp$ values. Furthermore, at the end of the section, we provide Figure \ref{fig:partialppartialz}, which is our numerical evidence for the validity of the conjecture. For the purposes of avoiding repetition, during the proof of Lemma \ref{L:invertiblemaps}, it is convenient to use notation that is defined below in the proof of Lemma \ref{L:speedofsound}.
However, logically speaking, the proof of Lemma \ref{L:invertiblemaps} comes before the proof of Lemma \ref{L:speedofsound}.

\begin{lemma} \label{L:invertiblemaps}
	Consider the smooth maps $\Ent = \mathfrak{H}(n,z)$ and $\pressure = \mathfrak{P}(n,z)$ defined
	\eqref{E:changeofstatespacevariablesH} and \eqref{E:changeofstatespacevariablesP} respectively.  
	Then the map $(n,z) \rightarrow \big(\mathfrak{H}(n,z), \mathfrak{P}(n,z)\big)$ is invertible with smooth inverse if 
	$0 < z \leq 1/10$ or $z \geq 70.$
\end{lemma}

\begin{proof}
Using \eqref{E:PnTEulerrelationintro} and \eqref{E:nEntzrelationintro}, it follows that $p\exp \Big(\frac{\Ent}{k_B} \Big)$
is a smooth function of $z$ alone. Therefore, by the implicit function theorem, since $p >0$ holds whenever $z > 0,$ we can locally solve for $z$ in terms of
$\Ent, \pressure$ if $\left. \frac{\partial \pressure}{\partial z} \right|_{\Ent} \neq 0.$ We will show that
$\left. \frac{\partial \pressure}{\partial z} \right|_{\Ent} < 0$ holds for $0 < z \leq 1/10$ and $z \geq 70.$

We begin by quoting equation \eqref{E:PderivativeoverP}, which states that
\begin{align}
	\frac{\partial_z |_{\Ent} \pressure}{\pressure} & = 3 \frac{K_1(z)}{K_2(z)} + z  \Big( \frac{K_1(z)}{K_2(z)} 
		\Big)^2 - z - \frac{4}{z}. \label{E:PderivativeoverPagain}	
\end{align}
Using \eqref{E:PderivativeoverPagain} and the expansions \eqref{E:K1overK2smallzexpansion} - \eqref{E:K1overK2squaredsmallzexpansion}, it follows that
\begin{align}
	z \frac{\partial_z |_{\Ent} \pressure}{\pressure} = - 4 + z^2/2 + 3 z \epsilon_1(z) + z^4/4 + z^2 \epsilon_2(z). 
	\notag
\end{align}
Thus, using the bounds \eqref{E:smallzepsilonbounds}, we obtain that
\begin{align}
	z \frac{\partial_z |_{\Ent} \pressure}{\pressure} < -3, && (0 < z \leq 1 /10).
	\notag
\end{align}
Since $\pressure \geq 0,$ it follows that $\partial_z |_{\Ent} \pressure$ whenever $0 < z \leq 1 /10$ as desired.

On the other hand, using the expansions \eqref{E:K1overK2largezexpansion} and \eqref{E:K1overK2squaredlargezexpansion},
it follows that
\begin{align}
	z \frac{\partial_z |_{\Ent} \pressure}{\pressure} = -5/2 + \frac{45}{8z} + 3z \widetilde{\epsilon}_1(z) + z^2 
	\widetilde{\epsilon}_2(z). 
	\notag
\end{align}
Therefore, using the bounds \eqref{E:largezepsilonbounds}, it follows that
\begin{align}
	z \frac{\partial_z |_{\Ent} \pressure}{\pressure} < -1, && (z \geq 70).
	\notag
\end{align}
Therefore, by the implicit function theorem, we can solve for $z$ in terms of $\Ent, p$ if $0 < z \leq 1/10$ or $z \geq 70.$ Since $n = m_0^{-1} c^{-2} \pressure z$ from \eqref{E:PnTEulerrelationintro}, the same is true of $n.$ This completes the proof of Lemma \ref{L:invertiblemaps}.
\end{proof}

\begin{remark}
It is clear from the proof of the lemma that Conjecture \ref{C:Tempisgood} can be shown by demonstrating the negativity of 
$\left. \frac{\partial \pressure}{\partial z} \right|_{\Ent}$ for all $z > 0.$ Thus, the numerical plot 
in Figure \ref{fig:partialppartialz}, which was created with Maple 11.0, is the motivation for our conjecture.
\end{remark}

\begin{figure}[h!]
  \begin{center}
 \scalebox{.50}{\includegraphics{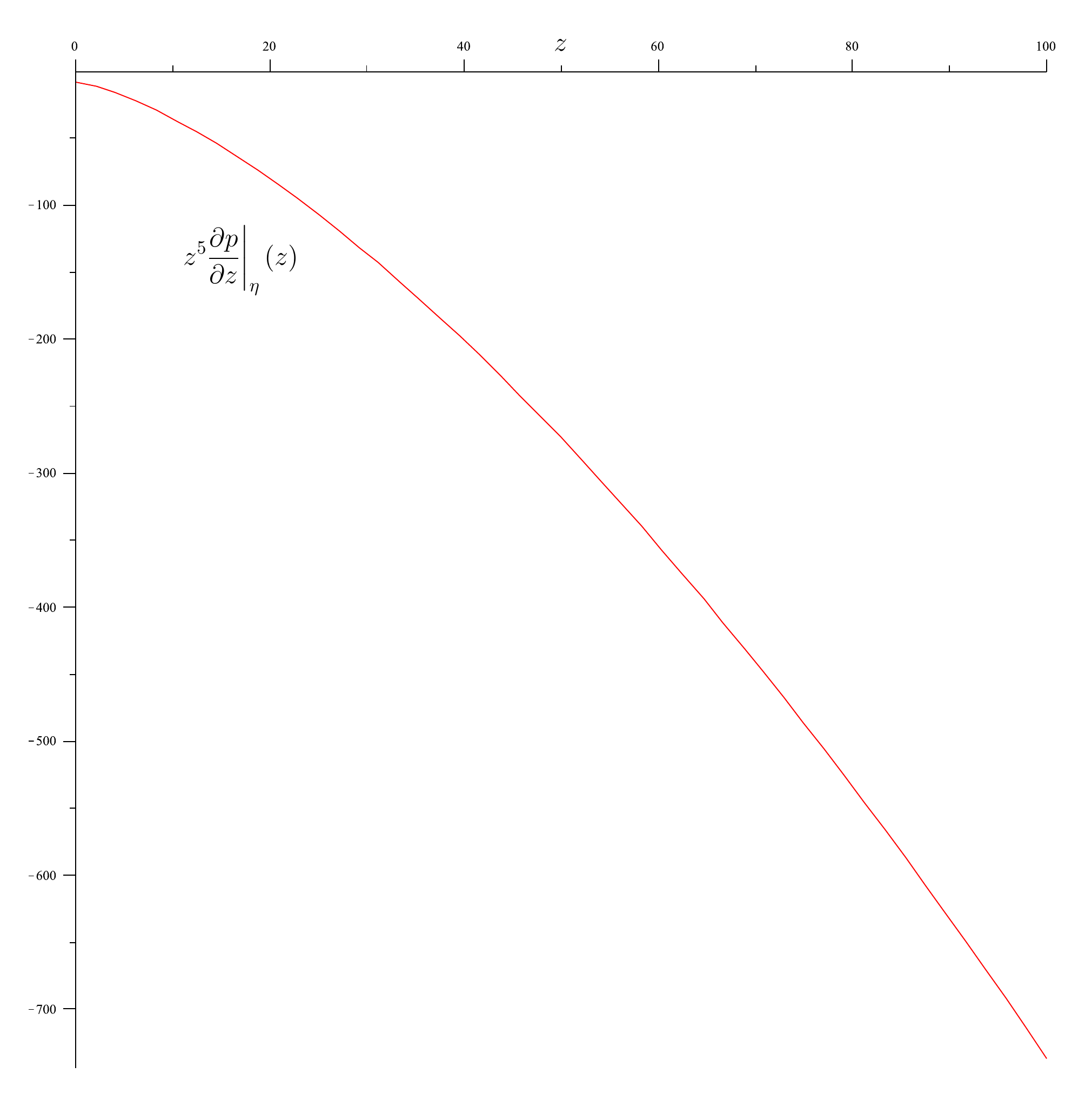}}
  \end{center}
  \caption{$\left. z^5 \frac{\partial \pressure}{\partial z} \right|_{\Ent}$ plotted as a function of $z.$}
  \label{fig:partialppartialz}
\end{figure}

\subsection[Regimes of hyperbolicity for the rE system]{Regimes of hyperbolicity for the rE system and the existence of the kinetic equation of state} \label{SS:hyperbolicity}

In this section, we prove that whenever $\Temp$ is sufficiently small and positive or sufficiently large, 
there exists an equation of state of the form \eqref{E:EOS}, i.e., of the form $\pressure = f_{kinetic}(\eta, \rho).$
Furthermore, under the same temperature assumptions, we show that the equation of state satisfies 
$0 < \left. \frac{\partial f_{kinetic}}{\partial \rho} \right|_{\Ent} < 1.$ We remark that $\left. \frac{\partial f_{kinetic}}{\partial \rho} \right|_{\Ent}$ can be expressed as a function of $\Temp$ alone. As previously discussed, this condition is sufficient to ensure the hyperbolicity of the rE system in these 
temperature regimes; in particular, as discussed in Remark \ref{R:hyperbolicity}, the condition $0 < \left. \frac{\partial f_{kinetic}}{\partial \rho} \right|_{\Ent}$ plays a fundamental role in the proof of local existence. Our result rigorously shows that outside of a compact set of $\Temp$ values, a slightly weaker version of Conjecture \ref{C:Speedisreal} holds. Furthermore, at the end of this section, we provide Figure \ref{fig:toucan}, which is our numerical evidence for the validity of the conjecture. Also see the discussion at the end of Section \ref{SS:rEintro}. For convenience, we use the variable $z$ from \eqref{zdef} during the statement and proof of the lemma.

\begin{lemma} \label{L:speedofsound} \textbf{(Hyperbolicity of the rE system)}
	Assume that the functional relations \eqref{E:PnTequationofstateMaxwellian}, \eqref{E:rhoMaxwellian}, and 
	\eqref{E:nzEntrelation} hold for the macroscopic variables $n, \Temp, \Ent, \pressure,$ and $\rho.$ 
	Then if $0 < z \leq 1/10$ or $z \geq 70,$ $\pressure$ can be expressed as a smooth function $f_{kinetic}$
	of $\Ent$ and $\rho:$ $\pressure = f_{kinetic}(\Ent,\rho).$

	Furthermore, the following estimate holds for $0 < z \leq 1/10:$
	\begin{align}
		\bigg| \big. \frac{\partial \pressure}{\partial \rho} \big|_{\Ent} -  \frac{1}{3} \bigg| \leq z^2.
		\label{E:speedofsoundsquared0asymptoticexpansion}	
	\end{align}
Additionally, the following estimate holds for $z \geq 70:$
	\begin{align}
	 \bigg| \big. \frac{\partial \rho}{\partial \pressure} \big|_{\Ent} - \frac{3z}{5} \bigg| & \leq 41. 
	 	\label{E:speedofsoundsquaredinfinityasymptoticexpansion}
	\end{align}
	
\end{lemma}

\begin{remark}
	We did not attempt to be optimal in our estimate of the error terms on the right-hand sides of the above inequalities.
\end{remark}

\begin{proof} 
	It follows from \eqref{E:PnTequationofstateMaxwellian}, \eqref{E:rhoMaxwellian}, and \eqref{E:nzEntrelation} 
	that
	\begin{align} 
		\pressure & = 4\pi e^4 m_0^4 c^5 h^{-3} \exp\Big(\frac{-\Ent}{k_B} 
		\Big)\frac{K_2(z)}{z^2} \exp\Big(z  \frac{K_1(z)}{K_2(z)} \Big), \label{E:pressureexpression} \\
		\rho & = \pressure \Big(z \frac{K_1(z)}{K_2(z)} + 3 \Big). \label{E:rhoPzrelation} 
	\end{align}
	We use the following version of the chain rule: $\frac{\partial \pressure}{\partial \rho} \big|_{\Ent} = \frac{\partial_z|_{\Ent} \pressure}{\partial_z|_{\Ent} \rho}$.
	Using \eqref{E:rhoPzrelation}, we further deduce that		
	\begin{align} \label{E:inversespeedofsoundsquaredexpression}
		\frac{\partial_z|_{\Ent} \rho}{\partial_z|_{\Ent} \pressure}
			= 3 + z \frac{K_1(z)}{K_2(z)} +  \Big(\frac{\pressure}{\partial_z|_{\Ent} \pressure} \Big)
			\frac{d}{dz} \Big[z \frac{K_1(z)}{K_2(z)}\Big].				
	\end{align}
	We then use the identities $z K_1' = K_1 - zK_2$ and $K_2' = -2 z^{-1} K_2 - K_1,$ which follow from
	\eqref{E:Besselrecursion} and \eqref{E:Besseldifferentialidentity}, to compute that	
	\begin{align}
		\frac{d}{dz} \Big( z  \frac{K_1(z)}{K_2(z)} \Big) & = 4 \frac{K_1(z)}{K_2(z)}
			+ z  \Big( \frac{K_1(z)}{K_2(z)} \Big)^2 -z, \label{E:zK1K2derivative} \\
		\frac{\partial_z |_{\Ent} \pressure}{\pressure} & = 3 \frac{K_1(z)}{K_2(z)} + z  \Big( \frac{K_1(z)}{K_2(z)} 
		\Big)^2 - z - \frac{4}{z}. \label{E:PderivativeoverP}
	\end{align}
	Combining \eqref{E:inversespeedofsoundsquaredexpression}, \eqref{E:zK1K2derivative}, and \eqref{E:PderivativeoverP},
	we have that
\begin{align} \label{E:inversespeedofsoundsquaredBessel}
		\frac{\partial_z|_{\Ent} \rho}{\partial_z|_{\Ent} \pressure}
			= 3 + z \frac{K_1(z)}{K_2(z)}
+  
\frac{ 4 \frac{K_1(z)}{K_2(z)} + z  \Big( \frac{K_1(z)}{K_2(z)} \Big)^2 -z }{3 \frac{K_1(z)}{K_2(z)} + z  \Big( \frac{K_1(z)}{K_2(z)} \Big)^2 - z - \frac{4}{z}}.	
\end{align}
Using Corollary \ref{C:Besselratio}, we can write
\begin{align} \label{E:K1overK2smallzexpansion}
	\frac{K_1(z)}{K_2(z)} & = \frac{z}{2} + \epsilon_1(z), \\
	\Big(\frac{K_1(z)}{K_2(z)}\Big)^2	& = \frac{z^2}{4} + \epsilon_2(z), \label{E:K1overK2squaredsmallzexpansion}
\end{align}
where for $0 < z \leq 1/10,$ we have the estimates 
\begin{align} \label{E:smallzepsilonbounds}
	|\epsilon_1(z)| & \leq 2z^2, && |\epsilon_2(z)| \leq 2z^3.
\end{align}
Inserting these expansions into \eqref{E:inversespeedofsoundsquaredBessel} and multiplying the numerator and denominator 
of the fraction by $z,$ we have that
\begin{align}  \label{E:inversespeedofsoundsquaredexpressionexpanded}
	\frac{\partial_z|_{\Ent} \rho}{\partial_z|_{\Ent} \pressure}
	& =	3 + \frac{z^2}{2} + z \epsilon_1 + 
	\frac{z^2 + 4 z \epsilon_1 + z^4/4 + z^2 \epsilon_2}{z^2/2 + 3 z \epsilon_1
+ z^4/4 + z \epsilon_2 - 4}.
\end{align}
Using the bounds \eqref{E:smallzepsilonbounds}, it follows that for $0 < z \leq 1/10,$ the second term on the right-hand side of \eqref{E:inversespeedofsoundsquaredexpressionexpanded} (i.e. the $z^2/2$ term)  partially cancels the last term (i.e., the large fraction which is negative), which implies that
\begin{align} \label{E:inversespeedofsoundsquaredestimated}
	\Big| \frac{\partial_z|_{\Ent} \rho}{\partial_z|_{\Ent} \pressure} - 3 \Big| & \leq \frac{7z^2}{10}, 
	&& (0 < z \leq 1/10).
\end{align}
The facts that $\pressure$ can be expressed as a smooth function of $\Ent$ and $\rho$ whenever $0 < z \leq 1/10,$ and that  inequality \eqref{E:speedofsoundsquared0asymptoticexpansion} is verified both easily follow from \eqref{E:inversespeedofsoundsquaredestimated}.

To prove \eqref{E:speedofsoundsquaredinfinityasymptoticexpansion}, we again use
Corollary \ref{C:Besselratio} to write
\begin{align}
	\frac{K_1(z)}{K_2(z)} & = 1 - \frac{3}{2z} + \frac{15}{8z^2} + \widetilde{\epsilon}_1(z), 
		\label{E:K1overK2largezexpansion} \\
	\Big(\frac{K_1(z)}{K_2(z)}\Big)^2	& = 1 - \frac{3}{z} + \frac{6}{z^2} + \widetilde{\epsilon}_2(z),
		\label{E:K1overK2squaredlargezexpansion}
\end{align}
where for $z \geq 10,$ we have  that
\begin{align} \label{E:largezepsilonbounds}
	|\widetilde{\epsilon}_1(z)| & \leq \frac{16}{z^2}, && |\widetilde{\epsilon}_2(z)| \leq \frac{40}{z^3}.
\end{align}
Inserting these expansions into \eqref{E:inversespeedofsoundsquaredBessel} and multiplying the numerator and denominator 
of the fraction by $z,$ it follows that 
\begin{align} \label{E:largezinversespeedofsoundsquaredexpressionexpanded}
	\frac{\partial_z|_{\Ent} \rho}{\partial_z|_{\Ent} \pressure}
	& = \frac{3z}{5} + \frac{3}{2} + \frac{15}{8z} + z\widetilde{\epsilon}_1
		+ \frac{\frac{9}{4} + \frac{15}{2z} + 4 z \widetilde{\epsilon}_1 + \frac{6}{5} z^2 \widetilde{\epsilon}_1 
		+ z^2 \widetilde{\epsilon}_2 + \frac{2}{5}z^3 \widetilde{\epsilon}_2}
		{-5/2 + \frac{45}{8z} + 3z \widetilde{\epsilon}_1 + z^2 \widetilde{\epsilon}_2}.
\end{align}
Using the bounds \eqref{E:largezepsilonbounds} and the expression \eqref{E:largezinversespeedofsoundsquaredexpressionexpanded}, it can be checked that for $z \geq 70,$ we have that
\begin{align} \label{E:dpdrholargezestimate}
	\Big|\frac{\partial_z|_{\Ent} \rho}{\partial_z|_{\Ent} \pressure} - \frac{3z}{5} \Big| & \leq 41.
\end{align}
The facts that $\pressure$ can be expressed as a smooth function of $\Ent$ and $\rho$ whenever $z \geq 70,$ and that  
inequality \eqref{E:speedofsoundsquaredinfinityasymptoticexpansion} is verified, both easily
follow from \eqref{E:dpdrholargezestimate}.
\end{proof}

\begin{remark}
Notice that the Conjecture \ref{C:Speedisreal} is equivalent to the conjecture that the right-hand side of \eqref{E:inversespeedofsoundsquaredBessel} is $> 3$ for all $z > 0.$ In Figure \ref{fig:toucan}, we present a numerical plot, which was created with Maple $11.0,$ that covers the set of $z$ values lying outside of the scope of Lemma \ref{L:speedofsound}, and that suggests that this conjecture is true. Note that the inequalities of the 
conjecture are stronger than those proved in the lemma. 
\end{remark}

\begin{figure}[h!]
\begin{center}
 \scalebox{.50}{\includegraphics{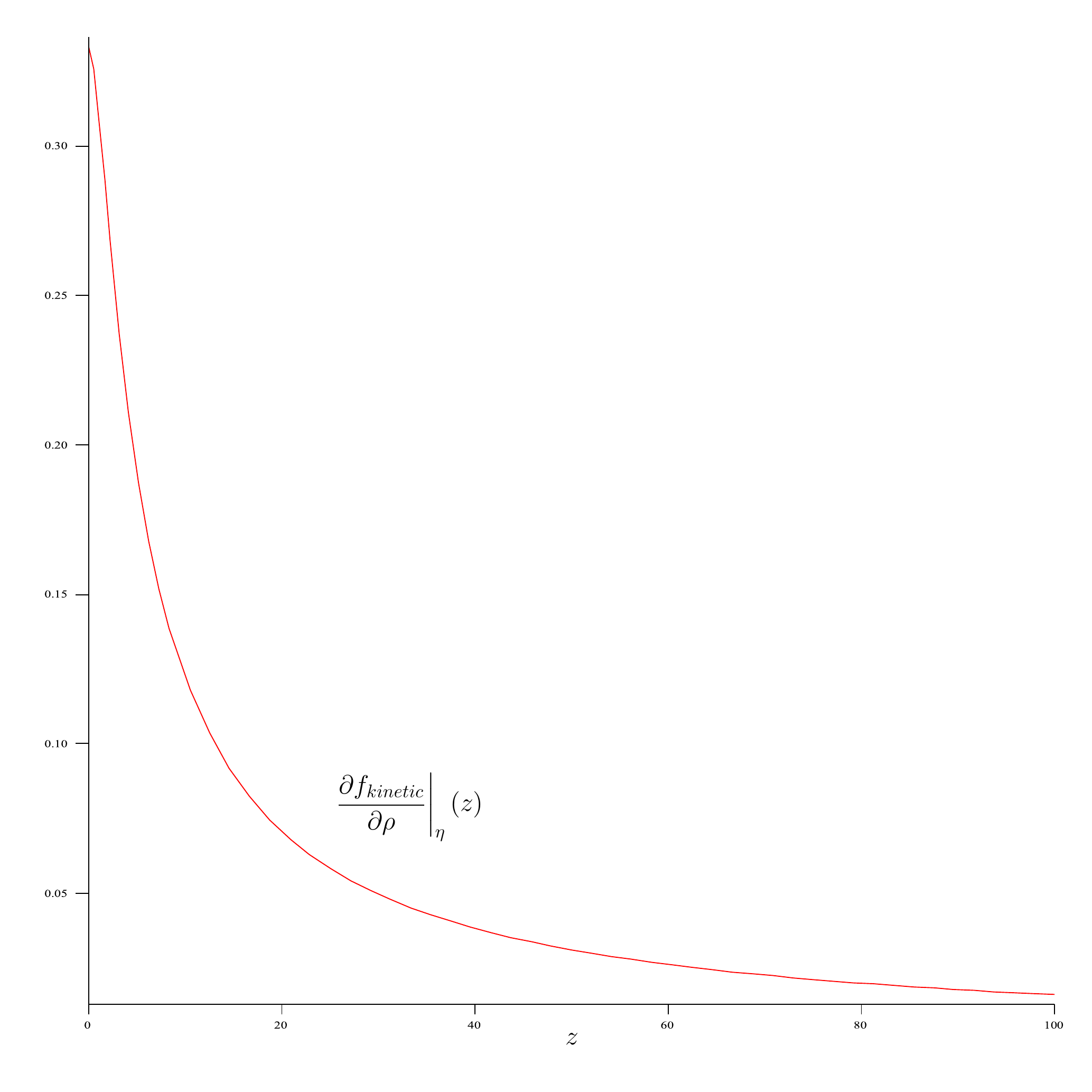}}
  \end{center}
  \caption{$\left. \frac{\partial \pressure}{\partial \rho} \right|_{\Ent}$ plotted as a function of $z.$}
  \label{fig:toucan}
\end{figure}

\subsection{Bessel function identities and inequalities} \label{SS:Besselfunction}

We now state the technical lemma that contains the Bessel function properties that we have used throughout this
article. The expansion \eqref{E:Besselasymptoticexpansion} (including the error terms) and inequality \eqref{E:Besselincreasinginj}
can be found in \cite{fO1974}. The remaining identities can be found in \cite[Chapter 2]{sdGsvLwvW1980}.

\begin{lemma}  \label{L:Bessel functions} (Properties of Bessel functions) 
	Let $K_j(z)$ be the Bessel function defined by
	\begin{align} 
	\label{E:Besseldef}
		K_j(z) & \eqdef \frac{(2^j)j!}{(2j)!} \frac{1}{z^j} \int_{\lambda = z}^{\lambda = \infty} e^{- \lambda}(\lambda^2 - z^2)^{j 
			-  (1/2)} \, d \lambda, && (j \geq 0).
	\end{align}
Then the following identities  hold:
	\begin{align}
		K_j(z)  & = \frac{2^{j-1}(j-1)!}{(2j - 2)!} \frac{1}{z^j} 
			\int_{\lambda = z}^{\lambda = \infty} \lambda e^{- \lambda}(\lambda^2 - z^2)^{j - (3/2)} \, d \lambda, 
			\quad (j > 0), 
			\label{E:Besselalternatedef} 
			\\
		K_{j+1}(z)  & = 2j \frac{K_{j}(z)}{z} + K_{j-1}(z), \quad (j \geq 1), 
		\label{E:Besselrecursion} 
\end{align}		
also
\begin{align}
		\frac{d}{dz} \Big(\frac{K_j(z)}{z^j} \Big)  & = - \Big(\frac{K_{j+1}(z)}{z^j} \Big), \quad (j \geq 0), 
			\label{E:Besseldifferentialidentity}
		\\
		K_j(z)  & = \sqrt{\frac{\pi}{2z}}e^{-z} \Big( \gamma_{j,n}(z)z^{-n} + \sum_{m=0}^{n-1} A_{j,m} z^{-m}
			\Big), \quad (j \geq 0, n \geq 1),	\label{E:Besselasymptoticexpansion} 
	\end{align}
	where the following additional identities and inequalities also hold:
\begin{align}
		A_{j,0} & = 1, \notag \\
		A_{j,m} & = \frac{(4j^2 - 1)(4j^2 - 3^2) \cdots (4j^2 - (2m -1)^2)}{m! 8^m}, \quad (j \geq 0, m \geq 1),
		\notag 
		\\
		|\gamma_{j,n}(z)| & \leq 2 \exp\big([j^2 - 1/4]z^{-1}\big)|A_{j,n}|, \quad (j \geq 0, n \geq 1),
		\notag 
		\\
		K_{j}(z) & < K_{j+1}(z), \quad (j \geq 0). \label{E:Besselincreasinginj} 	
	\end{align}	
\end{lemma}

The following corollary of Lemma \ref{L:Bessel functions} is used in the proof of Lemma \ref{L:speedofsound}.
\begin{corollary} \label{C:Besselratio}
	For $0 < z \leq 1/10,$ the following inequalities hold:
	\begin{subequations}
	\begin{align}
		\Big|\frac{K_1(z)}{K_2(z)} - \frac{z}{2} \Big| & \leq 2z^2, 
		\label{E:smallzRdecomposition} \\
		\Big|\Big(\frac{K_1(z)}{K_2(z)}\Big)^2 - \frac{z^2}{4} \Big| & \leq 2z^3.
			\label{E:smallzRsquareddecomposition}
	\end{align}
	\end{subequations}	
	For $z \geq 10,$ the following inequalities hold:
	\begin{subequations}
	\begin{align}
		\Big|\frac{K_1(z)}{K_2(z)} - 1 + \frac{3}{2z} - \frac{15}{8z^2} \Big| & \leq \frac{16}{z^3}, 
			\label{E:largezRdecomposition} \\
		\Big|\Big(\frac{K_1(z)}{K_2(z)}\Big)^2 - 1 + \frac{3}{z} - \frac{6}{z^2}\Big| & \leq \frac{40}{z^3}.
			\label{E:largezRsquareddecomposition}
	\end{align}
	\end{subequations}
\end{corollary}

\begin{proof}
	We remark that throughout the proof, we make no attempt to be optimal in our estimates.
	Using \eqref{E:Besseldef} in the case $j=1,$ and the fact that
	$\int_{\lambda = 0}^{\lambda = \infty} \lambda e^{- \lambda}  \, d \lambda = 1,$
	it follows that		
	\begin{align}\notag
		z K_1(z) - 1 =  - \int_{\lambda = 0}^{\lambda = z} e^{- \lambda}\lambda \, d \lambda
			+ \int_{\lambda = z}^{\lambda = \infty} \lambda e^{- \lambda}\bigg\lbrace \sqrt{1 - 
				\Big(\frac{z}{\lambda}\Big)^2} - 1 \bigg\rbrace \, d \lambda.
	\end{align}
	The first integral is trivially bounded in magnitude by $z^2/2.$ Using the 
	fact that $\big|\sqrt{1 - \big(z/\lambda \big)^2} - 1 \big| \leq (z/\lambda)^2$ on the domain $0 \leq z/\lambda \leq 1,$
	it follows that the second integral is bounded in magnitude by
	\begin{align}
		z \int_{\lambda = 0}^{\lambda = \infty} e^{- \lambda} \, d \lambda \leq z.
		\notag
	\end{align}
	We therefore conclude that
	\begin{align} \label{E:K1smallzexpansion}
		\Big|K_1(z) - \frac{1}{z}\Big| & \leq 1 + \frac{z}{2}.
	\end{align}
	Using \eqref{E:Besselalternatedef} in the case $j=2$ 
	and similar arguments, which we leave to the reader, we also conclude that
	\begin{align} \label{E:K2smallzexpansion}
		\Big|K_2(z) - \frac{2}{z^2}\Big| & \leq 1 + \frac{z}{3}.
	\end{align}	
	Using \eqref{E:K1smallzexpansion} and \eqref{E:K2smallzexpansion}, together with simple algebraic 
	estimates, it follows that for $0 \leq z \leq 1/ 10,$ we have	
	\begin{align}  \notag
		\Big|\frac{K_1(z)}{K_2(z)} - \frac{z}{2} \Big| \leq 2 z^2.
	\end{align}
	This proves \eqref{E:smallzRdecomposition}. Inequality \eqref{E:smallzRsquareddecomposition}
	follows from similar reasoning; we leave the details to the reader.

	To prove \eqref{E:largezRdecomposition}, we first decompose	
	\begin{align} \label{E:K1K2ratioexpansion}
		\frac{K_1(z)}{K_2(z)} = \frac{1 + A}{1 + B} = (1 + A)\Big\lbrace 1 - B + B^2 - \frac{B^3}{1 + B} \Big\rbrace,
	\end{align}
	where	
	\begin{subequations}
	\begin{align}
		A & = \frac{3}{8z} - \frac{15}{128z^2} + \frac{\gamma_{1,3}}{z^3},
			\label{E:Aexpansion}  \\
		B & = \frac{15}{8z} + \frac{105}{128z^2} + \frac{\gamma_{2,3}}{z^3}, \label{E:Bexpansion}
	\end{align}
	\end{subequations}
	and the $\gamma_{j,n}$ are from \eqref{E:Besselasymptoticexpansion}.
	
	For the remainder of the proof, we will now assume that $z \geq 10;$ all of our estimates will hold on
	this domain. Now using \eqref{E:Besselasymptoticexpansion}, it can be checked that the following inequalities hold:
	\begin{align} \label{E:gammabound}
		|\gamma_{1,3}| \leq \frac{1}{4}, && |\gamma_{2,3}| \leq 1.
	\end{align}
	Consequently, it is easy to check that the following estimates hold:	
	\begin{subequations}
	\begin{align}
		\frac{1}{4z} \leq A & \leq \frac{1}{2z}, \label{E:Abound} \\
		\frac{1}{z} \leq B & \leq \frac{2}{z}, \label{E:Bbound} \\
		\Big|B - \frac{15}{8z} \Big| & \leq \frac{1}{z^2}, \label{E:Bquadraticbound} \\
		\Big|B^2 - \Big(\frac{15}{8z}\Big)^2 \Big| & \leq \frac{4}{z^3}. \label{E:Bsquaredcubicbound}
	\end{align}
	\end{subequations}	
	Using simple algebraic calculations, it follows from the 
	expansions \eqref{E:K1K2ratioexpansion}, \eqref{E:Aexpansion}, and \eqref{E:Bexpansion} that
	\begin{align} \label{E:K1K2ratiozexpansion}
		\frac{K_1(z)}{K_2(z)} = 1 - \frac{3}{2z} + \frac{15}{8z^2} + O(z^{-3}).
	\end{align}
	In \eqref{E:K1K2ratiozexpansion}, the symbol $O(z^{-3})$ denotes the cubic (in $z^{-1}$) and higher-order terms
	that arise in the expansion of $\frac{K_1(z)}{K_2(z)}.$ We now estimate this $O(z^{-3})$ term by using
	the expansions \eqref{E:K1K2ratioexpansion} - \eqref{E:Bexpansion} to split it into the following $3$ pieces:
	\begin{subequations}
	\begin{align}
		I & = \Big(- \frac{15}{128z^2} + \frac{\gamma_{1,3}}{z^3}\Big)\Big(- B + B^2 - \frac{B^3}{1 + B}\Big), \label{E:termI} \\
		II & = \frac{3}{8z}\Big\lbrace -\Big(B - \frac{15}{8z}\Big) + B^2 - \frac{B^3}{1 + B} \Big\rbrace, \label{E:termII} \\
		III & = \frac{\gamma_{1,3}}{z^3} - \frac{\gamma_{2,3}}{z^3} + B^2 - \Big(\frac{15}{8z} \Big)^2 
			- \frac{B^3}{1 + B}. \label{E:termIII}
	\end{align}
	\end{subequations}
	It is easy to see by sign considerations (i.e., using $B > 0$) that 
	$|- B + B^2 - \frac{B^3}{1 + B}| \leq |B|.$ Using also \eqref{E:gammabound} and \eqref{E:Bbound}, we conclude 
	that the following inequality holds:
	\begin{align} \label{E:Iestimate}
		|I| \leq \left| - \frac{15}{128z^2} + \frac{\gamma_{1,3}}{z^3} \right|  |B| 
		\leq \left(\frac{15}{128z^2} + \frac{1}{4z^2}\right) 
		\frac{2}{z} \leq \frac{3}{4z^3}.
	\end{align}
	For the term $II,$ we use similar sign considerations, together with the estimates \eqref{E:Bbound}
	and \eqref{E:Bquadraticbound} to conclude that 
	\begin{equation} \label{E:IIestimate}
		|II|  \leq \frac{3}{8z} \Big\lbrace \Big|B - \frac{15}{8z}\Big| + |B|^2 \Big\rbrace 
		 \leq \frac{3}{8z} \Big\lbrace \frac{1}{z^2} + \frac{4}{z^2} \Big\rbrace 
			\leq \frac{2}{z^3}. 
	\end{equation}
	Finally, for the term $III,$ we use the fact that $B > 0,$ together with \eqref{E:gammabound},
	\eqref{E:Bbound}, and \eqref{E:Bsquaredcubicbound} to conclude that
	\begin{align} \label{E:IIIestimate}
		|III| \leq \Big\lbrace \Big|\frac{\gamma_{1,3}}{z^3} \Big| + 
		\Big|\frac{\gamma_{2,3}}{z^3} \Big| + \Big|B^2 - \Big(\frac{15}{8z} \Big)^2 \Big| + |B|^3 \Big\rbrace
		& \leq \frac{1}{4z^3} + \frac{1}{z^3} + \frac{4}{z^3} + \frac{8}{z^3} \leq \frac{53}{4z^3}. 
	\end{align}
	Adding \eqref{E:Iestimate}, \eqref{E:IIestimate}, and \eqref{E:IIIestimate}, we arrive at 
	\eqref{E:largezRdecomposition}.

	Inequality \eqref{E:largezRsquareddecomposition} can be shown directly from 
	\eqref{E:largezRdecomposition}; we omit the details.
\end{proof}

\subsection{The Hilbert expansion} \label{SS:HilbertExpansion}

In this section, we perform a Hilbert expansion for the rB equation \eqref{E:rBequivalent}.
We decompose the solution $F^\varepsilon$ as the sum \eqref{hilbertE}
where $F_0,F_1, \ldots, F_6$ in \eqref{hilbertE} will be independent of $\varepsilon$. Also, $\effRE$ is called the \emph{remainder} term; it will depend upon $\varepsilon$. Our main goal in this section is to explain how one can prove Proposition \ref{iterateBDS}, which summarizes the behavior of $F_0,F_1, \ldots, F_6;$ the remainder term $\effRE$ is analyzed in detail in the next section.

We begin by inserting the expansion \eqref{hilbertE} into \eqref{E:rBequivalent} to obtain
\begin{multline*}
\sum_{k=0}^6  \varepsilon^k (\partial_t + \phat \cdot \partial_{\barx}) F_k + \varepsilon^3(\partial_t + \phat \cdot \partial_{\barx}) \effRE 
 = 
 \sum_{k = 0}^6 \mathop{\sum_{i + j = k}}_{0 \leq i, \hspace{.02 in}j \leq 6} 
 \varepsilon^{i + j -1} \mathcal{Q}(F_i,F_j)
\\
+ \varepsilon^5 \mathcal{Q}(\effRE,\effRE) 
+ 
\sum_{k=0}^6 \varepsilon^{2 + k}
\left\{\mathcal{Q}(\effRE,F_k) + \mathcal{Q}(F_k,\effRE)\right\} 
+
\overline{A}.
\end{multline*}  
Above
$
\overline{A}
\eqdef
\sum_{k = 7}^{12} \sum_{i + j = k, 1 \leq i, j \leq 6} \varepsilon^{i + j -1} \mathcal{Q}(F_i,F_j).
$
Equating like powers of $\varepsilon$ on each side of the equation, we obtain
the following system:
\begin{align}
	  0 & = \mathcal{Q}(F_0,F_0), 
	  \label{eq1m}
	  \\
	  \partial_t F_0 + \phat \cdot \partial_{\barx} F_0 & = \mathcal{Q}(F_0,F_1) + \mathcal{Q}(F_1,F_0), 
	  \label{eq2m}
	  \\
	  \partial_t F_1 + \phat \cdot \partial_{\barx} F_1 & 
	  =	\mathcal{Q}(F_0,F_2) + \mathcal{Q}(F_2,F_0) + \mathcal{Q}(F_1,F_1), 
	  \notag
	  \\
	  & \ \vdots 
	  \notag
	  \\
	  \partial_t F_5 + \phat \cdot \partial_{\barx} F_5 & = \mathcal{Q}(F_0,F_6) + \mathcal{Q}(F_6,F_0) + 
	  	\mathop{\sum_{i + j = 6}}_{1 \leq i, \hspace{.02 in}j \leq 6} \mathcal{Q}(F_i,F_j), 
		\notag
\end{align}
while the remainder satisfies the equation
\begin{multline}
\label{remainder}
\partial_t \effRE +  \phat \cdot \partial_{\barx} \effRE 
-
 \frac{1}{\varepsilon} \Big \lbrace 
\mathcal{Q}\big(F_0,\effRE\big) 
+ 
\mathcal{Q}\big(\effRE,F_0 \big) \Big \rbrace 
=
\varepsilon^{2} \mathcal{Q}(\effRE,\effRE)
\\
+ 
\sum_{i=1}^6 \varepsilon^{i-1}
\left\{
\mathcal{Q}(F_i, \effRE) 
+ 
\mathcal{Q}(\effRE,F_i) 
\right\}
+ 
\varepsilon^2 A, 
\end{multline}
with
$$			
A  \eqdef - \varepsilon \big \lbrace \partial_t F_6 + \phat \cdot \partial_{\barx} F_6 \big \rbrace +
			\mathop{\sum_{i + j > 6}}_{1 \leq i, \hspace{.02 in}j \leq 6} 
			\varepsilon^{i + j -6} \mathcal{Q}(F_i,F_j).
$$
By \eqref{E:MaxwellianCondition}, equation \eqref{eq1m} implies that $F_0$ must be a relativistic local Maxwellian $F_0(t,\barx,\barp) = \mathcal{M} = \mathcal{M}(n(t,\barx),\Temp(t,\barx),u(t,\barx);\barp),$ as in \eqref{E:Maxwelliandef}. Consequently, the remaining equations in \eqref{eq2m} and below involve the linear operator: 
\begin{gather}
\notag
L(h) 
\eqdef
- 
\left\{
\mathcal{Q}\left( h,\mathcal{M} \right)
+
\mathcal{Q}\left(\mathcal{M}, h \right)
\right\}.
\end{gather}
We remark that $L$ is an integral operator involving only the momentum space variables.  Furthermore, $L$ is a linear 
Fredholm operator that can be inverted as long as the inhomogeneity (i.e., the terms in \eqref{eq2m}, and the equations below it, which are not of the form $L (F_i)$) 
is perpendicular to the five dimensional null space of the adjoint operator $L^\dagger:$
$\mbox{Null}(L^\dagger) = \mbox{span}\{\phi_1, \ldots , \phi_5\}
= \mbox{span} \{1, \barp , P^0\}.$ This null space can be seen easily from the standard pre-post change of variables. 
In the preceding discussion, the notion of perpendicular and adjoint is the one corresponding to the usual $L^2$ momentum space inner product defined in \eqref{L2inner}. The operator $L$ has the null space $\mbox{Null}(L) = 
\mbox{span} \{\mathcal{M}, \mathcal{M} \barp, \mathcal{M} P^0\}$.

The aforementioned perpendicularity conditions can be checked by direct calculation, 
so that equation \eqref{eq2m} and the one below it imply that
\begin{align}
	 & \left\langle \phi_i,   \partial_t F_0 + \phat \cdot \partial_{\barx} F_0 \right\rangle_{\barp}
	 =0, \quad (i=1,\cdots,5), 
	  \label{eq2m22}
	  \\
	  F_1 & = - L^{-1}\left( \partial_t F_0 + \phat \cdot \partial_{\barx} F_0 \right) + \Phi_1,
	  \notag
	  \\
	  	&\left\langle \phi_i,   \partial_t F_1 + \phat \cdot \partial_{\barx} F_1 - 
	  		\mathcal{Q}(F_1,F_1)\right\rangle_{\barp}
	  	=0, \quad (i=1,\cdots,5),
		\label{eq2m33}
	  \\
  F_2 & = - L^{-1}\left( \partial_t F_1 + \phat \cdot \partial_{\barx} F_1 - \mathcal{Q}(F_1,F_1) \right) + \Phi_2.
	  \notag
\end{align}
Above $\Phi_1$ and $\Phi_2$ are elements of the null space of $L$ (i.e., $L (\Phi_1) = L (\Phi_2)=0$).
Applying the operator $\partial_{\mu}$ to each side of \eqref{E:TBoltzmanndef}, differentiating under the integral,
and using \eqref{E:TBoltzMaxwellian}, it follows that equation \eqref{eq2m22} implies that the relativistic Euler equations \eqref{E:nuconservationlawintro} are verified by the macroscopic quantities $n[\mathcal{M}],$ 
$\Temp[\mathcal{M}],$ and $u[\mathcal{M}]$ corresponding to the Maxwellian $\mathcal{M} = F_0$. 
\emph{This explains the fact that in order to initiate the Hilbert expansion, we solve (with the help of Theorem \ref{MainThmEuler}) for a smooth solution to the rE system using the variables $(n, \Temp, u)$.} 

As discussed in Cercignani-Kremer \cite[Section 5.5]{MR1898707}, the parameters in the expansion of $\Phi_1$ in terms of the basis $\lbrace \mathcal{M},$ $\mathcal{M}P^0,$ $\mathcal{M}P^1$, $\mathcal{M}P^2,$ $\mathcal{M}P^3 \rbrace$ satisfy a linearized inhomogeneous version of the relativistic Euler equations; with the help of the expression for $F_1$ in \eqref{eq2m22}, this enables us to find a solution $F_1$ to equation \eqref{eq2m}. Furthermore, the higher order correction terms $F_2, F_3, \ldots, F_6$ can be solved for in the same way, where the corresponding inhomogeneous terms in the linearized relativistic Euler equations depend upon the previous terms in the expansion. We refer to Cercignani-Kremer \cite[Section 5.5]{MR1898707} for more details on these terms in the expansion. We remark that a careful treatment of the non-relativistic Hilbert expansion is found in \cite{MR0135535,MR0156656,MR586416}. In particular, we are using the argument from \cite{MR586416}. These arguments carry over directly (once one identifies the null spaces in the relativistic case, as we have done above).

We will use the following results, which are not studied in detail here. First, the terms $F_1,\ldots,F_6$ are smooth in $(t,\barx),$ and they also have decay in the momentum variables, $\barp$. Consider, for example, $F_1$. For the Newtonian version of $L$, Grad \cite{MR0156656} and Caflisch \cite{MR586416} argue that $L^{-1}$ preserves decay in the momentum $\barp$. Their argument carries over directly to our case of the relativistic Boltzmann equation as follows. We can combine the argument in \cite{MR0156656,MR586416} with the relativistic estimates \eqref{hypNU}, Lemma \ref{boundK2}, Lemma \ref{boundKinfX}, and arguments as in \cite{MR2366140},
to see that indeed $L^{-1}$ preserves momentum decay.
This leads to the conclusion that $- L^{-1}\left( \partial_t F_0 + \phat \cdot \partial_{\barx} F_0 \right)$ decays at infinity as fast as $\mathcal{M}^q$ for any $0< q< 1$.
Thus, $F_1$ will decay as fast as $\mathcal{M}^q$, and $F_1$ is smooth in $(t,\barx)$ since the parameters in the expansion of $\Phi_1$ in terms of $\mathcal{M}$, $\mathcal{M}P^0$, $\mathcal{M}P^1$, $\mathcal{M}P^2$,
$\mathcal{M}P^3$ solve linear equations with forcing terms coming from the smooth functions $n, \Temp, u$. This argument is similar for the higher order terms in the expansion.

The next proposition summarizes the estimates that we use in the next section.

\begin{proposition}
\label{iterateBDS}
Let $\big(n(t,\barx), \Temp(t,\barx), u(t,\barx) \big)$ be a smooth solution (see Remark \ref{R:Regularity}) of the rE equations
\eqref{E:nuconservationlawintro} on a time interval $[0,T] \times \mathbb{R}^3_{\barx}$. Form the relativistic Maxwellian 
$F_0 = \mathcal{M}(n,\Temp,u;\barp)$ as in \eqref{E:Maxwelliandef}. Then the terms $F_1, \ldots, F_6$ of the Hilbert expansion are smooth in $(t,x)\in [0,T] \times \mathbb{R}^3_{\barx}$ and for any $0< q< 1,$ they have momentum decay given by
$$
\left| F_j(t,\barx,\barp) \right| 
\le
 C(q) \mathcal{M}^q(n(t,\barx),\Temp(t,\barx),u(t,\barx);\barp), \quad (j=1,2,\cdots,6).
$$
The constants in this bound are independent of $(t,\barx,\barp)$.
\end{proposition}

\subsection{Relativistic Boltzmann estimates} \label{SS:rBestimates}

In this section we prove our main result, Theorem \ref{MainThm}.
Using Theorem \ref{MainThmEuler}, we may assume that there is a sufficiently smooth solution $(n,\Temp,u)$
to the relativistic Euler equations satisfying all of the desired properties in Theorem \ref{MainThm}. 
We can then construct the local relativistic Maxwellian $\mathcal{M}(n(t,\barx),\Temp(t,\barx),u(t,\barx);\barp)$
as in \eqref{E:Maxwelliandef}. After the analysis of Section \ref{SS:HilbertExpansion},
the main point left is to estimate solutions to the equation for the remainder \eqref{remainder}. 
We will outline the main strategy for these estimates after the statements of Lemma \ref{L2} and Lemma \ref{L0} below.

It will be useful to express the remainder as
\begin{equation}
f^\varepsilon
\eqdef
\effRE / \sqrt{\mathcal{M}} ~ .
\label{remR}
\end{equation}
We  use the notation $\| f \|_2 \eqdef \| f \|_{L^2(\mathbb{R}_{\barx}^3 \times \mathbb{R}_{\barp}^3)}$
throughout this section.  
We define the linearized relativistic Boltzmann collision operator around $\mathcal{M}$ by
\begin{gather}
\notag
\mathcal{L}(h) 
\eqdef
- 
\mathcal{M}^{-1/2}
\left\{
\mathcal{Q}\left(\sqrt{\mathcal{M}} h,\mathcal{M} \right)
+
\mathcal{Q}\left(\mathcal{M},\sqrt{\mathcal{M}} h \right)
\right\}.
\notag
\end{gather}
We also define a nonlinear operator by
\begin{gather}
\Gamma(h,f) 
\eqdef 
\mathcal{M}^{-1/2}
\mathcal{Q}\left(\sqrt{\mathcal{M}} h,\sqrt{\mathcal{M}} f \right).
\label{E:nonLinM}
\end{gather}
We recall the notation from Section \ref{SS:Notation}.

We further define the weighed $L^2(\mathbb{R}^3_{\barx}\times\mathbb{R}^3_{\barp})$ ``dissipation'' norm by 
$$
\| h \|_\nu^2 
\eqdef
\int_{\mathbb{R}^3_{\barx}}
d\barx 
\int_{\mathbb{R}^3_{\barp}}
d\barp
~\nu(\barp)~
|h(\barx,\barp)|^2.
$$
Above the the ``collision frequency'', 
$
\nu(\barp) 
\eqdef
\nu(J)(\barp),
$ 
is given by \eqref{hypNU}.  
Recall the weight function \eqref{weight}.
We will sometimes write
$
w \eqdef w(\barp) \eqdef w_1(\barp).
$
Furthermore,
\begin{equation}
h^\varepsilon
\eqdef
\effRE(t,\barx,\barp)  / \sqrt{J(\barp)}.
\label{remH}
\end{equation}
It will then be sufficient to estimate 
$
\| f^\varepsilon \|_2 (t)
$
and
$
\| h^\varepsilon \|_{\infty,\ell} (t)
$
to conclude Theorem \ref{MainThm}.  We prove the needed estimates in Lemma \ref{L2}
and Lemma \ref{L0} just below.

Let ${\bf P}$ denote the orthogonal $L^2(\mathbb{R}^3_{\barp})$ projection with respect to the null space of the linear operator $\mathcal{L}$,
which is
$$
\left\{ \sqrt{\mathcal{M}}, ~ \barp^1 \sqrt{\mathcal{M}}, ~ \barp^2 \sqrt{\mathcal{M}}, ~ \barp^3 \sqrt{\mathcal{M}}, ~ P^{0} \sqrt{\mathcal{M}} \right\}.
$$
We know from e.g. \cite{strainSOFT,MR1379589,MR1211782} 
that there exists a number $\delta_0>0$ such that
\begin{equation}
\langle \mathcal{L} h, h \rangle_{\barp}
\ge
\delta_0 
\| \{ {\bf I - P } \} h \|_\nu^2.
\label{lowerL}
\end{equation}
We will furthermore use the following $L^2$ - $L^\infty$ estimates.

\begin{lemma}
\label{L2}
($L^2$ Estimate):
We consider a smooth solution (see Remark \ref{R:Regularity}) $\big(n(t,\barx), \Temp(t,\barx), u(t,\barx)\big)$ to the relativistic Euler equations \eqref{E:nuconservationlawintro} generated by Theorem \ref{MainThmEuler}. Let $\mathcal{M}(n,\Temp,u;\barp),$
$f^\varepsilon$, $h^\varepsilon$ be defined in \eqref{E:Maxwelliandef}, \eqref{remR}, and \eqref{remH} respectively, 
and let $\delta_0 > 0$ be as in the coercivity estimate \eqref{lowerL}.
Then there exist constants $\varepsilon_0 > 0$ and $C = C(\mathcal{M} , F_0, F_1 , \ldots , F_6) > 0$,
such that for all $\varepsilon \in(0, \varepsilon_0)$ we have
\begin{equation*}
\frac{d}{dt} \| f^\varepsilon \|^2_2(t)
+
\frac{\delta_0}{2\varepsilon}
\| \{{\bf I-P}\}f^\varepsilon \|^2_\nu(t)
\le
C \{\sqrt{\varepsilon} \| \varepsilon^{3/2} h^\varepsilon \|_{\infty,\ell}(t) + 1\}
\left\{ \| f^\varepsilon \|^2_2 + \| f^\varepsilon \|_2 \right\}.
\end{equation*}
\end{lemma}

Above and below the constant $C(\mathcal{M} , F_0, F_1 , \ldots , F_6)$ depends upon the $L^2$ norms and the $L^\infty$ norms of the terms $\mathcal{M}$, $F_0$, $F_1, \ldots, F_6$ as well as their first derivatives.

\begin{lemma}
\label{L0}
($L^\infty$ Estimate):  
Under the assumptions of Lemma \ref{L2},
there exists $\varepsilon_0 > 0$ and a positive constant 
$C = C(\mathcal{M} , F_0, F_1 , \ldots , F_6) > 0$,
such that for all $\varepsilon \in(0, \varepsilon_0)$ and for any $\ell \ge 9$ we have
\begin{equation*}
\sup_{0\le s \le T}  \| \varepsilon^{3/2} h^\varepsilon \|_{\infty,\ell} (s) 
\le
C\left\{\| \varepsilon^{3/2} h_0 \|_{\infty,\ell} + \sup_{0\le s \le T} \| f^\varepsilon  \|_2(s) 
+\varepsilon^{7/2}\right\}.
\end{equation*}
\end{lemma}

As we will soon explain, these two lemmas together imply Theorem \ref{MainThm}. 
These estimates are motivated by the $L^2-L^\infty$ framework from \cite{G2}.  Similar lemmas have also been used to study the non-relativistic Hilbert expansion in \cite{MR2472156}.  Indeed, the short proof of Theorem \ref{MainThm} is 
essentially extracted from \cite{MR2472156}, modulo these lemmas.  

The main strategy is as follows. We first control the remainder equation \eqref{remainder} as in Lemma \ref{L2} using the $L^2$ energy estimates from \cite{strainSOFT} such as Lemma \ref{newNONlin} below.  To finish the proof of Theorem \ref{MainThm}, we also need $L^\infty$ estimates such as those in Lemma \ref{L0}.  These can be proven using  Duhamel's principle \eqref{duhamel} and further estimates from \cite{strainSOFT} such as Lemmas \ref{boundK2}, \ref{boundKinfX}, and \ref{sSOFT}  below. In this framework, the key idea is to control the solution in $L^\infty$ by the $L^2$ norms of the solution and the $L^\infty$ norm of the initial data. \\

\noindent {\it Proof of Theorem \ref{MainThm}}.
We will use the main estimates in Lemma \ref{L2} and Lemma \ref{L0}.  After applying the standard
 Gronwall inequality to
the differential inequality in Lemma \ref{L2},  we obtain
\begin{gather*}
\| f^\varepsilon \|^2_2(t)+1
\le 
C\left( \| f^\varepsilon \|^2_2(0)+1 \right)
e^{C  t  a(T,\varepsilon)},
\end{gather*}
where
$$
a(T,\varepsilon)
\eqdef
\sqrt{\varepsilon}  \sup_{0\le t \le T} \| \varepsilon^{3/2} h^\varepsilon \|_{\infty,\ell}(t) + 1.
 $$
 By Lemma
 \ref{L0},
 $
 a(T,\varepsilon)
\lesssim
 b(T,\varepsilon)
 $
 with
 $$
b(T,\varepsilon)
\eqdef
\sqrt{\varepsilon}
\left( \| \varepsilon^{3/2} h_0 \|_{\infty,\ell}
+
 \sup_{0\le s \le T} \| f^\varepsilon (s) \|_2
+\varepsilon^{7/2}
\right)
 + 1.
 $$
We have shown
$
\| f^\varepsilon \|_2(t)
\le 
C\left( \| f_0 \|_2+1 \right)
e^{C t b(T,\varepsilon)}.
$
Notice that 
$$
e^x \le C (1+x), 
\quad
\text{if}
\quad 
0\le x\le 1.
$$
Therefore, for $\varepsilon$ sufficiently small, on some short time interval, we have
\begin{gather*}
\| f^\varepsilon \|_2(t)
\le 
C \left( \| f_0 \|_2+1 \right)
\left\{1+ \sqrt{\varepsilon}
\left( \| \varepsilon^{3/2} h_0 \|_{\infty,\ell}
+
 \sup_{0\le s \le T} \| f^\varepsilon (s) \|_2
\right)\right\}.
\end{gather*}
Hence, there exists $\varepsilon_0>0$ such that for $0\le \varepsilon \le \varepsilon_0$ we may conclude
\begin{equation*}
\sup_{0\le s \le T} \| f^\varepsilon (s) \|_2
\le 
C_T \{1+\| f_0 \|_2+\| \varepsilon^{3/2} h_0 \|_{\infty,\ell} \}.
\end{equation*}
This procedure works for some short time interval, and then the inequality above follows in general by a continuity argument.
This last estimate and Lemma \ref{L0} together imply the main estimate in Theorem \ref{MainThm}.  
\qed \\

For the remainder of this paper, we will discuss the proofs of the Lemmas \ref{L2} and \ref{L0}.
To this end, we will use the following nonlinear estimate.

\begin{lemma}
\label{newNONlin}
For any $\ell \ge 9,$ we have the following estimate for the collision operator \eqref{E:nonLinM}:
\begin{align}\label{E:FirstEstimate}
\left| \langle \Gamma(h_1, h_2), h_3 \rangle_{\barp} \right|
\le
C \|  h_3 \|_{\infty,\ell} \| h_2 \|_{2} \| h_1 \|_{2}.
\end{align}
Furthermore, if $\chi$ is any rapidly decaying function, then we have
\begin{align} \label{E:SecondEstimate}
\left| \langle \Gamma(h_1, \chi), h_3 \rangle_{\barp} \right|
+
\left| \langle \Gamma(\chi,h_1), h_3 \rangle_{\barp} \right|
\le
 C  \| h_3 \|_{\nu} \| h_1 \|_{\nu}.
\end{align}
\end{lemma}
In the proof of this lemma below, we only require that $\chi$ satisfies the rapid decay condition
$
\left| \chi(\barp)\right| \le C (P^0)^{-m},
$
 with $m > 3$. On the other hand, in our applications below,
we will consider smooth functions $\chi$ with exponential decay, which is a property possessed by the relativistic Maxwellians 
defined in \eqref{E:Maxwelliandef}.
\\

\noindent {\it Proof of Lemma \ref{newNONlin}}. We notice
from \eqref{E:nonLinM} and \eqref{juttnerB}
 that
\begin{multline*}
 \left| 
\Gamma(h_1,h_2) 
\right| 
\lesssim
\int_{\mathbb{R}^3}  d\barq ~
\int_{\mathbb{S}^{2}} d\omega ~
~ v_{\o} ~ \sigma (\varrho,\vartheta ) ~ 
e^{-\alpha Q^0}
\left| h_1(\barp')h_2(\barq') \right| 
\\
+ \int_{\mathbb{R}^3}  d\barq ~
\int_{\mathbb{S}^{2}} d\omega ~
~ v_{\o} ~ \sigma (\varrho,\vartheta ) ~ 
e^{-\alpha Q^0}
	\left| h_1(\barp)h_2(\barq) \right| \eqdef I + II.
\end{multline*}
We estimate first the piece without post-collisional velocities, denoted $II$ above:
\begin{gather*}
 \left| 
\langle II, h_3 \rangle_{\barp} 
\right| 
\lesssim
\int_{\mathbb{R}^3}  d\barp ~
\int_{\mathbb{R}^3}  d\barq ~
\int_{\mathbb{S}^{2}} d\omega ~
~ v_{\o} ~ \sigma (\varrho,\vartheta ) ~ 
e^{-\alpha Q^0}
\left| h_1(\barp)h_2(\barq) h_3(\barp) \right| 
\\
\lesssim
\|  h_3 \|_{\infty,\ell}
\int_{\mathbb{R}^3}  d\barp ~
\int_{\mathbb{R}^3}  d\barq ~
\int_{\mathbb{S}^{2}} d\omega ~
~ v_{\o} ~ \sigma (\varrho,\vartheta ) ~ 
\frac{\left| h_1(\barp)h_2(\barq) \right| }{(Q^0)^\ell (P^0)^\ell}.
\end{gather*}
Above, we made use of the trivial estimate
$
e^{-\alpha Q^0}
\lesssim
(Q^0)^{-\ell}.
$
Applying the Cauchy-Schwarz inequality, we conclude that
\begin{gather*}
 \left| 
\langle II, h_3 \rangle_{\barp} 
\right|
\lesssim
\|  h_3 \|_{\infty,\ell}
\prod_{i=1,2}\left(
\int_{\mathbb{R}^3}  d\barp ~
\int_{\mathbb{R}^3}  d\barq ~
\int_{\mathbb{S}^{2}} d\omega ~
~ v_{\o} ~ \sigma (\varrho,\vartheta ) ~ 
\frac{\left| h_i(\barp) \right|^2 }{(Q^0)^\ell (P^0)^\ell}
\right)^{1/2}.
\end{gather*}
We use the $(\barp,\barq)$ symmetry to interchange the values of $\barp$ and $\barq$ in the integral involving $h_2$ above.
Since $\ell \ge 9,$
we observe from our hypothesis above \eqref{hypNU} that 
\begin{gather*}
(P^0)^{-\ell}
\int_{\mathbb{R}^3}  d\barq ~
\int_{\mathbb{S}^{2}} d\omega ~
~ ~ 
\frac{v_{\o} ~ \sigma (\varrho,\vartheta ) }{(Q^0)^{\ell} }
\lesssim 1.
\end{gather*}
Estimates of this type are proven for instance in \cite[Lemma 3.1]{strainSOFT}.
>From here, the first estimate in Lemma \ref{newNONlin} follows for term $II$.
Similarly, for term $I$ we have
\begin{gather*}
 \left| 
\langle I, h_3 \rangle_{\barp} 
\right| 
\lesssim
\int_{\mathbb{R}^3}  d\barp ~
\int_{\mathbb{R}^3}  d\barq ~
\int_{\mathbb{S}^{2}} d\omega ~
~ v_{\o} ~ \sigma (\varrho,\vartheta ) ~ 
e^{-\alpha Q^0}
\left| h_1(\barp')h_2(\barq') h_3(\barp) \right| 
\\
\lesssim
\|  h_3 \|_{\infty,\ell}
\prod_{i=1,2}\left(
\int_{\mathbb{R}^3}  d\barp ~
\int_{\mathbb{R}^3}  d\barq ~
\int_{\mathbb{S}^{2}} d\omega ~
~ v_{\o} ~ \sigma (\varrho,\vartheta ) ~ 
\frac{\left| h_i(\barp') \right|^2 }{(Q^0)^{\ell} (P^0)^{\ell}}
\right)^{1/2}.
\end{gather*}
Above, we used the $(\barp',\barq')$ symmetry to interchange the values of $\barp'$ and $\barq'$ in the integral involving $h_2(\barp')$ above. By the pre-post collisional change of variables,
which is
$
d\barp d\barq = \frac{P'^{0} Q'^{0}}{P^0 Q^0} d\barp' d\barq', 
$ 
we have
\begin{multline*}
\int_{\mathbb{R}^3}  d\barp ~
\int_{\mathbb{R}^3}  d\barq ~
\int_{\mathbb{S}^{2}} d\omega ~
~ v_{\o} ~ \sigma (\varrho,\vartheta ) ~ 
\frac{\left| h_i(\barp') \right|^2 }{(Q^0)^{\ell} (P^0)^{\ell}}
\\
=
\int_{\mathbb{R}^3}  d\barp ~
\int_{\mathbb{R}^3}  d\barq ~
\int_{\mathbb{S}^{2}} d\omega ~
~ v_{\o} ~ \sigma (\varrho,\vartheta ) ~ 
\frac{\left| h_i(\barp) \right|^2 }{(Q'^0)^\ell (P'^0)^\ell}.
\end{multline*}
Above, we used the fact that that the kernel of the integral is invariant with respect to the relativistic 
pre-post collisional change of variables from \cite{MR1105532}.

Now it can be seen that $P^0 \lesssim Q'^0 P'^0$
and
$Q^0 \lesssim Q'^0 P'^0$.  This is the content of \cite[Lemma 2.2]{MR1211782}.  From here,
since $\ell \ge 9$
we observe that there exists a small $\delta >0$ such that
\begin{gather*}
\int_{\mathbb{R}^3}  d\barq ~
\int_{\mathbb{S}^{2}} d\omega ~
~ ~ 
\frac{v_{\o} ~ \sigma (\varrho,\vartheta ) }{(Q'^0)^\ell (P'^0)^\ell}
\le
C
(P^0)^{-\ell/2 +\delta}
\int_{\mathbb{R}^3}  d\barq ~
\int_{\mathbb{S}^{2}} d\omega ~
~ ~ 
\frac{v_{\o} ~ \sigma (\varrho,\vartheta ) }{(Q^0)^{\ell/2 +\delta}}
\lesssim 1.
\end{gather*}
This establishes \eqref{E:FirstEstimate}.

To prove \eqref{E:SecondEstimate}, we choose a cut-off function $\chi$ satisfying $\chi(\barp) \lesssim (P^0)^{-m}$ for any $m > 3.$ 
We will estimate the $I$ and $II$ terms from the top of this proof with $h_2 = \chi$ in the first case. 
In particular, as before, we have that
\begin{multline*}
 \left| 
\langle II, h_3 \rangle_{\barp} 
\right| 
\le 
C\int_{\mathbb{R}^3}  d\barp ~
\int_{\mathbb{R}^3}  d\barq ~
\int_{\mathbb{S}^{2}} d\omega ~
~ v_{\o} ~ \sigma (\varrho,\vartheta ) ~ 
e^{-\alpha Q^0}
\left| h_1(\barp) h_3(\barp) \right| 
\\
\lesssim
 \| h_3 \|_{\nu} \| h_1 \|_{\nu}.
\end{multline*}
We also have
\begin{gather*}
 \left| 
\langle I, h_3 \rangle_{\barp} 
\right| 
\le
C
\int_{\mathbb{R}^3}  d\barp ~
\int_{\mathbb{R}^3}  d\barq ~
\int_{\mathbb{S}^{2}} d\omega ~
~ v_{\o} ~ \sigma (\varrho,\vartheta ) ~ 
e^{-\alpha Q^0} (Q'^0)^{-m}
\left| h_1(\barp') h_3(\barp) \right| 
\\
\le
C \prod_{j\in{1,3}}
\left(
\int_{\mathbb{R}^3}  d\barp ~
\int_{\mathbb{R}^3}  d\barq ~
\int_{\mathbb{S}^{2}} d\omega ~
~ v_{\o} ~ \sigma (\varrho,\vartheta ) ~ 
(Q^0)^{-m} (Q'^0)^{-m}
\left| h_j(\barp') \right|^2 
\right)^{1/2},
\end{gather*}
where we have used similar reasoning as we did for the previous terms.
After the pre-post change of variables, as before, we conclude that
$
 \left| 
\langle I, h_3 \rangle_{\barp} 
\right| 
\le 
C
 \| h_3 \|_{\nu} \| h_1 \|_{\nu}.
 $ 
 as desired.
We have thus shown that 
$
\left| \langle \Gamma(h_1, \chi), h_3 \rangle_{\barp} \right|
\lesssim
 \| h_3 \|_{\nu} \| h_1 \|_{\nu}.
$
The last estimate involving the term $\Gamma(\chi, h_1)$ follows in exactly the same way.
\qed \\

With this lemma, we are ready for the \\

\noindent {\it Proof of Lemma \ref{L2}}.  Since $\effRE$ satisfies \eqref{remainder},
the  function $f^{\varepsilon }$ from \eqref{remR} satisfies the equation
\begin{multline*}
\partial _{t}f^{\varepsilon }+\phat\cdot\partial _{\barx}f^{\varepsilon }+\frac{1}{%
\varepsilon }{\mathcal{L}}(f^{\varepsilon}) 
=
\left(\frac{\{\partial _{t}+\phat\cdot \partial _{\barx}\}\sqrt{\mathcal{M} }}{\sqrt{\mathcal{M} }}\right)
f^{\varepsilon }
+
\varepsilon ^{2}\Gamma \left(f^{\varepsilon
},f^{\varepsilon
}\right)
\\
+
\sum_{i=1}^6\varepsilon^{i-1}
\left\{\Gamma \left({F_{i}}/{\sqrt{\mathcal{M} }}%
,f^{\varepsilon }\right)
+
\Gamma \left(f^{\varepsilon },{F_{i}}/{\sqrt{\mathcal{M}
}}\right)\right\}
+
\varepsilon ^{2}\bar{A},
\end{multline*}
where 
$$
\bar{A}=
- \varepsilon \mathcal{M}^{-1/2}{\{\partial _{t}+\phat\cdot \partial _{\barx}\}F_{6}}
+
\mathop{\sum_{i + j > 6}}_{1 \leq i, \hspace{.02 in}j \leq 6}
\varepsilon^{i+j-6}
\Gamma \left({F_{i}}/{\sqrt{\mathcal{M} }},{F_{j}}/{\sqrt{\mathcal{M} }}\right).
$$
We now take the $L^{2}$ inner product (in $\mathbb{R}_{\barx}^3 \times \mathbb{R}_{\barp}^3$) of this equation with $f^{\varepsilon}$. It follows from \eqref{lowerL} that
$$
\left\langle
\partial _{t}f^{\varepsilon }+\phat\cdot\partial _{\barx}f^{\varepsilon }+\frac{1}{%
\varepsilon }{\mathcal{L}} (f^{\varepsilon }),
f^{\varepsilon }
\right\rangle_{\barx;\barp}
\ge
\frac{1}{2}
\frac{d}{dt} \| f^{\varepsilon } \|_2^2
+
\frac{\delta_0}{\varepsilon} \|\{{\bf I-P}\} f^{\varepsilon } \|_\nu^2.
$$
In the remainder of the proof we will give upper bounds for the other terms.
We point out that
$
\mathcal{M}^{-1/2}
\{\partial _{t}+\phat\cdot \partial _{\barx}\}\sqrt{\mathcal{M}}
$
is a first order polynomial in $p$.  Furthermore,
$$
\left|
\frac{\{\partial _{t}+\phat\cdot \partial _{\barx}\}\sqrt{\mathcal{M} }}{\sqrt{\mathcal{M} }}
\right|
\le w_1(\barp) ~ C(n,u,\theta).
$$
Here $C=C(n,u,\theta)$ depends upon the the $L^{\infty}$ norm of the first derivatives of the fluid variables.  
Choose ${\tilde{a}}\eqdef 2/(4-b)$, where $b$ is any number subject to the constraints described just above \eqref{hypNU}.  
With that, for any $\kappa >0,$ we obtain
\begin{multline*}
\left\langle \left(\frac{\{\partial _{t}+\phat\cdot \partial _{\barx}\}\sqrt{\mathcal{M} }}{\sqrt{%
\mathcal{M} }}\right)f^{\varepsilon },f^{\varepsilon }\right\rangle_{\barx;\barp}
=
\int_{w_1(\barp)\geq \frac{\kappa }{{\varepsilon^{\tilde{a}} }}}~d\barx ~ d\barp
+
\int_{w_1(\barp)\leq \frac{\kappa }{{\varepsilon^{\tilde{a}} }}} ~d\barx ~ d\barp
\\
\leq 
C_2(n,u,\theta)
\| w_1 f^{\varepsilon }\|_{L^\infty({w(\barp)\geq
\frac{\kappa }{{\varepsilon^{\tilde{a}} }}}) }
\|f^{\varepsilon}\|_{L^2} 
+
C_\infty(n,u,\theta) \|\sqrt{w_1}f^{\varepsilon }
\|_{L^2({w(\barp)  \leq \frac{\kappa }{{\varepsilon^{\tilde{a}} }}})}^{2}.
\end{multline*}
Above, $C_2(n,u,\theta)$ denotes a constant which depends upon $L^2$ norms of the first derivatives of the fluid variables and the $L^\infty$ norms of the fluid variables alone.  This estimate follows from the Cauchy-Schwarz inequality.  The constant $C_\infty(n,u,\theta)$ also depends only on the $L^\infty$ norms of the fluid variables and their first-order derivatives.  
Furthermore, from \eqref{juttnerB}, \eqref{remH} and \eqref{remR}, we observe that
$$
|w_1(\barp)|
\left|  f^{\varepsilon }(\barp) \right|
\le
w_{1-\ell }(\barp)
\left| 
w_{\ell }(\barp)
 h^{\varepsilon }
\right|
\leq
\frac{\varepsilon^2}{\kappa^{2/{\tilde{a}}}}
\left| w_{\ell }(\barp) h^{\varepsilon } \right|,
\quad
\text{if}
\quad
w_1(\barp)\ge \frac{\kappa}{\varepsilon^{\tilde{a}} }.
$$ 
The above estimate utilizes $\ell -1 \geq 8 \ge
(4-b)
=
2/{\tilde{a}}$ and  $J<C\mathcal{M}$ in
\eqref{juttnerB}. 
Additionally, since
$
\frac{1}{{\tilde{a}} } 
=
2-\frac{b}{2},
$
we have
the following weight estimates:
$$
w  
= 
w^{1/{\tilde{a}}} w^{b/2-1}
\le
\frac{\kappa^{1/{\tilde{a}}}}{\varepsilon} w^{b/2-1}
\lesssim
\frac{\kappa^{1/{\tilde{a}}}}{\varepsilon} (P^0)^{b/2},
\quad
\text{if}
\quad
w(\barp)\le \frac{\kappa}{\varepsilon^{\tilde{a}}}.
$$
With these two computations, we estimate
\begin{multline*}
\left|
\left\langle \left(\frac{\{\partial _{t}+\phat\cdot \partial _{\barx}\}\sqrt{\mathcal{M} }}{\sqrt{%
\mathcal{M} }}\right)f^{\varepsilon },f^{\varepsilon }\right\rangle_{\barx;\barp}
\right|
\\
\lesssim
\frac{\varepsilon ^{2}}{\kappa^{2/{\tilde{a}} }}\|h^{\varepsilon }\|_{\infty,\ell}
\|f^{\varepsilon }\|_{2}
+
C_\infty(n,u,\theta) \|\sqrt{w}f^{\varepsilon }
\|_{L^2({w(\barp)\leq \frac{\kappa }{{\varepsilon^{\tilde{a}}  }}})}^{2}
\\
\leq C_{\kappa }\varepsilon^{2}
\|h^{\varepsilon }\|_{\infty,\ell}
\|f^{\varepsilon }\|_{2}
+C
\|\sqrt{w}\mathbf{P}f^{\varepsilon }\|_{L^2({w(\barp)\leq \frac{\kappa }{{\varepsilon^{\tilde{a}}  }}})}^{2}
+
C\|\sqrt{w}\{\mathbf{I}-\mathbf{P}\}f^{\varepsilon }\|_{L^2({w(\barp)\leq \frac{\kappa }{{\varepsilon^{\tilde{a}}  }}})}^{2}
\\
\leq 
C_{\kappa }\varepsilon ^{2}\|h^{\varepsilon }\|_{\infty,\ell}
\|f^{\varepsilon }\|_{2}+C\|f^{\varepsilon }\|_{2}^{2}+
C
\frac{\kappa^{1/{\tilde{a}} }}{\epsilon}
\|\{\mathbf{I}-\mathbf{P}\}f^{\varepsilon }\|_{\nu}^{2}.
\end{multline*}
Choosing $\kappa$ to be sufficiently small, these are the estimates that we will use for the term involving derivatives of the local Maxwellian $\mathcal{M}.$

We use Lemma \ref{newNONlin}  and \eqref{remH} to obtain for $\ell \ge 9$ that
\begin{equation*}
\left| 
\varepsilon^{2}
\langle \Gamma (f^{\varepsilon },f^{\varepsilon}),f^{\varepsilon }\rangle_{\barx;\barp}
\right|
\lesssim
\sqrt{\varepsilon }
\|\varepsilon ^{3/2}h^{\varepsilon }\|_{\infty,\ell}
\|f^{\varepsilon }\|_{2}^{2}.
\end{equation*}
Above we have also used \eqref{juttnerB} with $J \le C \mathcal{M}$.
We utilize the second estimate in Lemma \ref{newNONlin} and Proposition \ref{iterateBDS} to achieve
\begin{multline*}
\left|
\sum_{i=1}^6\varepsilon^{i-1}
\left\{\left\langle \Gamma \left({ F_{i}}/{\sqrt{\mathcal{M} }},f^{\varepsilon }\right),f^{\varepsilon } \right\rangle_{\barx;\barp}
+
\left\langle \Gamma
\left(f^{\varepsilon },{ F_{i}}/{\sqrt{\mathcal{M} }}\right),f^{\varepsilon }\right\rangle_{\barx;\barp}\right\}
\right|
\\
\lesssim
\sum_{i=1}^6\varepsilon^{i-1}
\|f^{\varepsilon }\|_{\nu }^{2}
\\
\lesssim
\|\mathbf{P}f^{\varepsilon }\|_{\nu }^{2}
+
\|\{\mathbf{I}-\mathbf{P}\}f^{\varepsilon }\|_{\nu }^{2}
\lesssim
\|f^{\varepsilon }\|_{L^2}^{2}
+
\frac{\varepsilon}{\varepsilon}\|\{\mathbf{I}-\mathbf{P}\}f^{\varepsilon }\|_{\nu }^{2}.
\end{multline*}
We have used Proposition \ref{iterateBDS}, \eqref{juttner}, and \eqref{juttnerB} to conclude that
$
{F_{i}}/{\sqrt{\mathcal{M} }}
$
is a rapidly decaying function as in the statement of Lemma \ref{newNONlin}.
Similar to the above estimates, we have $\langle \varepsilon ^{2}\bar{A},f^{\varepsilon }\rangle_{\barp}
\lesssim
 \varepsilon ^{2} \|f^{\varepsilon }\|_{L^2}
\lesssim
 \|f^{\varepsilon }\|_{L^2}.$ 
Note that $\frac{\varepsilon}{\varepsilon} =1$ was added above so that we can obtain the estimate in Lemma \ref{L2} by absorbing this term into the l.h.s. of the inequality for $\varepsilon > 0$ sufficiently small.
In particular, we conclude our estimate  by choosing 
$\kappa$ small and then supposing that $\varepsilon \le \frac{\delta_0}{2C}$.
\qed \\

We are now ready to consider the $L^{\infty }$ estimate for $h^{\protect\varepsilon }$ in Lemma \ref{L0}.  We first expand
\begin{equation*}
-{J}^{-1/2}\{\mathcal{Q}(\mathcal{M} ,\sqrt{J}h)
+
\mathcal{Q}(\sqrt{J}h,\mathcal{M} )\}
=
\nu (\mathcal{M} )h-K(h),
\end{equation*}%
where $K=K_2-K_1$. This will be an important term in the equation for the remainder \eqref{remainder} once we plug in the ansatz \eqref{remH}.  This is observed in the proof of Lemma \ref{L0}. The operators $K_1(h)$ and $K_2(h)$ are defined as
\begin{align*}
K_1(h) 
&\eqdef
{J}^{-1/2}\mathcal{Q}^-(\mathcal{M} ,\sqrt{J}h),
\\
K_2(h) 
&\eqdef
{J}^{-1/2}\{\mathcal{Q}^+(\mathcal{M} ,\sqrt{J}h)
+
\mathcal{Q}^+(\sqrt{J}h,\mathcal{M} )\},
\end{align*}
while 
\begin{align*}
\nu (\mathcal{M})
	&\eqdef {J}^{-1/2}\mathcal{Q}^-(\sqrt{J},\mathcal{M} )
	= \mathcal{Q}^-(1,\mathcal{M}).
\end{align*}
Above, the operators $\mathcal{Q}^\pm$ are the usual gain and loss parts of $\mathcal{Q}$ from \eqref{collisionCM}.
More specifically, the operators $K_{i}$ can be expressed as
\begin{align*}
K_1(h)
 & \eqdef 
\int_{\mathbb{R}^3\times \mathbb{S}^{2}}
~ d\omega d\barq
~ v_{\o} ~ \sigma(\varrho,\vartheta)
~ \left\{
\sqrt{J(\barq)}\frac{\mathcal{M}(\barp)}{\sqrt{J(\barp)}}h(\barq)
\right\},
\\
K_2(h)
 & \eqdef 
\int_{\mathbb{R}^3\times \mathbb{S}^{2}} ~ d\omega d\barq
~ v_{\o} ~ \sigma(\varrho,\vartheta) ~
\left\{
\mathcal{M}(\barq')\frac{\sqrt{J(\barp')}}{\sqrt{J(\barp)}} 
h(\barp')
\right\}
\\
& \ \ \ +
\int_{\mathbb{R}^3\times \mathbb{S}^{2}} ~ d\omega d\barq
~ v_{\o} ~ \sigma(\varrho,\vartheta) ~
\left\{
\mathcal{M}(\barp')\frac{\sqrt{J(\barq')}}{\sqrt{J(\barp)}} 
h(\barq')
\right\}. 
\end{align*}
Given any small number $\eta>0$, we choose a smooth cut-off function $\chi =\chi (\varrho)$ satisfying
\begin{equation}
\chi (\varrho)=
\left\{
\begin{array}{cl}
 1 & {\rm if } ~~  \varrho \ge 2\eta, 
 \\
  0 &
  {\rm if } ~~ \varrho \le \eta.  
\end{array}
\right.
\notag
\end{equation}
This Lorentz invariant cut-off function was previously used in \cite{strainSOFT}.
We now define
\begin{gather}
K_2^{1-\chi}(h) 
\eqdef
\int_{\mathbb{R}^3\times \mathbb{S}^{2}} ~ d\omega d\barq
~ \left( 1 - \chi(\varrho) \right)
~ v_{\o} ~ \sigma(\varrho,\vartheta) ~
\mathcal{M}(\barq')\frac{\sqrt{J(\barp')}}{\sqrt{J(\barp)}} 
h(\barp')
\notag
\\
+
\int_{\mathbb{R}^3\times \mathbb{S}^{2}} ~ d\omega d\barq
~ \left( 1 - \chi(\varrho) \right)
~ v_{\o} ~ \sigma(\varrho,\vartheta) ~
\mathcal{M}(\barp')\frac{\sqrt{J(\barq')}}{\sqrt{J(\barp)}} 
h(\barq').
\label{cut2K}
\end{gather}
We also define $K_1^{1-\chi}(h)$ in the same way.
We will use the splitting
$
K
=
K^{1-\chi}
+
K^{\chi}.
$
We will use the Hilbert-Schmidt form for $K^\chi$, it is given by
$$
K^\chi (h) = 
\int_{\mathbb{R}^3}d\barq~ k^\chi(\barp,\barq) ~ h(\barq),
\quad
(i=1,2).
$$
This form is computed explicitly in \cite[Appendix]{strainSOFT}; see also \cite{sdGsvLwvW1980}. Since the closed form expression for the kernel $k^\chi(\barp,\barq)$ can be quite complicated, we simply state the following useful estimate:

\begin{lemma}
\label{boundK2}  
\cite[Lemma 3.2]{strainSOFT}.   
There exists a constant 
$
\zeta
>0
$
such that the kernel enjoys the estimate
$$
 \left| k^{\chi}(\barp,\barq) \right|
\lesssim
 \left( P^0 Q^0 \right)^{-\zeta}\left( C_1+(P^0 + Q^0)^{-(\beta)_-/2} \right)
e^{-c |\barp- \barq|},
$$
where 
$
(\beta)_- = \max\{ - \beta, 0\},
$
$\beta$ is the parameter, and $C_1$ is the constant from above \eqref{hypNU}.
\end{lemma}

Note that in the notation of this paper we use $P^0$ to denote the quantity which is called  $p_0$ in \cite{strainSOFT}.
The estimate above is proved in \cite[Lemma 3.4]{strainSOFT} for the case of soft potentials, i.e. $(\beta)_- =b,$ $a= 0,$ 
and $C_1 = 0$ in \eqref{hypNU}.  However the generalization to the case above is immediate and follows directly from the proof. In the case above, we have the exact estimate
$
\zeta = \min\{1-(a+(\gamma)_-)/2, ~ \min\left\{2-|\gamma|, 4-b,2\right\}/4\}.
$
We will now quote the estimate of the operator $K^{1-\chi}$ as follows:

\begin{lemma}
\label{boundKinfX}  
\cite[Lemma 4.6]{strainSOFT}. 
Fix $\ell \ge 0$.
Then given any small $\eta >0$, we have
$$
\left| w_\ell(\barp) K^{1-\chi}(h(\barp)) \right| 
\le
\eta e^{- \gamma P^0} \| h\|_{\infty}.
$$
Above the constant $\gamma >0$ is independent of $\eta$.
\end{lemma}

We will also use the following nonlinear estimate:

\begin{lemma}
\label{sSOFT}
\cite[Lemma 5.2]{strainSOFT}.
For any $\ell \geq 0,$
we have the following $L^\infty$ estimate for the nonlinear Boltzmann collision operator:
$$
\left| w_\ell(\barp)
{J}^{-1/2}(\barp)
\mathcal{Q}\left(h_1\sqrt{J},h_2\sqrt{J}\right)(\barp)
\right| 
\lesssim
 \nu(\barp) 
\|  h_1 \|_{\infty,\ell} \| h_2 \|_{\infty,\ell}.
$$
\end{lemma}

We are now ready to prove Lemma \ref{L0}. \\

\noindent {\it Proof of Lemma \ref{L0}}.
From \eqref{remainder} and \eqref{remH}, we obtain
\begin{multline*}
\partial _{t}h^{\varepsilon }+\phat\cdot \partial _{\barx}h_{{}}^{\varepsilon }
+
\frac{\nu(\mathcal{M})}{\varepsilon }h^{\varepsilon}
-
\frac{1}{\varepsilon }K(h^{\varepsilon })
=
\frac{\varepsilon^{2}}{\sqrt{J}}
\mathcal{Q}\left(h^{\varepsilon }\sqrt{J},h^{\varepsilon }\sqrt{J}\right)
\\
+
\sum_{i=1}^6
\varepsilon^{i-1}\frac{1}{\sqrt{J}} 
\left\{
\mathcal{Q}\left(F_i,h^{\varepsilon }\sqrt{J}\right)
+
\mathcal{Q}\left(h^{\varepsilon }\sqrt{J}, F_i\right)\right\}
+
\varepsilon ^{2}\tilde{A},
\end{multline*}
with
$$
\tilde{A}\eqdef-\frac{\varepsilon}{\sqrt{J}}\{\partial _{t}+\phat\cdot \partial _{\barx}\}F_{6}
+
\sum_{i+j > 6,i\leq 6,j\leq 6}\varepsilon^{i+j-6}
\frac{1}{\sqrt{J}}\mathcal{Q}(F_{i},F_{j}).
$$
We will prove Lemma \ref{L0} by iterating twice the representation of solutions to this equation in terms of the Duhamel formula.

For purposes of the proof, we define $y_1 \eqdef \barx - \phat(t-s)$.   Let $K = K^{1-\chi} + K^{\chi}$ be the splitting 
from \eqref{cut2K} and define 
$$
\widetilde{\nu}(t,s)
\eqdef
\int_s^t ~ d \tau ~ \nu (\mathcal{M} )(\tau,\barx ).
$$
Furthermore, using \eqref{juttnerB} we have
$
\nu (J )(\barp) \approx
\nu(\mathcal{M})(\barp).
$
>From this and \eqref{hypNU}, we have that
$$
\widetilde{\nu}(t,s)
\approx 
\nu (J)(\barp) (t-s)
\approx 
(P^0)^{\beta/2} (t-s).
$$
Here and below $A\approx B$ indicates that $\exists \ C \geq 1$ such that 
$\frac{1}{C} A \le B \le C A$.

Now by the Duhamel formula we have
\begin{multline}
 \label{duhamel}
h^{\varepsilon } (t,\barx,\barp) = 
 \exp\Big(-\frac{\widetilde{\nu}(t,0) }{\varepsilon } \Big) ~
h^{\varepsilon}_0(\barx - \phat t,\barp)
 \\
 +
\frac{1}{\varepsilon }
\int_{0}^{t} ~ ds ~ 
\exp\Big(-\frac{\widetilde{\nu}(t,s) }{\varepsilon } \Big)
K^{1-\chi}(h^{\varepsilon }) (s,y_1,\barp)
 \\
+
\frac{1}{\varepsilon }
\int_{0}^{t} ~ ds ~ 
\exp\Big(-\frac{\widetilde{\nu}(t,s) }{\varepsilon } \Big)
K^{\chi}(h^{\varepsilon }) (s,y_1,\barp)
 \\
 +\int_{0}^{t} ~ ds ~ 
e^{-\frac{\widetilde{\nu}(t,s)}{\varepsilon }   }
\frac{\varepsilon ^{2} }{\sqrt{J}}
\mathcal{Q}\left(h^{\varepsilon }\sqrt{J},h^{\varepsilon }\sqrt{J}
\right)(s,y_1,\barp)
 \\
 + \int_{0}^{t} ~ ds ~ 
\exp\Big(-\frac{\widetilde{\nu}(t,s)}{\varepsilon }   \Big)
\sum_{i=1}^6 \varepsilon^{i-1}\frac{1}{\sqrt{J}}
\left\{\mathcal{Q}\left(F_i,h^{\varepsilon }\sqrt{J}\right) +
 \mathcal{Q}\left(h^{\varepsilon }\sqrt{J}, F_i\right)\right\}(s,y_1,\barp)
 \\
+\int_{0}^{t} ~ ds ~ \exp\Big(-\frac{\widetilde{\nu}(t,s)}{\varepsilon} \Big) \varepsilon ^{2}%
\tilde{A}(s,y_1,\barp).
\end{multline}
We will now simply write $\widetilde{\nu}(t,s)$ as
$
\nu  (t-s).
$
We will use the following basic estimate several times below
\begin{gather*}
\int_{0}^{t}~ ds ~ \exp\Big(-\frac{\nu (t-s)}{\varepsilon }\Big) ~ \nu 
\lesssim
\varepsilon .
\end{gather*}
We will now estimate each of the terms in \eqref{duhamel}.

Given $\eta >0$, we recall the splitting 
$K = K^{1-\chi} + K^{\chi}$
as in \eqref{cut2K}.
>From  Lemma \ref{boundKinfX}, for any $\eta>0$, the $K^{1-\chi}$ term in \eqref{duhamel} multiplied by $w_\ell$ is bounded
as
\begin{multline*}
\left|
\frac{w_\ell}{\varepsilon }
\int_{0}^{t} ~ ds ~ 
\exp\Big(-\frac{\widetilde{\nu}(t,s) }{\varepsilon } \Big)
K^{1-\chi} (h^{\varepsilon }) (s,y_1,\barp)
\right|
\\
\lesssim
\frac{\eta}{2}
\sup_{0\leq t\leq T }||h^{\varepsilon }(t)||_{\infty,\ell}
\int_{0}^{t} ~ ds ~ \exp\Big(-\frac{\nu (t-s)}{\varepsilon
}\Big) \nu ~
\lesssim
\eta~\varepsilon\sup_{0\leq t\leq T
}||h^{\varepsilon }(t)||_{\infty}.
\end{multline*}
Next, we use Lemma \ref{sSOFT} to conclude that
$$
\left|\frac{w_\ell}{\sqrt{J}}
\mathcal{Q}\left(h^{\varepsilon }\sqrt{J},h^{\varepsilon }\sqrt{J}\right)
\right|
\lesssim
\nu(\barp) ~ \|h^{\varepsilon }\|_{\infty,\ell }^{2}.
$$  
Then the fourth line in \eqref{duhamel} is
bounded by
\begin{gather}
C\varepsilon ^{2}\int_{0}^{t} ~ ds ~ \exp\Big(-\frac{\nu (t-s)}{\varepsilon }\Big)
\nu ~\|h^{\varepsilon }(s)\|_{\infty,\ell }^{2}   
\label{h2} 
\lesssim
\varepsilon ^{3}\sup_{0\leq s\leq t}\|h^{\varepsilon
}(s)\|_{\infty,\ell }^{2}.
\end{gather}
For the following terms in \eqref{duhamel}, from Lemma \ref{sSOFT} again, 
we have
\begin{multline*}
\left|
\sum_{i=1}^6 \varepsilon^{i-1}\frac{w_\ell}{\sqrt{J}}
\left\{ 
\mathcal{Q}\left(F_i,h^{\varepsilon }\sqrt{J}\right)
+
\mathcal{Q}\left(h^{\varepsilon }\sqrt{J},F_i\right)\right\}
\right|
\\
\le
C
\sum_{i=1}^6 ~ \varepsilon^{i-1} ~
\nu (\barp) ~ \|h^{\varepsilon }\|_{\infty,\ell }~ \left\| F_{i} / \sqrt{J} \right\|_{\infty,\ell }
\le
C \nu (\barp) ~ \|h^{\varepsilon }\|_{\infty,\ell }.
\end{multline*}
We have used Proposition \ref{iterateBDS} and the upper bound in \eqref{juttnerB}
to conclude that 
$
\left\| F_{i}/ \sqrt{J} \right\|_{\infty,\ell }
\le  C.
$
Thus, the fifth line in \eqref{duhamel} is bounded by
\begin{equation}
\int_{0}^{t} ~ ds ~ \exp\Big(-\frac{\nu (t-s)}{\varepsilon }\Big)\nu  \|h^{\varepsilon }(s)\|_{\infty,\ell }
~
\le
C
\varepsilon
\sup_{0\leq t\leq T}  \|h^{\varepsilon }(t)\|_{\infty,\ell }.
\label{h1}
\end{equation}
As above, the last line in (\ref{duhamel}) is clearly bounded by $C\varepsilon^{3}.$

We have now estimated all the terms in \eqref{duhamel} save one, which we denote by $I_{\gamma},$ and which is defined by
\begin{gather*}
I_{\gamma} \eqdef 
 \frac{1}{\varepsilon }
 \int_{0}^{t} ~ ds ~ 
\exp\Big(-\frac{\widetilde{\nu}(t,s) }{\varepsilon } \Big)
K^{\chi}\left(h^{\varepsilon }\right) (s,y_1,\barp). 
\end{gather*}
Collecting the above estimates,
 we have established
\begin{multline*}
\sup_{0\le t \le T} \| \varepsilon^{3/2} h^\varepsilon (t) \|_{\infty ,\ell}
\le
C\varepsilon \sup_{0\le t \le T} \| \varepsilon^{3/2} h^\varepsilon (t)  \|_{\infty, \ell} 
+
C\varepsilon^{3} \sup_{0\le t \le T} \| \varepsilon^{3/2} h^\varepsilon (t) \|_{\infty, \ell}^2
+C\varepsilon^{3}
\\
+
C \| \varepsilon^{3/2} h_0 \|_{\infty, \ell} 
+ C w_\ell(\barp) \varepsilon^{3/2} I_{\gamma}. 
\end{multline*}
To bound $I_{\gamma},$ we will input $h^{\varepsilon }$ in the form of
\eqref{duhamel} back into $I_{\gamma}$ just below.

We recall that the kernel of $K^{\chi}(h)$ is $k^{\chi}(\barp,\barq)$.  With this notation,
it follows that
\begin{gather*}
I_{\gamma}
=
\int_{0}^{t} ~ ds ~ 
\exp\Big(-\frac{\widetilde{\nu}(t,s) }{\varepsilon } \Big) \frac{1}{\varepsilon }
\int_{\mathbb{R}^3} ~ d\barq ~
 k^{\chi} (\barp,\barq) ~ h^{\varepsilon }  (s,y_1,\barq).
\end{gather*}
We plug \eqref{duhamel} for $h^\varepsilon$ into the above to obtain
\begin{multline}
 \label{duhamelD}
I_{\gamma}
=
\int_{0}^{t} ~ ds ~ 
\exp\Big(-\frac{\widetilde{\nu}(t,s) }{\varepsilon } \Big) \frac{1}{\varepsilon }
\int_{\mathbb{R}^3} ~ d\barq ~
  k^{\chi} (\barp,\barq) ~ 
  \\
  \times
 \exp\Big(-\frac{\widetilde{\nu}'(\barq)(s,0) }{\varepsilon } \Big) ~
h^{\varepsilon}_0(\barx - \phat(t-s)-\qhat s,\barq)
 +
I_{\gamma,\gamma}
+
H_{\gamma}.
\end{multline}
We also introduce the notation
$
y_2 
\eqdef
\barx - \phat(t-s)-\qhat(s-s').
$
Furthermore, we define
\begin{multline*}
I_{\gamma,\gamma}
\eqdef
\int_{0}^{t} ~ ds ~ 
\exp\Big(-\frac{\widetilde{\nu}(\barp)(t,s) }{\varepsilon } \Big) \frac{1}{\varepsilon }
\int_{\mathbb{R}^3} ~ d\barq ~
k^{\chi} (\barp,\barq) ~
  \\
  \times
\int_{0}^{s} ~ ds' ~ \exp\Big(-\frac{\widetilde{\nu}'(\barq)(s,s')}{\varepsilon }   \Big)
 \frac{1}{\varepsilon }K^{\chi} \left( h^{\varepsilon }\right) (s',y_2,\barq).
\end{multline*}
Additionally, the term $H_{\gamma}$ can be expressed as
\begin{multline}
 \label{duhamelDH}
H_{\gamma}
=
\int_{0}^{t}  ds ~ 
\exp\Big(-\frac{\widetilde{\nu}(\barp)(t,s) }{\varepsilon } \Big) \frac{1}{\varepsilon }
\int_{\mathbb{R}^3}  d\barq  ~ k^{\chi} (\barp,\barq)
\\
\times
 \int_{0}^{s} ~ ds' ~ \exp\Big(-\frac{\widetilde{\nu}'(\barq)(s,s')}{\varepsilon }   \Big)
\frac{1}{\varepsilon }K^{1-\chi}\left( h^{\varepsilon }\right) (s',y_2,\barq)
\\
+
\int_{0}^{t} ~ ds ~ 
\exp\Big(-\frac{\widetilde{\nu}(\barp)(t,s) }{\varepsilon } \Big) \frac{1}{\varepsilon }
\int_{\mathbb{R}^3} ~ d\barq ~
k^{\chi} (\barp,\barq) ~\int_{0}^{s} ~ ds' ~ \exp\Big(-\frac{\widetilde{\nu}'(\barq)(s,s')}{\varepsilon }   \Big)
 \\
 \times
\frac{\varepsilon ^{2}}{\sqrt{J}}\mathcal{Q}\left(h^{\varepsilon }\sqrt{J},h^{\varepsilon }\sqrt{J}
\right)(s',y_2,\barq)
\\
+
\int_{0}^{t} ~ ds ~ 
\exp\Big(-\frac{\widetilde{\nu}(\barp)(t,s) }{\varepsilon } \Big) \frac{1}{\varepsilon }
\int_{\mathbb{R}^3} ~ d\barq ~
k^{\chi} (\barp,\barq) ~\int_{0}^{s} ~ ds' ~ \exp\Big(-\frac{\widetilde{\nu}'(\barq)(s,s')}{\varepsilon }   \Big)
 \\
 \times
\sum_{i=1}^6 \varepsilon^{i-1}\frac{1}{\sqrt{J}}\left\{
\mathcal{Q}\left(F_i,h^{\varepsilon}\sqrt{J}\right)
+
\mathcal{Q}\left(h^{\varepsilon}\sqrt{J},F_i\right)
\right\}(s',y_2,\barq)
\\
+
\int_{0}^{t} ~ ds ~ 
\exp\Big(-\frac{\widetilde{\nu}(\barp)(t,s) }{\varepsilon } \Big) \frac{1}{\varepsilon }
\int_{\mathbb{R}^3} ~ d\barq ~
k^{\chi} (\barp,\barq) ~
\\
\times
\int_{0}^{s} ~ ds' ~ \exp\Big(-\frac{\widetilde{\nu}'(\barq)(s,s')}{\varepsilon }   \Big)\varepsilon ^{2}%
\tilde{A}(s',y_2,\barq).
\end{multline}
Above we are using the necessary additional notation
$$
\widetilde{\nu}'(\barq)(s,s')
\eqdef
\int_{s'}^s ~ d \tau ~ \nu (\mathcal{M} )(\tau,\barx  ).
$$
We have
$
\nu (J )(\barq) \approx
\nu(\mathcal{M})(\barq),
$
which implies
$$
\widetilde{\nu}'(\barq)(s,s')
\approx 
\nu (J)(\barq) (s-s'),
$$
which we will now simply write as
$
\nu  (s-s')
$
as we did in the previous case of $\widetilde{\nu}.$

Using similar arguments to
(\ref{h2}) and (\ref{h1}), we can
control all the terms in \eqref{duhamelD} and \eqref{duhamelDH}
 except the second term in \eqref{duhamelD} 
by the following upper bound:
\begin{equation*}
C\left\{\|h_0\|_{\infty ,\ell}+\varepsilon ^{3}\sup_{0\leq
s\leq T}\|h^{\varepsilon }(s)\|_{\infty ,\ell}^{2}+\varepsilon
\sup_{0\leq s\leq T}\|h^{\varepsilon }(s)\|_{\infty,\ell}^{{}}+C\varepsilon ^{3}\right\}.
\end{equation*}
We now concentrate on the second term in \eqref{duhamelD},
$I_{\gamma,\gamma}$.

We claim that for any small $\eta'>0,$ the following estimate holds:
$$
\left| w_\ell I_{\gamma,\gamma} \right|
\le
\eta'
\sup_{0\leq s\leq T}\|h^{\varepsilon
}(s)\|_{\infty ,\ell}
+
C_{\eta'}\sup_{0\leq s\leq T}\|f^{\varepsilon }(s)\|_{2}.
$$
This is proved in \cite[Lemma 4.4]{strainSOFT}.   We note that \cite[Lemma 4.4]{strainSOFT} is explained in detail for the soft potentials and for general momentum weights parametrized by $k\ge 0$. Here we only use the $k=0$ case from  \cite[Lemma 4.4]{strainSOFT}.
Furthermore, this estimate can be easily extended to the full range of hard and soft-potentials.  The proof for the hard potentials follows in exactly the same way as the estimate for the soft potentials, but several technical simplifications of the proof are possible in the hard potential case.

With this last estimate we will now finish the proof. We first collect all of our estimates as follows:
\begin{multline*}
\sup_{0\leq s\leq T}
\| \varepsilon ^{3/2} h^{\varepsilon}(s)\|_{\infty, \ell} 
\le
C \{\eta+\eta'\}
\sup_{0\leq s\leq T}\| \varepsilon^{3/2} h^{\varepsilon}(s)\|_{\infty, \ell}
+
\varepsilon^{7/2} C
\\
+
C_{\eta}
\|\varepsilon ^{3/2}h_{0}\|_{\infty, \ell}
+
\varepsilon^{3/2}~ C~\sup_{0\leq s\leq T}
\| \varepsilon^{3/2} h^{\varepsilon }(s)\|_{\infty, \ell}^{2}
+
C_{\eta'} 
\sup_{0\leq s\leq T}\|f^{\varepsilon }(s)\|_{2}.
\end{multline*}
We now choose $\eta$ and $\eta'$ small
in such a way that
 $
 C\{\eta+\eta'\} < \frac{1}{2}.
 $
 Then
 for sufficiently small $\varepsilon>0,$
 we obtain
\begin{equation*}
\sup_{0\leq s\leq T }
\| \varepsilon ^{3/2} h^{\varepsilon}(s)\|_{\infty,\ell}
\le 
C\left\{\|\varepsilon ^{3/2}h_{0}\|_{\infty,\ell}
+
\sup_{0\leq s\leq T }\|f^{\varepsilon }(s)\|_{2}+\varepsilon^{7/2}\right\},
\end{equation*}
and we conclude our proof.
\qed

\subsection*{Acknowledgments} 
RMS thanks Princeton University, where this project was initiated, for its generous hospitality.
JS echoes the sentiments of RMS, and also thanks Mihalis Dafermos and Willie Wong for useful discussions.  
RMS and JS both thank the University of Cambridge for support during the completion of this article.

\bibliography{hydro}

\begin{bibdiv}
\begin{biblist}

\bib{arsenioPHD}{article}{
      author={Ars{\'e}nio, Diogo},
       title={On the {B}oltzmann equation: Hydrodynamic limit with long-range
  interactions and mild solutions},
        date={September 2009},
     journal={Ph.D. dissertation, Department of Mathematics, New York
  University},
}

\bib{MR1213991}{article}{
      author={Bardos, Claude},
      author={Golse, Fran{\c{c}}ois},
      author={Levermore, C.~David},
       title={Fluid dynamic limits of kinetic equations. {II}. {C}onvergence
  proofs for the {B}oltzmann equation},
        date={1993},
        ISSN={0010-3640},
     journal={Comm. Pure Appl. Math.},
      volume={46},
      number={5},
       pages={667\ndash 753},
         url={http://dx.doi.org/10.1002/cpa.3160460503},
      review={\MR{MR1213991 (94g:82039)}},
}

\bib{MR1650310}{article}{
      author={Bardos, Claude},
      author={Golse, Fran{\c{c}}ois},
      author={Levermore, C.~David},
       title={Acoustic and {S}tokes limits for the {B}oltzmann equation},
        date={1998},
        ISSN={0764-4442},
     journal={C. R. Acad. Sci. Paris S\'er. I Math.},
      volume={327},
      number={3},
       pages={323\ndash 328},
         url={http://dx.doi.org/10.1016/S0764-4442(98)80154-7},
      review={\MR{MR1650310 (99k:76128)}},
}

\bib{MR1054287}{article}{
      author={Bardos, Claude},
      author={Golse, Fran{\c{c}}ois},
      author={Levermore, David},
       title={Sur les limites asymptotiques de la th\'eorie cin\'etique
  conduisant \`a\ la dynamique des fluides incompressibles},
        date={1989},
        ISSN={0764-4442},
     journal={C. R. Acad. Sci. Paris S\'er. I Math.},
      volume={309},
      number={11},
       pages={727\ndash 732},
      review={\MR{MR1054287 (91k:76139)}},
}

\bib{MR1115587}{article}{
      author={Bardos, Claude},
      author={Golse, Fran{\c{c}}ois},
      author={Levermore, David},
       title={Fluid dynamic limits of kinetic equations. {I}. {F}ormal
  derivations},
        date={1991},
        ISSN={0022-4715},
     journal={J. Statist. Phys.},
      volume={63},
      number={1-2},
       pages={323\ndash 344},
      review={\MR{MR1115587 (92d:82079)}},
}

\bib{MR1115292}{article}{
      author={Bardos, Claude},
      author={Ukai, Seiji},
       title={The classical incompressible {N}avier-{S}tokes limit of the
  {B}oltzmann equation},
        date={1991},
        ISSN={0218-2025},
     journal={Math. Models Methods Appl. Sci.},
      volume={1},
      number={2},
       pages={235\ndash 257},
         url={http://dx.doi.org/10.1142/S0218202591000137},
      review={\MR{MR1115292 (93a:76023)}},
}

\bib{MR1026740}{article}{
      author={Boisseau, B.},
      author={van Leeuwen, W.~A.},
       title={Relativistic {B}oltzmann theory in {$D+1$} spacetime dimensions},
        date={1989},
        ISSN={0003-4916},
     journal={Ann. Physics},
      volume={195},
      number={2},
       pages={376\ndash 419},
      review={\MR{MR1026740 (90j:82019)}},
}

\bib{MR586416}{article}{
      author={Caflisch, Russel~E.},
       title={The fluid dynamic limit of the nonlinear {B}oltzmann equation},
        date={1980},
        ISSN={0010-3640},
     journal={Comm. Pure Appl. Math.},
      volume={33},
      number={5},
       pages={651\ndash 666},
      review={\MR{MR586416 (81j:76072)}},
}

\bib{MR1898707}{book}{
      author={Cercignani, Carlo},
      author={Kremer, Gilberto~Medeiros},
       title={The relativistic {B}oltzmann equation: theory and applications},
      series={Progress in Mathematical Physics},
   publisher={Birkh\"auser Verlag},
     address={Basel},
        date={2002},
      volume={22},
        ISBN={3-7643-6693-1},
      review={\MR{MR1898707 (2003f:82078)}},
}

\bib{dC2007b}{article}{
      author={Christodoulou, Demetrios},
       title={The {E}uler equations of compressible fluid flow},
        date={2007},
        ISSN={0273-0979},
     journal={Bull. Amer. Math. Soc. (N.S.)},
      volume={44},
      number={4},
       pages={581\ndash 602 (electronic)},
      review={\MR{MR2338367}},
}

\bib{dC2007}{book}{
      author={Christodoulou, Demetrios},
       title={The formation of shocks in 3-dimensional fluids},
   publisher={European Mathematical Society},
     address={\textnormal{Z\"{u}rich}, Switzerland},
        date={2007},
}

\bib{sdGsvLwvW1980}{book}{
      author={de~Groot, S.~R.},
      author={van Leeuwen, W.~A.},
      author={van Weert, Ch.~G.},
       title={Relativistic kinetic theory},
   publisher={North-Holland Publishing Co.},
     address={Amsterdam},
        date={1980},
        ISBN={0-444-85453-3},
        note={Principles and applications},
      review={\MR{MR635279 (83a:82024)}},
}

\bib{MR1029125}{article}{
      author={De~Masi, A.},
      author={Esposito, R.},
      author={Lebowitz, J.~L.},
       title={Incompressible {N}avier-{S}tokes and {E}uler limits of the
  {B}oltzmann equation},
        date={1989},
        ISSN={0010-3640},
     journal={Comm. Pure Appl. Math.},
      volume={42},
      number={8},
       pages={1189\ndash 1214},
         url={http://dx.doi.org/10.1002/cpa.3160420810},
      review={\MR{MR1029125 (90m:35152)}},
}

\bib{MR1014927}{article}{
      author={DiPerna, R.~J.},
      author={Lions, P.-L.},
       title={On the {C}auchy problem for {B}oltzmann equations: global
  existence and weak stability},
        date={1989},
        ISSN={0003-486X},
     journal={Ann. of Math. (2)},
      volume={130},
      number={2},
       pages={321\ndash 366},
      review={\MR{MR1014927 (90k:82045)}},
}

\bib{MR1031410}{article}{
      author={Dudy{\'n}ski, Marek},
       title={On the linearized relativistic {B}oltzmann equation. {II}.
  {E}xistence of hydrodynamics},
        date={1989},
        ISSN={0022-4715},
     journal={J. Statist. Phys.},
      volume={57},
      number={1-2},
       pages={199\ndash 245},
      review={\MR{MR1031410 (91b:82043)}},
}

\bib{MR933458}{article}{
      author={Dudy{\'n}ski, Marek},
      author={Ekiel-Je{\.z}ewska, Maria~L.},
       title={On the linearized relativistic {B}oltzmann equation. {I}.
  {E}xistence of solutions},
        date={1988},
        ISSN={0010-3616},
     journal={Comm. Math. Phys.},
      volume={115},
      number={4},
       pages={607\ndash 629},
      review={\MR{MR933458 (89h:82017)}},
}

\bib{DEnotMSI}{article}{
      author={Dudy{\'n}ski, Marek},
      author={Ekiel-Je{\.z}ewska, Maria~L.},
       title={The relativistic {B}oltzmann equation - mathematical and physical
  aspects},
        date={2007},
     journal={J. Tech. Phys.},
      volume={48},
       pages={39\ndash 47},
}

\bib{MR1379589}{book}{
      author={Glassey, Robert~T.},
       title={The {C}auchy problem in kinetic theory},
   publisher={Society for Industrial and Applied Mathematics (SIAM)},
     address={Philadelphia, PA},
        date={1996},
        ISBN={0-89871-367-6},
      review={\MR{MR1379589 (97i:82070)}},
}

\bib{MR1105532}{article}{
      author={Glassey, Robert~T.},
      author={Strauss, Walter~A.},
       title={On the derivatives of the collision map of relativistic
  particles},
        date={1991},
        ISSN={0041-1450},
     journal={Transport Theory Statist. Phys.},
      volume={20},
      number={1},
       pages={55\ndash 68},
      review={\MR{MR1105532 (92f:81222)}},
}

\bib{MR1211782}{article}{
      author={Glassey, Robert~T.},
      author={Strauss, Walter~A.},
       title={Asymptotic stability of the relativistic {M}axwellian},
        date={1993},
        ISSN={0034-5318},
     journal={Publ. Res. Inst. Math. Sci.},
      volume={29},
      number={2},
       pages={301\ndash 347},
      review={\MR{MR1211782 (94c:82063)}},
}

\bib{MR2182829}{incollection}{
      author={Golse, Fran{\c{c}}ois},
       title={The {B}oltzmann equation and its hydrodynamic limits},
        date={2005},
   booktitle={Evolutionary equations. {V}ol. {II}},
      series={Handb. Differ. Equ.},
   publisher={Elsevier/North-Holland, Amsterdam},
       pages={159\ndash 301},
      review={\MR{MR2182829 (2007e:35025)}},
}

\bib{MR2025302}{article}{
      author={Golse, Fran{\c{c}}ois},
      author={Saint-Raymond, Laure},
       title={The {N}avier-{S}tokes limit of the {B}oltzmann equation for
  bounded collision kernels},
        date={2004},
        ISSN={0020-9910},
     journal={Invent. Math.},
      volume={155},
      number={1},
       pages={81\ndash 161},
         url={http://dx.doi.org/10.1007/s00222-003-0316-5},
      review={\MR{MR2025302 (2005f:76003)}},
}

\bib{MR2517786}{article}{
      author={Golse, Fran{\c{c}}ois},
      author={Saint-Raymond, Laure},
       title={The incompressible {N}avier-{S}tokes limit of the {B}oltzmann
  equation for hard cutoff potentials},
        date={2009},
        ISSN={0021-7824},
     journal={J. Math. Pures Appl. (9)},
      volume={91},
      number={5},
       pages={508\ndash 552},
         url={http://dx.doi.org/10.1016/j.matpur.2009.01.013},
      review={\MR{MR2517786}},
}

\bib{MR0135535}{incollection}{
      author={Grad, Harold},
       title={Principles of the kinetic theory of gases},
        date={1958},
   booktitle={Handbuch der {P}hysik (herausgegeben von {S}. {F}l\"ugge), {B}d.
  12, {T}hermodynamik der {G}ase},
   publisher={Springer-Verlag},
     address={Berlin},
       pages={205\ndash 294},
      review={\MR{MR0135535 (24 \#B1583)}},
}

\bib{MR0156656}{incollection}{
      author={Grad, Harold},
       title={Asymptotic theory of the {B}oltzmann equation. {II}},
        date={1963},
   booktitle={Rarefied {G}as {D}ynamics ({P}roc. 3rd {I}nternat. {S}ympos.,
  {P}alais de l'{UNESCO}, {P}aris, 1962), {V}ol. {I}},
   publisher={Academic Press},
     address={New York},
       pages={26\ndash 59},
      review={\MR{MR0156656 (27 \#6577)}},
}

\bib{2009arXiv0910.5512G}{article}{
      author={{Guo}, Y.},
      author={{Jang}, J.},
       title={{Global Hilbert Expansion for the Vlasov-Poisson-Boltzmann
  System}},
        date={2009-10},
     journal={ArXiv e-prints},
      eprint={0910.5512},
}

\bib{MR2013332}{article}{
      author={Guo, Yan},
       title={Classical solutions to the {B}oltzmann equation for molecules
  with an angular cutoff},
        date={2003},
        ISSN={0003-9527},
     journal={Arch. Ration. Mech. Anal.},
      volume={169},
      number={4},
       pages={305\ndash 353},
      review={\MR{MR2013332 (2004i:82054)}},
}

\bib{MR2172804}{article}{
      author={Guo, Yan},
       title={Boltzmann diffusive limit beyond the {N}avier-{S}tokes
  approximation},
        date={2006},
        ISSN={0010-3640},
     journal={Comm. Pure Appl. Math.},
      volume={59},
      number={5},
       pages={626\ndash 687},
         url={http://dx.doi.org/10.1002/cpa.20121},
      review={\MR{MR2172804 (2007b:35047)}},
}

\bib{MR2275331}{article}{
      author={Guo, Yan},
       title={Erratum: ``{B}oltzmann diffusive limit beyond the
  {N}avier-{S}tokes approximation'' [{C}omm. {P}ure {A}ppl. {M}ath. {\bf 59}
  (2006), no. 5, 626--687; mr2172804]},
        date={2007},
        ISSN={0010-3640},
     journal={Comm. Pure Appl. Math.},
      volume={60},
      number={2},
       pages={291\ndash 293},
         url={http://dx.doi.org/10.1002/cpa.20171},
      review={\MR{MR2275331 (2007g:35014)}},
}

\bib{G2}{article}{
      author={Guo, Yan},
       title={{D}ecay and continuity of {B}oltzmann equation in bounded
  domains},
        date={2010},
     journal={Arch. Ration. Mech. Anal.},
      volume={197},
      number={3},
       pages={173\ndash 809},
}

\bib{MR2472156}{article}{
      author={Guo, Yan},
      author={Jang, Juhi},
      author={Jiang, Ning},
       title={Local {H}ilbert expansion for the {B}oltzmann equation},
        date={2009},
        ISSN={1937-5093},
     journal={Kinet. Relat. Models},
      volume={2},
      number={1},
       pages={205\ndash 214},
      review={\MR{MR2472156}},
}

\bib{lH1997}{book}{
      author={H{\"o}rmander, Lars},
       title={Lectures on nonlinear hyperbolic differential equations},
      series={Math\'ematiques \& Applications (Berlin) [Mathematics \&
  Applications]},
   publisher={Springer-Verlag},
     address={Berlin},
        date={1997},
      volume={26},
        ISBN={3-540-62921-1},
      review={\MR{MR1466700 (98e:35103)}},
}

\bib{MR1842343}{article}{
      author={Lions, P.-L.},
      author={Masmoudi, N.},
       title={From the {B}oltzmann equations to the equations of incompressible
  fluid mechanics. {I}, {II}},
        date={2001},
        ISSN={0003-9527},
     journal={Arch. Ration. Mech. Anal.},
      volume={158},
      number={3},
       pages={173\ndash 193, 195\ndash 211},
         url={http://dx.doi.org/10.1007/s002050100143},
      review={\MR{MR1842343 (2002m:76085)}},
}

\bib{MR2043729}{article}{
      author={Liu, Tai-Ping},
      author={Yang, Tong},
      author={Yu, Shih-Hsien},
       title={Energy method for {B}oltzmann equation},
        date={2004},
        ISSN={0167-2789},
     journal={Phys. D},
      volume={188},
      number={3-4},
       pages={178\ndash 192},
      review={\MR{MR2043729 (2005a:82091)}},
}

\bib{aM1984}{book}{
      author={Majda, Andrew},
       title={Compressible fluid flow and systems of conservation laws in
  several space variables},
   publisher={Springer-Verlag},
     address={New York},
        date={1984},
}

\bib{MR1911233}{incollection}{
      author={Masmoudi, Nader},
       title={Some recent developments on the hydrodynamic limit of the
  {B}oltzmann equation},
        date={2002},
   booktitle={Mathematics \& mathematics education ({B}ethlehem, 2000)},
   publisher={World Sci. Publ., River Edge, NJ},
       pages={167\ndash 185},
      review={\MR{MR1911233 (2003j:35038)}},
}

\bib{MasmoudiLevermore}{article}{
      author={Masmoudi, Nader},
      author={Levermore, C.~David},
       title={From the {B}oltzmann equation to an incompressible
  {N}avier-{S}tokes-{F}ourier system},
        date={2009},
     journal={Arch. Rational Mech. Anal.},
      volume={in press},
}

\bib{MR0503305}{article}{
      author={Nishida, Takaaki},
       title={Fluid dynamical limit of the nonlinear {B}oltzmann equation to
  the level of the compressible {E}uler equation},
        date={1978},
        ISSN={0010-3616},
     journal={Comm. Math. Phys.},
      volume={61},
      number={2},
       pages={119\ndash 148},
      review={\MR{MR0503305 (58 \#20087)}},
}

\bib{fO1974}{book}{
      author={Olver, Frank W.~J.},
       title={Asymptotics and special functions},
      series={AKP Classics},
   publisher={A K Peters Ltd.},
     address={Wellesley, MA},
        date={1997},
        ISBN={1-56881-069-5},
        note={Reprint of the 1974 original [Academic Press, New York; MR0435697
  (55 \#8655)]},
      review={\MR{MR1429619 (97i:41001)}},
}

\bib{jSmS1998}{book}{
      author={Shatah, Jalal},
      author={Struwe, Michael},
       title={Geometric wave equations},
      series={Courant Lecture Notes in Mathematics},
   publisher={New York University Courant Institute of Mathematical Sciences},
     address={New York},
        date={1998},
      volume={2},
        ISBN={0-9658703-1-6; 0-8218-2749-9},
      review={\MR{MR1674843 (2000i:35135)}},
}

\bib{jS2008c}{thesis}{
      author={Speck, Jared},
       title={On the questions of local and global well-posedness for the
  hyperbolic pdes occurring in some relativistic theories of gravity and
  electromagnetism},
        type={{PhD} dissertation},
     address={Piscataway, NJ},
        date={2008},
}

\bib{jS2008b}{article}{
      author={Speck, Jared},
       title={The non-relativistic limit of the {E}uler-{N}ordstr\"om system
  with cosmological constant},
        date={2009},
        ISSN={0129-055X},
     journal={Rev. Math. Phys.},
      volume={21},
      number={7},
       pages={821\ndash 876},
         url={http://dx.doi.org/10.1142/S0129055X09003748},
      review={\MR{MR2553428}},
}

\bib{jS2008a}{article}{
      author={Speck, Jared},
       title={Well-posedness for the {E}uler-{N}ordstr\"om system with
  cosmological constant},
        date={2009},
        ISSN={0219-8916},
     journal={J. Hyperbolic Differ. Equ.},
      volume={6},
      number={2},
       pages={313\ndash 358},
         url={http://dx.doi.org/10.1142/S0219891609001885},
      review={\MR{MR2543324}},
}

\bib{nla.cat-vn2540748}{book}{
      author={Stewart, John~M.},
       title={Non-equilibrium relativistic kinetic theory},
        type={Book},
    language={English},
   publisher={Springer-Verlag, Berlin, New York,},
        date={1971},
        ISBN={3540056521},
}

\bib{strainSOFT}{article}{
      author={Strain, Robert},
       title={Asymptotic stability of the relativistic {B}oltzmann equation for
  the {S}oft {P}otentials},
        date={2010},
        ISSN={0010-3616},
     journal={Communications in Mathematical Physics},
       pages={1\ndash 69},
         url={http://dx.doi.org/10.1007/s00220-010-1129-1},
        note={10.1007/s00220-010-1129-1},
}

\bib{strainCOOR}{article}{
      author={Strain, Robert},
       title={{C}oordinates in the relativistic {B}oltzmann theory},
        date={2010},
     journal={preprint},
}

\bib{strainNEWT}{article}{
      author={Strain, Robert~M.},
       title={Global {N}ewtonian limit for the relativistic {B}oltzmann
  equation near vacuum},
        date={2010},
     journal={SIAM J. Math. Anal.},
      volume={42},
      number={4},
       pages={1568\ndash 1601},
      eprint={arXiv:1004.5407v1},
}

\bib{strainPHD}{article}{
      author={Strain, Robert~M.},
       title={An energy method in collisional kinetic theory},
        date={May 2005},
     journal={Ph.D. dissertation, Division of Applied Mathematics, Brown
  University},
}

\bib{MR2100057}{article}{
      author={Strain, Robert~M.},
      author={Guo, Yan},
       title={Stability of the relativistic {M}axwellian in a collisional
  plasma},
        date={2004},
        ISSN={0010-3616},
     journal={Comm. Math. Phys.},
      volume={251},
      number={2},
       pages={263\ndash 320},
         url={http://dx.doi.org/10.1007/s00220-004-1151-2},
      review={\MR{MR2100057 (2005m:82155)}},
}

\bib{MR2209761}{article}{
      author={Strain, Robert~M.},
      author={Guo, Yan},
       title={Almost exponential decay near {M}axwellian},
        date={2006},
        ISSN={0360-5302},
     journal={Comm. Partial Differential Equations},
      volume={31},
      number={1-3},
       pages={417\ndash 429},
      review={\MR{MR2209761 (2006m:82042)}},
}

\bib{MR2366140}{article}{
      author={Strain, Robert~M.},
      author={Guo, Yan},
       title={Exponential decay for soft potentials near {M}axwellian},
        date={2008},
        ISSN={0003-9527},
     journal={Arch. Ration. Mech. Anal.},
      volume={187},
      number={2},
       pages={287\ndash 339},
      review={\MR{MR2366140 (2008m:82008)}},
}

\bib{MR0088362}{book}{
      author={Synge, J.~L.},
       title={The relativistic gas},
   publisher={North-Holland Publishing Company, Amsterdam},
        date={1957},
      review={\MR{MR0088362 (19,508b)}},
}

\bib{MR1942465}{incollection}{
      author={Villani, C{\'e}dric},
       title={A review of mathematical topics in collisional kinetic theory},
        date={2002},
   booktitle={Handbook of mathematical fluid dynamics, {V}ol. {I}},
   publisher={North-Holland},
     address={Amsterdam},
       pages={71\ndash 305},
         url={http://dx.doi.org/10.1016/S1874-5792(02)80004-0},
      review={\MR{MR1942465 (2003k:82087)}},
}

\end{biblist}
\end{bibdiv}

\end{document}